\documentclass[10pt]{amsart}

\usepackage[utf8]{inputenc}
\usepackage{amsmath,amsthm,amsfonts,amssymb,mathrsfs}
\usepackage[foot]{amsaddr}
\usepackage{hyperref}
\usepackage[margin=1.45in]{geometry}
\usepackage{todonotes,color}
\usepackage{algorithmic,natbib}
\usepackage{mathrsfs}
\usepackage[section]{algorithm}
\usepackage{enumitem}
\usepackage{parskip}
\usepackage{comment}
\usepackage[T1]{fontenc}
\newlist{steps}{enumerate}{1}
\setlist[steps, 1]{label = Step \arabic*:}
\usepackage{tikz}
\usepackage{pgfplots}
\pgfplotsset{compat=1.17} 
\usepackage[toc,page]{appendix}

\allowdisplaybreaks
\def\R{\mathbb{R}}

\def\H{\mathbb{H}}
\def\SO{\mathrm{SO}}
\def\S{\mathbb{S}}

\def\P{\mathcal{P}}
\newcommand{\diff}{\text{d}}
\newcommand{\E}{\mathbb{E}}
\newcommand{\dist}{\mathrm{dist}}
\newcommand{\op}{\text{op}}

\newcommand{\Var}{\mathrm{Var}}

\newcommand{\vvec}{\mathsf{vec}}
\newcommand{\halfvec}{\mathsf{hvec}}

\newcommand{\Exp}{\mathrm{exp}}
\newcommand\numberthis{\addtocounter{equation}{1}\tag{\theequation}}

\def\dvol{\mathrm{dvol}_g}
\def\Hess{\mathrm{Hess}}
\def\Ric{\mathrm{Ric}}

\def\Tr{\mathrm{Tr}}

\def\det{\mathrm{det\;}}
\def\div{\mathrm{div}}

\renewcommand\subsubsection{\@startsection{subsubsection}{3}}

\theoremstyle{plain}
\newtheorem{theorem}{Theorem}[section]

\newtheorem{lemma}[theorem]{Lemma}

\theoremstyle{definition}
\newtheorem{definition}[theorem]{Definition}
\newtheorem{example}{Example}[section]
\newtheorem{remark}{Remark}
\numberwithin{equation}{section}
\theoremstyle{remark}
\newtheorem*{assumption}{Assumption}
\numberwithin{remark}{section}

\newcommand{\rf}[1]{{\color{black}  #1}} 
\makeatletter

\renewcommand\subsubsection{\@secnumfont}{\bfseries}%
\renewcommand\subsubsection{\@startsection{subsubsection}{3}
  \z@{.5\linespacing\@plus.7\linespacing}{-.5em}%
  {\itshape\bfseries}}
  
  \makeatother

\title{Sampling and estimation on manifolds using the Langevin diffusion}
\author{Karthik Bharath$^{*}$}
\author{Alexander Lewis$^{\dagger}$}
\author{Akash Sharma$^\ddagger$}
\author{Michael V. Tretyakov$^{*}$}
\address{$^{*}$ School of Mathematical Sciences, University of Nottingham,  UK}
\address{$^{\dagger}$ Institut f\"ur Mathematische Stochastik, Georg-August-Universit\"at G\"ottingen, Germany}
\address{$^{\ddagger}$ Department of Mathematical Sciences,
Chalmers University of Technology and the University of Gothenburg, Sweden}

\begin{document}
\date{}
\maketitle

\begin{abstract}
	Error bounds are derived for sampling and estimation using a discretization of an intrinsically defined Langevin diffusion with invariant measure $\diff\mu_\phi \propto e^{-\phi} \dvol $ on a compact Riemannian manifold.  Two estimators of linear functionals of $\mu_\phi $ based on the discretized Markov process are considered: a time-averaging estimator based on a single trajectory and an ensemble-averaging estimator based on multiple independent trajectories. Imposing no restrictions beyond a nominal level of smoothness on $\phi$, first-order error bounds, in discretization step size, on the bias and variance/mean-square error of both estimators are derived. The order of error matches the optimal rate in Euclidean and flat spaces, and leads to a first-order bound on distance between the invariant measure $\mu_\phi$ and a stationary measure of the discretized Markov process. This order is preserved even upon using retractions when exponential maps are unavailable in closed form, thus enhancing practicality of the proposed algorithms. Generality of the proof techniques, which exploit links between two partial differential equations and the semigroup of operators corresponding to the Langevin diffusion, renders them amenable for the study of a more general class of sampling algorithms related to the Langevin diffusion. 
Conditions for extending analysis to the case of non-compact manifolds are discussed. Numerical illustrations with distributions, log-concave and otherwise, on the manifolds of positive and negative curvature elucidate on the derived bounds and demonstrate practical utility of the sampling algorithm.
\end{abstract}


\section{Introduction}
\label{sec:intro}
The Langevin algorithm for sampling from a target probability measure $  \diff \mu_{\phi} \propto e^{-\phi(x)} \diff x$ on $\R^{q}, q \geq 1$, is based on constructing a discrete time Markov process $\{X_n^h, n \geq 1\}$ with time-step size $h=T/\rf{N}>0$, $t_n=hn$, obtained by discretizing the ergodic Langevin diffusion 
\begin{equation}
	\label{eq:Langevin}
	\diff X(t)=-\frac{1}{2}\nabla \phi(X(t))\diff t+\diff B(t), \quad 0 \leq t \leq T,
\end{equation}
where $B(t)$ is the $q$-dimensional standard Brownian motion. Popular amongst the numerical methods for constructing $X_n^h$ is the (explicit) Euler method with Gaussian noise to approximate Brownian motion increments, under which various types of approximation errors in using $X_n^h$ in place of $X(t_n)$ have been ascertained \citep[e.g.,][]{talayinv,RT96,dalayan17, MT07,milstein2004stochastic}. The natural regime to examine error bounds for such algorithms is when step size $h \to 0$.  

The Langevin diffusion \eqref{eq:Langevin} has been successfully used for sampling from posterior distributions in Bayesian inference \citep[e.g.,][]{RT96,DM17}, for canonical-ensemble calculations in molecular dynamics \citep{leimkuhler_mathews_15} and  for approximating minimizers in convex/non-convex optimisation \citep[e.g.,][]{JMLR:v17:teh16a}. Recently, there has been growing interest in finding effective ways to sample from distributions on manifolds, and in providing corresponding theoretical guarantees \citep[e.g.,][]{girolami2011,Giro22,vempala2022convergence, sra2022,LE2023}. Our work aims to provide theoretical guarantees for an intrinsically-defined Langevin algorithm for sampling from a target probability measure on Riemannian manifolds. 

In a statistical context, when sampling and estimation are of primary interest, quantitative 
error bounds, in terms of the time-step $h$ and integration time $T$,  on two interconnected quantities related to distributional aspects of the approximations are relevant: bias of an estimator based on $X^h_n$ of the linear functional $\mu_\phi(\varphi):=\int \varphi \diff \mu_\phi$ over a class $\mathcal H$ of test functions $\varphi$; and an integral probability metric with respect to $\mathcal H$ between the stationary measure of $X_n^h$ and $\mu_\phi$. Such bounds are referred to as \emph{weak error bounds}. In statistical applications (e.g., Bayesian computing), the commonly used estimator of $\mu_\phi(\varphi)$ is based on a single trajectory of $X^h_n$ until a large time \rf{$T=Nh$}, known as the time-averaging estimator.  Under suitable conditions on $\phi$, when the law of $X(t)$ converges to its invariant measure $\mu_\phi$ as $t \to \infty$, it is known that for the Euler method the optimal rate at which bias of the time-averaging estimator disappears is $O(h)$ \citep{talayinv,MT07,MST10,milstein2004stochastic}; the first-order (in $h$) bound on the bias then results in a commensurate first-order bound on the integral probability metric. 

Modern statistical and machine learning applications have generated a growing need for extending the above program to a $q$-dimensional smooth Riemannian manifold $M$. 
To better understand the challenges, some context on the necessary ingredients is helpful. Broadly, \rf{the type of application engenders} two perspectives of, and hence coordinate systems on, $M$: (i) as a submanifold of $\R^k, k \geq q$, with embedded geometry inherited from that of $\R^k$; (ii) as a $q$-dimensional differentiable manifold equipped with an intrinsic geometry.  
These two points of view lead to different formulations of the Langevin diffusion on $M$ with different invariant measures. The focus of this article is on the intrinsic perspective, and on investigating a geodesic-based sampling algorithm built on an intrinsic Langevin diffusion. 

A Riemannian metric $g$ prescribes a notion of curvature of the manifold $M$ at a point and determines a unique intrinsic Brownian motion on it \citep[e.g.,][]{David82,hsu2002stochastic,Wang2013}. The  curvature of $M$ and behavior of the Brownian motion are intimately related: positive curvature makes curves come closer and thus prevents a Brownian motion from wandering away to infinity (non-explosion) while ensuring recurrence and fast convergence to ergodicity \citep[e.g.,][]{ichihara1982curvature1, ichihara1982curvature2}; this is true even for its approximation based on a geodesic random walk \citep{mangoubi2018rapid}. Negative curvature, on the other hand,  disperses curves and encourages Brownian motion to exit compact sets and drift away to infinity (explosion) while discouraging recurrence \citep[e.g.,][]{kendall1984brownian}.  Ricci curvature at a point, which is, roughly, \rf{the} average of the sectional curvatures of two-dimensional subspaces of the tangent space at the point, captures the notion of curvature alluded to above. Compactness (unboundedness) and positive (non-positive) curvature of $M$ are closely related, although there are non-trivial exceptions depending on the Riemannian metric $g$ (for instance, manifolds with hyperbolic metric $g$ \citep{everitt2000}). Thus, for negatively curved (typically) non-compact spaces\rf{,} a uniform lower bound on the Ricci curvature is needed to ensure that the Brownian motion is well-behaved, while this is automatically true for positively curved (typically) compact manifolds. These issues, as one can now imagine,  are subtler for a Brownian motion with drift.  

It would thus be reasonable to expect the intrinsic curvature of $M$, and regularity and behavior at infinity of its drift, to influence error of approximations of an intrinsically defined Langevin diffusion and its invariant measure.  This is apparently the case in recent works \citep[e.g.,][]{vempala2022convergence, sra2022}, where curvature-dependent constant terms show up in upper bounds on distances between the stationary measure $\mu^h_\phi$ of a discretized Markov processes $X^h_n$ and $\mu_\phi$ on $M$ defined with respect to the Riemannian volume measure $\dvol$. Importantly, however, the curvature-terms \emph{do not} affect the claimed error order with respect to step size $h$.

A differentiable manifold $M$ is locally Euclidean,  covered by a family of charts (open subsets of $\R^q$) and smooth mappings between them. In a small region of $M$\rf{,} local analysis of the weak error within a chart, where $X$ solves the appropriate stochastic differential equation (SDE), proceeds as in $\R^q$ using the local coordinate representation of the metric $g$ in the chart, which accounts for the local curvature.  The challenge lies in adapting the analysis when one transitions to another chart, where the local representation of the metric $g$, and hence that of $X$, changes. Thus, on flat manifolds $M$ with zero curvature, error analysis proceeds unabated as in $(\R^q,g)$. Indeed, \cite{MST10} demonstrate precisely this with weak error analysis when $M=\mathbb T^q\cong \R^q/\mathbb Z^q$, the $q$-dimensional compact torus, equipped with the flat (quotient) metric from $\R^q$. The main technical ingredient in their analysis was the link between the solution of a Poisson equation (given by a partial differential equation (PDE)) on $\mathbb T^q$ and the infinitesimal generator of the Langevin diffusion $X$.  Aided by the fact that a weak \rf{error} analysis merely requires bounding and matching of a handful of moments, assured by compactness of $\mathbb T^q$, an \emph{optimal first-order} $O(h)$ weak error bound, matching the situation in $\R^q$, was established for general discretization schemes, including the Euler method.

It thus seems plausible to posit that weak error bounds of a suitable Euler method on compact manifolds $M$ should match the optimal first-order rate seen on $\mathbb T^q$. This leads us to the two related questions of interest in this work, stated informally as: 
\begin{enumerate}
	\item [(Q1)] Can an intrinsic Euler discretization method of an intrinsically defined Langevin diffusion on a compact manifold $M$ match the first-order weak error bounds obtained in the Euclidean setting?
	\item [(Q2)] Is it possible to construct estimators of linear functionals of the target measure on $M$ with first-order bias terms?
\end{enumerate}

\subsection{Contributions and related work}
\label{sec:relatedwork}
Our main contribution is in providing affirmative answers to (Q1) and (Q2), and in doing so demonstrate utility of the PDE-based tools in obtaining error bounds of Langevin-based sampling algorithms for compact manifolds. 

Specifically, on a connected compact Riemannian manifold $(M,g)$ without boundary\rf{,} we consider an intrinsically defined Langevin diffusion 
\[
\diff X(t) = -\dfrac{1}{2}\nabla\phi(X(t)) \diff t+\diff B^M(t) ,\quad X(0)=x \in M,
\]
with invariant measure $\diff \mu_{\phi} \propto e^{-\phi(x)}\dvol$, where $B^M$ is a Brownian motion on $M$, and its Euler discretization $X^h_n$ based on moving along geodesics, or their approximations, determined by drift-dependent  tangent direction and $\R^q$-valued noise. The resulting algorithm is termed \emph{Riemannian Langevin}. The potential $\phi$ is not required to be (geodesically) convex. At every step $X^h_n$ automatically belongs to the manifold $M$ owing to use of the exponential map (or more generally, retractions), i.e., the algorithm respects the geometry of the problem. It is well-known in deterministic \citep{HLW02} and stochastic \citep{leimkuhler_mathews_15,milstein2004stochastic} numerical analysis that geometric integrators (i.e., numerical methods which naturally preserve geometric features of differential equations they approximate) are computationally more efficient (e.g., allowing larger time steps while still producing accurate results) for long time simulations as ones arising in sampling problems. The Riemannian Langevin algorithm considered here belongs to the class of such geometric integrators, and is hence suitable for sampling on manifolds.

Two estimators of the linear functional $\mu_\phi(\varphi)$ for a smooth $\varphi$ are proposed: an \emph{ensemble-averaging} estimator defined using \rf{multiple} independent trajectories of $X^h_n$, and a \emph{time-averaging} estimator based on a single trajectory. The former is of great practical merit. First,  it enables decoupling and management of three sources of error: (i) length of the integration time $T$ that controls proximity of $X$ to its ergodic limit; (ii) discretization error controlled by the step size $h$; (iii) the number of independent trajectories of $X^h_n$ that controls the Monte Carlo error (sample variance). Second, the estimator is particularly useful in a parallel computing environment (GPUs or array of CPUs) where independent  trajectories of $X^h_n$ can be generated in parallel and combined in the end. Analysis of the time-averaging estimator leads to first-order bound on a distance between the empirical measure $\mu^h_\phi$ of $X^h_n$ and $\mu_\phi$,  with respect to a smooth class $\mathcal H$ of test functions. 

Our technical analyses are based on exploiting links between the Markov semigroup of $X$ and the Kolmogorov and Poisson PDEs; a positive answer to (Q1), bearing in mind the optimal error rate in $\mathbb T^q$,  is realised by effectively `flattening' the compact $M$, at least as far as the deviation of the stationary measure $\mu^h_{\phi}$ of $X^{h}_{n}$ from $\mu_\phi$ is concerned. The price paid for the proof technique used is seen in the absence of \emph{explicit} curvature-related constants in the weak error bounds. However, in addition to deriving the optimal order for weak error, there are other important benefits to the semigroup approach which make it suitable for analyzing sampling algorithms: (i) analysis can be extended to sampling on non-compact manifolds (Section \ref{sec:non-comp}); (ii) variants of the Riemannian Langevin algorithm with many families of distributions for increments of the Markov chain $X^h_n$, and based on splitting methods, can be analysed in a unified manner (Section~\ref{sec:other}). 

In existing works on SDE-based algorithms for sampling on manifolds \citep{vempala2022convergence,sra2022} weak convergence of the algorithm is established by first proving mean-square (strong) convergence, since weak convergence follows from mean-square; a consequence of this approach is \rf{usually} a non-optimal weak convergence order when compared to the situation in $\R^q$. A notable exception to this is \cite{TAL96}, which we discuss in Section~\ref{sec:RLMC}. 
\rf{Mean-square convergence of numerical methods for intrinsically defined SDEs on a Riemannian manifold is studied in \cite{man_msq}.}

Methods convergent in mean-square are needed when one is interested in proximity of trajectories of an SDEs' approximation to trajectories of the exact solution \citep{milstein2004stochastic}. However, the task of approximating expectations with respect to 
SDEs, including that of ergodic limits, is in principle simpler: it was noted in \cite{GN78a} that, roughly, weak-sense convergence needs only moments of the SDEs' solution and its approximation to be matched up to some order. Consequently, demands on approximations convergent in weak sense are less stringent than for mean-square schemes, especially in the case of higher-order approximations. In particular on $\R^q$, the Euler scheme for general SDEs has mean-square order $1/2$, while its weak order is $1$, and there are constructive weak order schemes of higher orders and no constructive mean-square schemes of higher order than $1/2$ in the general case \citep{milstein2004stochastic}.

Thanks to the PDE-based technique for proving weak convergence  \citep{GN85,TAL83a,talayinv,MST10} (see also \cite{milstein2004stochastic}), our proofs of weak convergence for Riemannian Langevin algorithms have a unified nature and are easily adaptable to a variety of schemes, for example, involving retraction maps (see Section~\ref{sec:retractions}). 
The consequence of this proof technique is that we do not monitor dependence of the coefficients (at $h$) in upper error bounds on parameters of the problem (such as dimension, curvature, spectral gap, etc). However, using the Talay-\rf{Tubaro} expansion \citep{TAT90} (see Section~\ref{sec:TT})\rf{,} one can estimate the leading error term of a weak scheme (i.e., the posteriori error), which is valuable in practice.

The central challenge in extending the analysis to non-compact $M$ lies in the fact that, unlike when $M=\R^{q}$ or a Ricci-flat manifold $M$ (e.g., $\mathbb T^{q}$), it is insufficient to merely control behavior of the drift term $\nabla \phi$, since the non-positive curvature of $M$ also dictates behavior of an intrinsic Brownian motion $B^M$. The Langevin diffusion can hence fail to be ergodic even when its drift term is well-behaved on a geodesically complete negatively curved $M$.  We discuss conditions that would enable obtaining first-order weak error bounds in Section~\ref{sec:non-comp}, and also provide numerical illustrations of the same in Section~\ref{sec:numpd}, but postpone the theoretical analysis to future work. 

As mentioned earlier, there are some existing works that provide curvature-dependent bounds for distances between $\mu^h_\phi$ and $\mu_\phi$ under different intrinsic discretizations $X^h_n$ and settings: \cite{TAL96} for chart-dependent Euler scheme (see further discussion at the end of Section~\ref{sec:RLMC}); \cite{sra2022} for the geodesic Euler scheme considered here claim to develop $O(h^{1/2})$  bound in expected \rf{squared} Riemannian distance between $X(\rf{t_n})$ and $X_{n}^{h}$, $\rf{t_n} = nh$; 
\cite{vempala2022convergence} provide an $O(h)$ bound on the Kullback-Leibler distance for Hessian manifolds, but assume possibility of \emph{exact} sampling of a Brownian increment on $M$, which currently is available for spheres \citep{mijatovic2020note}. 
Further, the above works do not consider and analyse an estimator of $\mu_\phi(\varphi)$ while we obtain an $O(h)$ bound on the bias of the estimators.

Although not directly related to our work, we note that there are quite a few papers on extrinsic discretization methods for an embedded Langevin diffusion $X$ defined by viewing $M$ as a submanifold of $\R^k, k \geq q$ \citep[e.g.,][]{lelievre2012, sharma2021, vilmart2022, armstrong2022}. A critical limitation of extrinsic methods is that while $X$ is constrained to move on $M$, the discretized Markov process $X^h_n$ need not, and repeated, computationally expensive, projections on to $M$ are required to ensure that $X^h_n$ does not leave $M$; an exception is when $M$ is a Lie group, and the group action enables transport along the manifold \cite[e.g.,][]{MW08,Ruslan09,Johan10,Ruslan15}. Use of projections results in an additional source of error over the original discretization of $X$, avoided entirely by intrinsic methods.  Intrinsic methods may require additional computation of geodesics when they are unknown in closed-form, but this can be addressed by using computationally cheaper, higher-order retractions which do not affect the theoretical weak error bounds (see Section~\ref{sec:retractions}). 
\rf{We note that as with the approximation of the projection map in the extrinsic setting, an efficient computation of retraction in our intrinsic setting does not depend on the potential $\phi$ and noise, and depends only on the manifold.}

\rf{
Projection-based schemes are inappropriate when the target probability density to sample from on a manifold $M$ is defined with respect to a volume form coming from a specific Riemannian metric $g$ that is not the restriction of the standard metric in $\mathbb R^k, k\geq q$. Moreover, when $M$ is an \emph{open} submanifold of $\mathbb R^k$, in the absence of a level-set representation of $M$, projection-based schemes are inappropriate. A practically relevant example that highlights the above two issues concerns the manifold $\mathcal P_m$ of $m \times m$ real symmetric positive definite matrices of dimension $q=m(m+1)/2$ (see Section \ref{sec:numpd}). Firstly, it is an open cone within the same dimensional vector space of symmetric matrices so that $k=m(m+1)/2$; an orthogonal projection would map a symmetric matrix with a single negative eigenvalue to a positive definite matrix with smallest eigenvalue very close to zero, close to the boundary of $\mathcal P_m$, which is infinitely far away with respect to the intrinsic geometry of $\mathcal P_m$, and the projection-based sampling scheme will fail to sample from the target measure with support in $\mathcal P_m$. Secondly, there are numerous choices of intrinsic Riemannian metrics motivated by practical applications \citep[e.g.][Chapter 3]{pennec2006riemannian} which lead to specific target probability densities that cannot be sampled with projection-based schemes  that views $M$ as a submanifold of $\mathbb R^k$ with the induced metric. 

Additionally, projection-based schemes are typically less stable (requiring the use of appreciably smaller time steps) compared to geometric integrators, which by construction automatically lie on the manifold; see, for example, \citet{Johan10} when $M=\S^2$. For $M$ that is a Lie group or a homogeneous space of one, Lie group integrators preserving the geometry represent an efficient choice. However, for general manifolds $M$ with specific metrics $g$ the intrinsic methods considered in this paper will typically have an advantage when compared with projection-based schemes. 
}


\section{Preliminaries}

In Section~\ref{sec_diff_geom_con}\rf{,} we briefly review concepts and fix notation from differential geometry needed for analysis on a Riemannian manifold $M$; for  more details we refer to some well-known sources \citep[e.g.,][]{jost2008riemannian,cheeger1975comparison,gallot1990riemannian, bishop2011geometry}. 
In Section~\ref{sec:sde}\rf{,} we introduce the Langevin diffusion on $M$ with a given invariant measure,  the target measure for sampling.  

\subsection{Differential geometric concepts} \label{sec_diff_geom_con}

Let $M$ be a $q$-dimensional topological manifold. By this we mean a Hausdorff, paracompact, locally Euclidean topological space. We further assume that $M$ is connected and without a boundary. A chart of $M$ is a double $(U,\psi)$ where $U$ is an open set of $M$ and $\psi:M \rightarrow \R^q$ is a homeomorphism. A collection of these charts that cover $M$ entirely is called an atlas. For any two different charts $(U_i,\psi_i)$ and $(U_j,\psi_j)$ that overlap (for any $i,j$ such that $U_i\cap U_j \neq \emptyset$), the map $\psi_j \circ \psi_i^{-1}:\psi_i(U_i\cap U_j) \rightarrow \psi_j(U_i\cap U_j)$ is known as a transition map. Then $M$ is a $C^k$-differentiable manifold if the transition maps are $C^k$ with $C^k$ inverses. In this paper, we consider the case when $k=\infty$, then $M$ is known as a smooth manifold.

Tangent vectors at a point $x \in M$ may be defined using derivations (linear operators on smooth functions on $M$ compatible with the product rule of differentiation) or as (time) derivatives of an equivalence class of curves passing through $x$. Denote by $T_xM$ the $q$-dimensional tangent space of $x \in M$ and by $TM=\cup_x T_xM$ the tangent bundle of $M$, itself a differentiable manifold of dimension $2q$. A local coordinate system $x=(x^1,\ldots,x^q)^\top$ in a chart induces a basis $\partial_x=\{\partial _{x^1},\ldots,\partial _{x^q}\}$ at $T_xM$ such that any $v \in T_xM$ may be expressed as $v=v^i\partial _{x^i}$, where the Einstein summation notation, which will be used throughout, is used. 
The map $X:M \to TM$ is a vector field with $X(x) \in T_xM$ for $x \in M$. 

Tangent spaces $T_xM$ and \rf{the} tangent bundle $TM$ have dual spaces, the cotangent space $T_x^*M$ and cotangent bundle $T^*M$ consisting of linear functionals $\rf{\nu}:T_xM \to \mathbb R$. The coordinate basis $\{\partial _{x^1},\ldots,\partial _{x^d}\}$ induces a similar basis $\{\diff x^1,\ldots,\diff x^q\}$ on $T^*_xM$ defined via the evaluation $\diff x^j(\partial_{x^i})=\delta^j_i$, where $\delta^j_i$ is the delta function, such that $\nu=\nu_i\diff x^i$ for every $\nu \in T^*_xM$.

A differentiable manifold $M$ becomes a Riemannian  manifold $(M,g)$ when equipped with a Riemannian metric $g:T_xM \times T_x M \to \mathbb R$, a family $\{g(x), x \in M\}$ of \rf{smoothly varying} symmetric positive definite bilinear \rf{forms} (inner products). In a chart around $x$, the metric $g(x)$ can be expressed as a symmetric positive definite matrix whose $(i,j)$\rf{-}th entry is $g(\partial_{x^i},\partial_{x^j})$, \rf{which we write in index notation as $g_{ij}$}.


Denote by $D$ the unique torsion-free Levi-Civita connection, or covariant derivative, compatible with $g$. The connection enables definition of a derivative $D_XY$ of a vector field $\rf{Y}$ along a vector field $\rf{X}$. For a smooth function $f:M \to \R$, denote by $X(f)$ the derivative of $f$ in the direction of $X$ based on the interpretation of tangent vectors as derivations. The covariant derivative is linear in $X$ and follows the product rule in $Y$ so that $D_X(fY)=X(f)Y+fD_XY$. More conveniently, $D_XY$ may be written in local coordinates at $x$ using the relation $D_{\partial_{x^i}}\partial_{x^j}=\Gamma^k_{ij}\partial_{x^k}$, where $\Gamma^k_{ij}$ are the Christoffel symbols. 

The length of a curve $c:[0,1]\rightarrow M$ connecting $c(0)=x$ and $c(1)=y$ is given by
\[
L(c) = \int_0^1\sqrt{g( \dot{c}(t),  \dot{c}(t))}\diff t,
\]
where $\dot{c} \in T_{c(t)}M$ is the tangent vector to the curve at $c(t)$. 
A curve that minimises the length functional is called a \emph{geodesic}, and this may not be unique. A geodesic $\gamma$ satisfies the equation $D_{\dot \gamma}\dot \gamma=0$, and in local coordinates may be expressed as
\begin{equation}\label{eq:geodesic eq}
	\dfrac{\diff ^2\gamma^k}{\diff t^2} + \Gamma_{ij}^k \dfrac{\diff \gamma^i}{\diff t}\dfrac{\diff \gamma^j}{\diff t}=0.
\end{equation}
Geodesics enjoy a rescaling property: if $t \mapsto \gamma(t)$ is a geodesic with $\gamma(0)=x$ and $\dot \gamma(0)=v$,  then $t \mapsto \gamma(\alpha t)$ is a geodesic with $\dot \gamma(0)=\alpha v$ for every $(x,v) \in TM$ and $\alpha, t \in \mathbb R$. Aided by the rescaling property of geodesic, the Riemannian distance $(x,y)\mapsto \rho(x,y)$ between two points $x,y \in M$ is defined to be the length of the geodesic $\gamma$ with $\gamma(0)=x$ and $\gamma(1)=y$. 

A manifold $M$ is said to be geodesically complete  if every geodesic $\gamma:(-a,a) \to M$ can be extended to a geodesic with domain $\R$, and by the Hopf-Rinow theorem, the space $(M,\rho)$ is complete as a metric space, and the two notions of completeness coincide; this ensures that for such manifolds there exists at least one minimizing geodesic between any two points whose length is the distance between the points. Moreover, $(M,\rho)$ is said to be compact if it is a compact metric space, and every compact Riemannian manifold without boundary is complete. 

Theory of differential equations guarantees that for every $(x,v) \in TM$, there exists a unique geodesic $\gamma$ starting at $x$ with $\gamma(0)=x$ with $\dot \gamma(0)=v$. This results in a surjective map on $T_xM$ for complete $M$ that maps a vector $v \in T_xM$ to a point on $M$ along geodesics: straight lines through the origin in $T_xM$ map to geodesics through $x$ on $M$. 
\begin{definition}\label{def:expmap}
	Let $\gamma:[0,1]\rightarrow M$ be a geodesic with $\gamma(0) = x$ and $\dot \gamma(0)=v$. Then
 \[
\exp_x: T_xM \to M, \quad v \mapsto \exp_x(v):=\gamma(1)
 \]
 is known as the exponential map at $x$. 
\end{definition}     
Although surjectivity ensures that $\exp_x$ is defined on all of $T_xM$ for every $x \in M$, it is injective only in a neighbourhood $N_x$ around the origin in $T_xM$. Every geodesic $\gamma(tv)=\exp_x(tv)$ starting from $t=0$ is either minimizing on all $t \in \R$ or is minimizing until $s<\infty$; then $\gamma(sv)$ is called a \emph{cut point} of $x$. The set of all cut points of all geodesics starting from $x$ is the cut locus $\mathrm{cut}_x$ of $x$. 

The neighbourhood $N_x$ is star-shaped and connected, and $\exp_x(N_x)=M-\mathrm{cut}_x$.  Relatedly, the radius of injectivity, $\mathrm{inj}(x)$, of a point $x \in M$ is the maximal radius of balls around the origin in $T_xM$ on which the exponential map is injective; or, in other words, it is the radius $r$ of the largest geodesic ball centered at $x$ such that each geodesic connecting to $x$ within the ball is minimizing. The injectivity radius $\mathrm{inj}_M$ of the manifold $M$ is the smallest of such radii across $M$.


It is necessary to introduce a handful of tensors in preparation for the analysis of the algorithm. First, the exterior derivative $\diff$ is a differential operator that maps $C^1$ functions to the cotangent \rf{bundle} $T^*M$, and in local coordinates assumes the form $\diff f(x) = \frac{\partial f(x)}{\partial x^i}\diff x^i$. The quantity $\diff f$ is known as a 1-form. More generally, an $r$-form,  for $r\in \{1,\ldots,q\}$,  can be considered on $\wedge^r T_x^*M$ with anti-symmetric wedge product $\wedge$ spanned by $\diff x^1\wedge \cdots \wedge \diff x^r$. For brevity, the wedge symbol is omitted when discussing differential forms. Then, the exterior derivative $\diff $ maps $r$-forms to $(r+1)$-forms and $\diff ^2=0$. For the purposes of integration, we introduce the volume form, a $q$-form, $\dvol = \sqrt{|g|}\diff x^1 \ldots \diff x^q$, where $|g|$ denotes the determinant of the Riemannian metric $g$ in local coordinates around $x$.  

For $v \in T_xM$,  the relation
$$ \diff f(x)(v) = g(\nabla f(x),v) $$
defines the Riemannian gradient operator $\nabla$ acting on smooth functions $f:M \to \R$ as the dual of $\diff f$. In local coordinates,  the gradient is written as $\nabla f = g^{ij}\frac{\partial f}{\partial x^i}\frac{\partial}{\partial x^j}$ \rf{where $g^{ij}$ are the elements of the inverse metric}; it is often convenient to omit dependence on $x$ and use $\diff f(v), v \in T_xM$.  

The $L^2$-adjoint of the gradient is the negative of the divergence operator. For a tangent vector $v \in T_xM$ the local coordinate representation of divergence is $\mathrm{div}(v) = \frac{1}{\sqrt{|g|}}\frac{\partial}{\partial x^i}(\sqrt{|g|}v^i)$.

The Hessian of a $C^2$ function $f:M \to \R$ is 
$$\Hess^f = D^2f = D\diff f.$$
Then, the Laplace-Beltrami operator, as an analogue of the Laplacian in $\R^q$,  is defined as the trace (with respect to the metric tensor) of the Hessian
\[
\rf{\Delta_M} f = \mathrm{Tr}\Hess^f = g^{ij}\Hess_{ij}^f.
\]

For higher derivatives of functions, $D^k: C^\infty(M) \rightarrow (T^*M)^{\otimes k}$, i.e. $D^k$ maps smooth functions to the $k$ times cotangent space. We can contract $D^kf$ by interpreting it as the mapping $D^k f:(TM)^{\otimes k} \rightarrow \R$. For the contracted tensor $S \in (T^{*}M)^{\otimes k}$, $S(v_1,\ldots,v_k)$ for arbitrary $v \in TM$, the covariant derivative acts multilinearly on the $(0,k)$-tensor $S$ and also on each of the vector fields. More precisely, $D_{v} (S(v_1,\ldots,v_k)) = (D_v S)(v_1,\ldots,v_k) + S(D_v v_1,v_2,\ldots,v_k) + \cdots + S(v_1,\ldots,v_{k-1},D_v v_k)$ \citep[e.g.,][pp. 202]{jost2008riemannian}. The norm $g$ can be extended to tensors, for a tensor $S \in T^{p,q}_xM$ the norm of $S$ is 
\[
|S|^2 = g_{i_1k_1} \cdots g_{i_pk_p}g^{j_1l_1} \cdots g^{j_ql_q} S^{i_1 \cdots i_p}_{j_1 \cdots j_q} S^{k_1 \cdots k_p}_{l_1 \cdots l_p}.
\]
The parallel transport $\Pi_{\gamma_{x,y}}:T_xM \to T_yM$ along a curve $\gamma_{x,y}:[0,1]\to M$ connecting $x$ and $y$ is a linear isometry determined by a unique vector field $X$ satisfying $D_{\dot \gamma(t)}X=0$ for all $t$, such that $g(\Pi_{\gamma_{x,y}}u,\Pi_{\gamma_{x,y}}v)=g(u,v)$ for all $u,v \in T_xM$. More generally, the parallel transport $\Pi_{\gamma_{x,y}}$ of a tensor $S \in T_x^{p,q}M$ along an arbitrary geodesic (or even curve) $\gamma$ from the point $x=\gamma(0)$ to $y = \gamma(1)$ is such that it preserves the inner product between $S$ and $\dot \gamma$, $g(S,\dot \gamma(0)) = g(\Pi_{\gamma_{x,y}}S,\dot\gamma(1))$, and also its norm, $ g(S,S) = |S|^2 = |\Pi_{\gamma_{x,y}}S|^2$. It is similarly defined by solving the ODE $D_{\dot\gamma}\Pi S=0$.

Define the operator norm for a $(0,k)$-tensor $\mathfrak T$ on $M$ \rf{\citep[e.g.][]{le2024diffusion}}
\begin{align}
	|\; \mathfrak T\;|_{\op} := \sup_{V_{1},\dots,V_{k} \in T_x M,\; |V_i| \neq 0} \frac{|\mathfrak T(V_1,\dots,V_k)|}{\prod_{i=1}^{k}|V_i| }
\end{align}
for each $x \in M$. In addition to spaces of continuously differentiable functions $C^l(M)$, we will need H{\"o}lder spaces  $C^{l,\epsilon}(M)$,  $  0 <\epsilon < 1$, containing real-valued functions $u(x)$, $x \in M$,  with corresponding number of H{\"o}lder continuous derivatives and with the norm
\begin{align*}
	\Vert u\Vert_{C^{l,\epsilon } (M)}	:=	\sum\limits_{k=0}^{l}\sup_{x \in M} |D^{k} u(x)|_{\op}
	+\sup_{\gamma_{x,y}, x \neq y \in M}\frac{|D^{l} u(x)-\Pi_{\gamma_{x,y}} (D^{l} u(y))|_{\op}}{(\rho(x,y))^{\epsilon}},  
\end{align*}
where $\gamma_{x,y}$ denotes any possible minimal geodesic from
$y$ to $x$ and $\Pi_{\gamma_{x,y}}$ denotes the parallel transport from $T_y M$  to $T_x M$ along $\gamma_{x,y}$.
Analogously, $C^{l/2,l,\epsilon } ([0,T]\times M )$ is a space of sufficiently smooth functions $u(t,x)$ with the norm for even $l$:
\begin{align*}
	&\Vert u \Vert_{C^{l/2,l,\epsilon } ([0,T]\times M)}	
	:=	\sum\limits_{k=0}^{l/2} \sum_{n=0}^{l-2k}
	\sup_{t  \in [0,T]}\sup_{x \in M}\bigg|D^n\frac{\partial^k}{\partial t^k}u(t,x)\bigg|_{\op} \\
	&+ \sum\limits_{k=0}^{l/2} \sup_{t  \in [0,T]} \sup_{\gamma_{x,y}, x \neq y \in M}\frac{|D^{l-2k}\frac{\partial^k}{\partial t^k}u(t,x)-\Pi_{\gamma_{x,y}} (D^{l-2k}\frac{\partial^k}{\partial t^k}u(t,y))|_{\op}}{(\rho(x,y))^{\epsilon}}.
\end{align*}

For vector fields $X,Y,Z$, the Riemannian curvature tensor \rf{at a point $x\in M$} \newline$\rf{R:T_xM \times T_xM \times T_xM\rightarrow T_xM}$ of $M$ is defined as
$$  R(X,Y)Z := (D_XD_Y - D_YD_X - D_{[X,Y]})Z, $$
where the commutativity bracket \rf{(or Lie bracket)} \rf{is} defined  via its application to a smooth function $f:M \to \R$ \rf{by} $[X,Y]f=X(Y(f))-Y(X(f))$. Roughly, it quantifies the difference between the vectors obtained by transporting a vector $Z(x)$ first along $X(x)$ then $Y(x)$, in contrast to moving first along $Y(x)$ and then $X(x)$. The sectional curvature at $x$  is the quadratic form \cite[pp. 89]{hsu2002stochastic}
$$ K(X,Y) := g(R(X,Y)Y,X), \quad X,Y \in T_xM,$$and 
for an orthonormal basis $\{X_i(x)\}$ on $T_xM$, the Ricci curvature at $x$ is a contraction of the sectional curvature:
$$ \Ric(X,\rf{X}): = \sum_{i=1}^q  g(R(X,X_i)X_i,X). $$

\subsection{Langevin diffusion on a Riemannian manifold}\label{sec:sde}

Let $(M,g)$ be a connected compact Riemannian manifold of dimension $q\geq1$ without boundary; this will be the main setting assumed throughout the paper except in Sections~\ref{sec:non-comp} and~\ref{sec:numpd}.
Consider the target probability measure $\mu_\phi$ on $M$ 
\begin{equation}\label{invmeasure}
	\diff \mu_\phi = p \, \dvol\thinspace = \thinspace\dfrac{1}{C_\phi}e^{-\phi}\dvol,
\end{equation}
absolutely continuous with density $p$ with respect to the volume measure $\dvol$, where $C_\phi = \int_M e^{-\phi} \dvol<\infty$. The following conditions on smoothness of the metric $g$ and the density $p$ are imposed through that of $\phi$.
\begin{assumption}
  \label{assump:g}
		Assume that $x\mapsto g^{ij}(x) \in C^{\rf{3},\epsilon}(M)$.
\end{assumption}
\begin{assumption}
  \label{assump:phi}
		Assume that $\phi \in C^{3,\epsilon}(M).$ 
\end{assumption}
No further restrictions are placed on $\phi$ (e.g., convexity). The level of smoothness imposed in these \rf{a}s\rf{s}umptions is required primarily for proofs of Theorems~\ref{theorem4.1}-\ref{theorem4.3}, but may be weakened for the content of this section. 

Since we aim to sample from a given probability measure $\mu_\phi$, we must  find an (second order) elliptic linear operator $ \mathcal{A^*}$ acting on probability densities $p$ (with respect to $\dvol$) so that the density $p$ from (\ref{invmeasure}) satisfies the stationary Fokker-Planck equation \citep[Chapter~5, Proposition~4.5]{ikeda2014stochastic}:
\begin{equation}\label{eq:FP}
	\mathcal{A^*} p = 0.
\end{equation}
For example, one can verify that the above is satisfied by the operator
\begin{equation}\label{eq:operA*}
	\mathcal{A^*} p = \dfrac{1}{2}\Delta_M p + \frac{1}{2}\div( p \nabla\phi ),
\end{equation}
which is adjoint to the elliptic operator \cite[Chapter~5, Proposition 4.4]{ikeda2014stochastic}
\begin{equation}\label{eq:operA}
	\mathcal{A} = \dfrac{1}{2}\Delta_M - \frac{1}{2}g(\nabla\phi,\nabla)
\end{equation} 
acting on $C^2(M)$-functions, where $\Delta_M$ is the Laplace-Beltrami operator. 

The Markov process with  infinitesimal generator $\mathcal{A}$ is governed by the following SDE \citep[e.g.,][]{David82,ikeda2014stochastic}
\begin{equation}\label{eq:diffusion}
	\diff X(t) = -\dfrac{1}{2}\nabla\phi(X(t)) \diff t+\diff B^M(t), \quad 0 \leq t \leq T;  \thickspace X(0)=x,
\end{equation}
where $B^M$ is a Brownian motion on the manifold $M$ defined using the Eells-Elworthy-Malliavin frame bundle construction \citep[e.g.,][p.87]{hsu2002stochastic}. 

The Brownian motion $B^M$ can be intuited as follows: given a filtered probability space $(\Omega,\mathcal F, \mathcal F_t,\mathbb P), t \geq 0$,  and an $\{\mathcal F_t\}_{t>0}$-adapted $\R^q$-valued standard Brownian motion $B$, the construction uses the dynamics of $B$ on $\R^q$ to induce one on $M$ by first, mapping the dynamics of $B$ onto the orthonormal frame bundle using the notion of horizontal vector fields, and then projecting down from the frame bundle onto the manifold. Thus, in local coordinates in a chart the SDE  \eqref{eq:diffusion} assumes the It\^o form
\begin{align}\label{eq:diffusion local}
	&\diff X^{i}(t) = -\frac{1}{2}g^{ij}(X(t))\frac{\partial \phi}{\partial x^{j}}(X(t)) \diff t -\frac{1}{2}g^{kj}(X(t))\Gamma_{kj}^{i}(X(t))\diff t + (g^{1/2}(X(t))\rf{)}^{ij}\diff B_{j}(t), \\
	& \quad \quad X(0)=x, \notag 
\end{align}
where $B_{j}$ are $\mathbb{R}$-valued components of $B$. \rf{Under Assumptions~\ref{assump:g} and~\ref{assump:phi}},   \eqref{eq:diffusion local} is well-defined. The frame $g^{-1/2}(x):\R^q \rightarrow T_xM$ maps Brownian dynamics from $\R^q$ to $T_xM$,  and this leads to the diffusion coefficient $g^{-1/2}$ in \eqref{eq:diffusion local}.

On a Riemannian manifold, not necessarily compact, it is known \citep{bakry1986critere} that the diffusion $X$ does not explode, i.e., it is defined for any $t \geq 0$, if
\begin{equation}\label{eq:stoch complete}
	\Ric(v,v)+\Hess^{\rf{\phi}}(v,v) \geq -\kappa g(v,v) \,\, \text{ for all } \, (x,v) \in TM
\end{equation}
for some $\kappa>0$.  The condition is satisfied for any compact $M$, since  the Hessian of $\phi$ is bounded under Assumption~\ref{assump:phi}; moreover, the Ricci curvature is bounded from below since the sectional curvature at any point in $M$ is bounded from below \citep[pp. 167]{bishop2011geometry}. We will revisit this condition in Section~\ref{sec:non-comp} on non-compact manifolds, \rf{where we also discuss how it relates to the Euclidean case}.

A key requirement for sampling and estimation using $X(t)$ is its ergodicity. Recall that a diffusion $X(t)$ is ergodic if there exists a unique invariant measure $\mu$ of $X(t)$, and independently of the initial condition $x \in M$,  the limit
\begin{equation}
	\lim_{t\rightarrow \infty }\E\varphi (X_x(t))=\int_M \varphi (x)\,\diff \mu_\phi
	(x)=:\mu_\phi(\varphi)  \label{PA31}
\end{equation}%
exists for any bounded and measurable function $\varphi :$ $M \rightarrow \R$. 
Furthermore, the process $X(t)$ is exponentially ergodic if for any $x\in M$ and any bounded function $\varphi $ we have the following strengthening of (\ref{PA31}): 
\begin{equation}
	\left\vert \E\varphi (X_x(t))-\mu_\phi(\varphi) \right\vert \leq Ce^{-\lambda
		t},\ \ t\geq 0, 
  \label{PA34}
\end{equation}%
where $C>0$ and $\lambda >0$ are some constants. Since the process $X(t)$ is a non-degenerate diffusion on a compact smooth manifold $M$, its distribution converges at an exponential rate to the  invariant probability measure $\mu_\phi$ \citep[p. 97]{freidlin1985functional}.  Consequently, we can exploit the SDE (\ref{eq:diffusion}) (or its local version (\ref{eq:diffusion local})) for computing ergodic limits $\mu_\phi(\varphi)$.

Consider two examples of Langevin diffusions on compact $M$. 

\begin{example}\label{exmp:vmf}
	Let $M=\S^q$ with canonical round metric $g=\rf{\mathrm{d}} r^2+\sin^2 r\,\rf{\mathrm{d}} \theta^2$ with intrinsic coordinate $x=(r,\theta)$ for $r\in[0,\pi)$ and $\theta \in \S^{q-1}$. The Fisher-Watson distribution has the probability density function $p(x)\propto \exp(\lambda \cos^2 r )$ with respect to the volume form for $\lambda >0$, $r\in [0,\pi)$ \citep{watson1965equatorial}, and thus $\phi(x) = -\lambda \cos^2 r$, which clearly is in $C^{\infty}(M) \subset C^{3,\epsilon}(M)$. At least, two charts are needed for the local representation of $X$ since the cut locus of every point is its antipode, and non-empty. 
\end{example}

\begin{example}
	Let $M=\SO(m)$ be the space of real $m\times m$ orthogonal matrices with determinant 1, a manifold of dimension $q=m(m-1)/2$. Consider the bi-invariant metric $g(E_1,E_2) = -\frac{1}{2}\Tr(E_1E_2)$ for $x \in \SO(m)$, where $E_1,E_2 \in T_x \SO(m) \cong \{x\} \times \mathfrak{so}(m)$ and where $\mathfrak{so}(m)$ is the Lie algebra of $\SO(m)$ containing skew-symmetric matrices. The matrix generalisation of the von-Mises Fisher distribution, known as the matrix von-Mises distribution \citep{downs1972orientation,jupp1979maximum}, has the density function $p(x) \propto \exp(c\Tr(x_0x))$ for $c>0$ and $x,x_0 \in \SO(m)$. With $\phi = -c\Tr(x_0 x)$ it can be verified that for $x \in \SO(m)$ and $e \in \mathfrak{so}(m)$, $\Hess^\phi(xe,xe) = -c\Tr(xe^2)$, which shows that $\phi \in C^2(M)$ \citep[Section 4.3]{lewis2023contributions}.
\end{example}

\subsection{Partial differential equations on $M$ related to the Langevin diffusion}

Central to our analysis of estimators of a linear functional $\mu_\phi(\varphi)$ of the invariant measure $\mu_\phi$  are two PDEs that link the generator $\mathcal A$ to $\mu_\phi(\varphi)$. 

\subsubsection{Backward Kolmogorov PDE}
The backward Kolmogorov PDE on manifold $M$ associated with the SDE \eqref{eq:diffusion} \citep[e.g.,][]{ikeda2014stochastic} is given by: 
\begin{align}
	\frac{\partial u}{\partial t} (t,x) + \mathcal{A}u(t,x) &= 0, \;\;\;\;\; (t,x) \in [0,T]\times M; \nonumber\\ 
	u(T,x) &= \varphi(x), \;\;\;\;\; x \in M.
	\label{BKPDE}
\end{align}

Under Assumptions~\ref{assump:g} and \ref{assump:phi} and $\varphi \in C^{4,\epsilon}(M)$, there is a unique solution of the parabolic problem \eqref{BKPDE} which belongs to $C^{2,4,\epsilon}([0,T]\times M)$ \citep{PDE1,Lad68,lunbook,ikeda2014stochastic} and 
\begin{equation}\label{eq:BKmax}
	\Vert u \Vert_{C^{2,4,\epsilon } ([0,T]\times M)} \leq C \Vert \varphi \Vert_{C^{4,\epsilon }(M)},
\end{equation}
where $C>0$ is a constant independent of $\varphi$. 
The solution of \eqref{BKPDE} also has the following property \citep[Lemma 6.4]{TAL96}, which links geometric ergodicity of the Markov process $X(t)$ solving (\ref{eq:diffusion}) to the decay of the norm of derivatives of its semigroup: there exist constants $\lambda >0$ and $C >0$ that do not depend on $T$ and choice of $\varphi$ such that
\begin{equation}\label{eq:BKdecay}
\sum\limits_{k=0}^{2} \sum_{\substack{n=0 \\ k+n \neq 0}}^{4-2k}
\left\Vert \frac{\partial^k}{\partial t^k}u(t,\cdot)\right\Vert_{C^{n,\epsilon }(M)} \leq  
   C \Vert \varphi \Vert_{C^{4,\epsilon }(M)} e^{-\lambda (T-t)}.
\end{equation}   

For every $(t_0,x_0) \in [0,T] \times M$, the Feynman-Kac formula \citep{ikeda2014stochastic, freidlin1985functional} provides a probabilistic representation of the solution of \eqref{BKPDE} by linking it to the law of the Langevin diffusion $X(t)$ via
\begin{equation}\label{feynmankazformula}
	u(t_0, x_{0}) = \mathbb{E}(\varphi(X(T))),  
\end{equation}
where $X$ is from \eqref{eq:diffusion} with initial condition $X(t_{0}) = x_{0}$. The representation enables construction of the ensemble-averaging estimator of $\mu_\phi(\varphi)$ using multiple independent trajectories of the discretized $X^h_n$.  

\subsubsection{Poisson PDE}
The Poisson PDE on $M$ associated with the SDE \eqref{eq:diffusion} has the form
\begin{align}
	\label{PPDE}
	\mathcal{A}u(x) = \varphi(x) - \mu_\phi(\varphi), \quad x \in  M.
\end{align}
It is clear that $\int_{M}(\varphi(x) - \mu_\phi(\varphi)) \diff \mu_{\phi}(x) = 0$, i.e., the centering condition is satisfied  \citep[e.g.,][]{Miranda,freidlin1985functional}. Under Assumptions~\ref{assump:g} and \ref{assump:phi} and $\varphi \in C^{2,\epsilon}(M)$, there is a unique (upto an additive constant) solution of the elliptic problem (\ref{PPDE}) which belongs to $C^{4,\epsilon}(M)$ \citep{PDE1,PDE2,Miranda} and 
\begin{equation}\label{eq:Pmax}
	\Vert u\Vert_{C^{4,\epsilon } (M)}  \leq  C\Vert \varphi \Vert_{C^{2,\epsilon }(M)}, 
\end{equation}
where $C>0$ is a constant independent of a choice of $\varphi$.

Using It\^o's formula, one can derive the probabilistic representation of solution of (\ref{PPDE}) as \rf{\citep[e.g.,][]{freidlin1985functional,le2024diffusion}}
\begin{align*}
	u(x) = -\mathbb{E}\int_{0}^{\infty}\big(\varphi(X(t)) -\mu_\phi(\varphi)\big)\diff t + \mu_\phi(u).  
\end{align*}
The representation motivates an estimator of $\mu_\phi(\varphi)$ based on a long single trajectory of $X^h_n$, which can then be analysed by the mean ergodic theorem. 

\section{Langevin Sampling Algorithms} 
\label{section3}

Section \ref{sec:RLMC} introduces the intrinsic geodesic Euler discretization of the intrinsic Langevin SDE (\ref{eq:diffusion}) to sample from a given distribution $\mu_{\phi}$ on a Riemannian manifold. The resulting algorithm can be used on non-compact $M$ as well, and numerical experiments with the non-compact manifold of symmetric positive definite matrices in Sections~\ref{sec:numpd1} and~\ref{sec:numpd2} demonstrate this. Section~\ref{sec:retractions} generalises the Riemannian Langevin algorithm to one based on retractions, of which the exponential map is a special case. 
Section \ref{sec:other} discusses possible variants and extensions of the Riemannian Langevin algorithm. 

\subsection{Riemannian Langevin algorithm} 
\label{sec:RLMC}

For $T > 0$ and $t_{0} = 0$, consider a uniform partition $t_{0}<\dots<t_{N} = T$ of the interval $[0,T]$ with time step $h := T/N$, i.e. $t_{n+1} - t_{n} = h$, $ n = 0,\dots, N-1$.
In $\R^q$, the Euler method to discretize \eqref{eq:Langevin} assumes the form (see e.g. \cite{milstein2004stochastic})
\[
X^h_{n+1}=X^h_n-\dfrac{h}{2}\nabla \phi(X_n^h)+\sqrt{h}\xi_{n+1},
\]
where $\xi_n$, $n=1,\ldots,N$, are independent standard Gaussians on $\R^q$ such that $\sqrt{h}\xi_{n}$ represents a Brownian increment in a small time step of size $h$. Evidently, the  dynamics are then along the straight line from $X^h_n$ with respect to the direction 
\[
v_{n+1}:=-\frac{h}{2}\nabla \phi(X^h_n)+\sqrt{h}\xi_{n+1}.
\]

Its natural generalisation to the curved $M$ is via a Markov process whose next state moves along a geodesic with a direction chosen in the tangent space of the current state. Upon isometrically identifying the tangent space with $\R^q$, the following two-step procedure ensues:
\begin{enumerate}
	\item [(i)] In the tangent space $T_{X^h_n}M$, select a random direction $v_{n+1}$ via an $\R^q$-valued random vector $\xi_{n+1}$; 
	\item [(ii)] follow the geodesic $[0,1] \ni t \mapsto \gamma(t)\in M$ with $\gamma(0)=X^h_n, \dot \gamma (0)=v_{n+1}$ to $t=1$, and set $X^h_{n+1}:=\gamma(1)$. 
\end{enumerate}
Note that geodesics can be parameterized using the exponential map $TM \ni (x,v) \mapsto \Exp_x(v) \in M$ defined in Section \ref{sec_diff_geom_con} such that $\gamma(1)=\Exp_{X^h_n}(v_{n+1})$. Surjectivity of the exponential map on complete manifolds $M$ ensures that the two-step procedure is well-defined without requiring any restriction on the support of distribution of the random vector $\xi_{n+1}$. It is possible to replace the geodesic with  curves arising from retractions of suitable order without affecting the order of convergence of the algorithm (see Section~\ref{sec:retractions}).

Choice of the distribution of $\{\xi_n\}$ may be of practical consequence (Section \ref{sec:numpd}). Nevertheless, when studying weak convergence, it suffices to consider a distribution that matches $k$ moments of a Gaussian, where $k$ is related to the order of weak convergence \citep{milstein2004stochastic} and in the case of first order weak methods like the Euler scheme it is sufficient to match the first three moments. 
Thus, in contrast to the algorithms currently available in the literature \citep[e.g.,][]{sra2022} which use Gaussian distributed $\xi_n$ in $T_{X^h_n}$ with respect to the inner product $g(X^h_n)$, we only require that the components of $\xi^i_n, i=1,\ldots,q$ satisfy
\begin{align}\label{eq:xi}
	\E[\xi_n^i] = 0,\;\; \E[(\xi_n^i)^2] = 1,\; \;\E[(\xi_n^i)^3] = 0, \;\;\E[(\xi_n^i)^4] < \infty
\end{align}
for every $n=0,\ldots,N-1$. It suffices, for example, to use a discrete distribution so that for each $n=0,\ldots,N-1$, 
\begin{equation}\label{eq:xi2}
	P(\xi_n^i=\pm 1)=\frac{1}{2}, \quad i=1,\ldots,q.
\end{equation}

The  Markov chain $\{X^h_n,n=0,1,\ldots,N\}$ with $X^h_0=x$, defined as
\begin{equation}\label{eq:algorithm}
	{X}^h_{n+1} = \Exp_{{X}^h_n}\bigg(-\dfrac{h}{2}\nabla\phi({X}^h_n) + h^{1/2}g^{-1/2}({X}^h_n) \xi_{n+1}\bigg),
\end{equation}
represents the discretization of \eqref{eq:diffusion}. The intrinsic method of discretization without using an embedding into $\R^k, \;k \geq q,$ ensures that it suffices to use a $q$-dimensional random vector $\xi_n$, as opposed to a higher-dimensional one used in extrinsic approaches \citep[e.g.][]{vilmart2022, sharma2021, armstrong2022}. 

Importantly, use of the exponential map, or more generally, retractions (see Section~\ref{sec:retractions}), ensures that ${X}^h_{n}$ always lies on the manifold $M$. As discussed in the Introduction, this feature of the algorithms discussed here and in the next subsection is important for their long time stability (in terms of being able to use larger time steps $h$ and hence resulting in overall faster simulations), which is crucial for efficient sampling.

The proposed Riemannian Langevin algorithm may be fruitfully compared to the one proposed by \cite{TAL96}. In their paper an Euler scheme was defined in local coordinates of a chart, and the random variables were constrained to have compact support within the chart to ensure that ${X}^h_{n}$ and ${X}^h_{n+1}$ lie in the same chart  for sufficiently small $h$, to facilitate proof of weak convergence of the scheme. In contrast, our algorithm, and the subsequent proof of weak convergence, requires no such constraints (beyond the natural moment matching condition \eqref{eq:xi}) since the Markov chain arising from our algorithm moves directly along geodesics, and there is no need to verify that ${X}^h_{n}$  and ${X}^h_{n+1}$ belong to the same chart. In Section~\ref{sec:numMF}, we experimentally compare the use of Gaussian and discrete random variables within (\ref{eq:algorithm}).

\subsection{Retractions}\label{sec:retractions}

Closed form expressions for the exponential maps are available for very few manifolds $(M,g)$. Practical utility of the Riemannian Langevin algorithm may be enhanced by using more general retraction maps $F_x:T_xM \to M$ in place of the exponential map. Then the update rule for the Markov chain becomes 
	\begin{equation}
		{X}_{n+1}^{h} = F_{X^h_n}\left( -\dfrac{h}{2}\nabla\phi({X}^h_n) + h^{1/2}g^{-1/2}({X}^h_n) \xi_{n+1}\right). \label{eqn_3.5}
	\end{equation}
A \emph{retraction} is a smooth map from the tangent bundle $F: TM \to M$, whose restriction $F_x$ to $T_{x}M$ satisfies: (i) $F_x(0)=x$; (ii) $DF_x(0)=\text{id}_{T_xM}$, the identity mapping on $T_xM$. In the definition, $0$ is the origin of the concerned tangent space.  We associate with the retraction $F_{X_n^h}$ a curve $c$ which satisfies the initial conditions
 \begin{align}
     c(0) &= X_n^h, \label{init_con_ret1}\\
     \dot c(0) &= V:=-\frac{\sqrt{h}}{2}\nabla\phi(x) + g^{-1/2}(x)\xi.\label{init_con_ret2}
 \end{align}
  Then by \eqref{eqn_3.5}, $c(\sqrt{h}) = X_{n+1}^h$.

There are several ways to construct retractions. 
One way of obtaining a computationally inexpensive retraction is via numerical approximation of the geodesic equation. 
The geodesic equation \eqref{eq:geodesic eq} takes the following form in local coordinates:
\begin{align}\label{eq:geod2}
    \begin{bmatrix} \dot{\gamma} \\ \rf{\dot{v}} \end{bmatrix} = \begin{bmatrix}
        \rf{v} \\ - \Gamma_{ij} \rf{v^{i}} \rf{v^{j}}
    \end{bmatrix} 
\end{align}
 with initial condition 
 \begin{align}
     \begin{bmatrix}
         \gamma(0) \\ \rf{v(0)} 
     \end{bmatrix}
     = \begin{bmatrix}
         x \\ V
     \end{bmatrix},
 \end{align}
 where $$V = -\frac{\rf{\sqrt{h}}}{2}\nabla \phi(x)  + g^{-1/2}(x) \xi .$$

To obtain the first weak order of convergence of the Riemannian Langevin algorithm (\ref{eqn_3.5}) as in the case of the exponential map (\ref{eq:algorithm}), it is necessary to impose conditions on the retraction. This relates to the leading $h^2$ terms in the local error expansion of Lemma~\ref{eq:onesteplemma}. Consequently, we make the following assumption\rf{.}

 \begin{assumption} \label{assump:2nd order ret}
 For a curve $c:[0,\sqrt{h}]\rightarrow M$ with initial conditions (\ref{init_con_ret1})-(\ref{init_con_ret2})
 corresponding to the retraction $F_x$, we require
		\[
  |\E[D_{\dot c(s)}\dot c(s)|_{s=0}]| \leq Ch
  \]
  and
\[
  |\E[D_{\dot c(s)}D_{\dot c(s)}\dot c(s)|_{s=0}]|\leq C h^{1/2},
  \]
  where $C>0$ is independent of $h$. 
\end{assumption}

 To illustrate our retraction based approach, we solve the geodesic equation \eqref{eq:geod2}  by using a single step of size $\sqrt{h}$ of  the standard $(1/6,1/3,1/3,1/6)$ 4th-order Runge-Kutta method (RK4) (see e.g. \cite{HNW93}). This is employed in Section~\ref{sec:numMF} on the 2-sphere $\S^2$, where the exponential map is not readily available in analytic form in \rf{the} intrinsic coordinates $(r,\theta)$ for the round metric.  \rf{Since this approximation occurs within a specific, localized coordinate system, the resulting solution belongs to the manifold $M$,} and hence a projection onto $M$ is not required, i.e., there is no additional loss in accuracy.

We now verify that the curve obtained by applying RK4 method to \eqref{eq:geod2} with time step $s \in (0,\sqrt{h}]$ satisfies Assumption~\ref{assump:2nd order ret}.
 We denote
\begin{align}
p(s)=
 \begin{bmatrix}
        c(s) \\ \tilde{c}(s) 
    \end{bmatrix} :=   \begin{bmatrix}
        \gamma(s) \\ v(s) 
    \end{bmatrix}  + e(s) =  \begin{bmatrix}
        \gamma(s) \\ v(s) 
    \end{bmatrix}  +  \begin{bmatrix}
        e_{1}(s) \\ e_{2}(s) 
    \end{bmatrix}, 
\end{align}
where $|e_i(s)| \leq Cs^{5}$, $i = 1,2$. Note that, since the curve $p(s)$ in $TM$ corresponds to the one step of RK4 method with step $s$, we can write the error $e_{1}(s)$ as the following expansion \citep{HNW93} in terms of $s$:
\begin{align}\label{eq:310}
    e_{1}(s) = s^{5}a_{1} + s^{6}a_{2} + s^{7}a_{3} + {O}(s^{8}), 
\end{align}
where $a_{i} \in \mathbb{R}^{q}$, $ i = 1,\dots,3$, are independent of $s$.

Writing $D_{\dot{c}(s)} \dot{c}(s)$ in local coordinates we obtain 
\begin{align*}
      \frac{\diff^{2}}{\diff  s^{2}}c^{k} &+ \Gamma_{ij}^{k}\frac{\diff c^{i}}{\diff s}\frac{\diff c^{j}}{\diff s} =  \frac{\diff^{2}}{\diff  s^{2}}\gamma^{k} + \Gamma_{ij}^{k}\frac{\diff\gamma^{i}}{\diff s}\frac{\diff\gamma^{j}}{\diff s} +  \frac{\diff^{2}}{\diff s^{2}}e^{k}_{1} + \Gamma_{ij}^{k}\Big(\frac{\diff e^{i}_{1}}{\diff s}\frac{\diff e^{j}_{1}}{\diff s} + \frac{\diff\gamma^{i}}{\diff s}\frac{\diff e^{j}_{1}}{\diff s}+ \frac{\diff e^{i}_{1}}{\diff s}\frac{\diff \gamma^{j}}{\diff s}\Big) \\&  = \frac{\diff^{2}}{\diff  s^{2}}e^{k}_{1} + \Gamma_{ij}^{k}\Big(\frac{\diff e^{i}_{1}}{\diff s}\frac{\diff e^{j}_{1}}{\diff s} + \frac{\diff \gamma^{i}}{\diff s}\frac{\diff e^{j}_{1}}{\diff s}+ \frac{\diff e^{i}_{1}}{\diff s}\frac{\diff\gamma^{j}}{\diff s}\Big),
\end{align*}
where $ k = 1,\dots,q$. It is then clear that $D_{\dot{c}(s)} \dot{c}(s)|_{s= 0} = 0$ a.s..  Let $(c,\vartheta)$ be a curve on $TM$ such that $\vartheta(s) \in T_{c(s)}M$ and $k-$th component of $\vartheta (s)$ is given by
\begin{align}
     \vartheta^k(s):=\frac{\diff^{2}}{\diff s^{2}}e^{k}_{1} + \Gamma_{ij}^{k}\Big(\frac{\diff e^{i}_{1}}{\diff s}\frac{\diff e^{j}_{1}}{\diff s} + \frac{\diff \gamma^{i}}{ \diff s}\frac{\diff e^{j}_{1}}{\diff s}+ \frac{\diff e^{i}_{1}}{\diff s}\frac{\diff \gamma^{j}}{\diff s}\Big),
\end{align}
and we are interested in $D_{\dot{c}(s)}D_{\dot{c}(s)} \dot{c}(s) = D_{\dot{c}(s)}\vartheta (s)$. We can write $D_{\dot{c}(s)}\vartheta (s)$ in local coordinates as
\begin{align*}
\frac{\diff}{\diff s}\vartheta^{k} + \Gamma_{ij}^{k}\vartheta^{i}\frac{\diff}{\diff s}c^{j}(s),
\end{align*}
which  is of order $O(s^{2})$ since all the terms have a derivative of $e_1$ and the highest derivative is of 3rd order (see (\ref{eq:310})). Consequently, $D_{\dot{c}(s)}D_{\dot{c}(s)} \dot{c}(s)|_{s = 0} = 0$ a.s.. Thus, the approximation of the geodesic equation \eqref{eq:geodesic eq} \rf{by the one step of RK4 with time step $\sqrt{h}$} satisfies Assumption~\ref{assump:2nd order ret}. 

For other applications of retractions, see \cite{schwarz2023efficient}, where retraction-based geodesic random walks on compact Riemannian manifolds were studied, and \cite{absil2008optimization}, where retractions are used in optimization algorithms on manifolds.

\subsection{Other Langevin-based sampling algorithms}\label{sec:other}

It is always beneficial to have a range of algorithms for solving problems numerically. Here we consider two further variants of the Riemannian Langevin algorithm, each of which again ensures that $X^h_n$ always stays on $M$. 
Similarly to the algorithms from the previous two subsections, these algorithms are likewise of weak order 1 (see Remark~\ref{rem:other}). 

\begin{itemize}[leftmargin=*]
	\item Consider different random variables than $\xi_n$.
	
	An example of which is the generalisation of the \textit{walk-on-sphere} \citep[Section 7.1.1]{milstein2004stochastic} on $\R^q$ to Riemannian manifolds given by
	\begin{equation}
		X_{n+1}^{h} =  \Exp_{{X}^h_n}\bigg(-\dfrac{h}{2}\nabla\phi({X}^h_n) + \sqrt{qh}g^{-1/2}({X}^h_n) \eta_{n+1}\bigg),\label{eqn_3.4}
	\end{equation}
	where $\{\eta_{n}\}$ are independent uniformly \rf{distributed on the unit sphere in $\R^q$}, and $\sqrt{q}$ is used to take into account that $\mathbb E(\eta_n^i \eta_n^j)=\delta_{ij}/q$ so that $\sqrt{q} \eta_n^i $ satisfy the conditions on moments (\ref{eq:xi}).

	\item Consider \textit{splitting methods} that divide each iteration of the sampling algorithm into first approximating the drift $\nabla \phi$ followed by approximating a Brownian increment.

	An example is the following Markov chain based on the idea of splitting of the SDE flow \citep[e.g.,][]{milstein2004stochastic}:
	\begin{align}
		\hat{X}_{n+1}^{h} &= F_{X_{n}^{h}}\bigg(-\frac{h}{2}\nabla \phi(X_{n}^{h})\bigg),\label{eqn_3.6}\\ 
		X_{n+1}^{h} &= P_{\hat{X}_{n+1}^{h}}\big(h,\xi_{n+1}\big),\label{eqn_3.7}
	\end{align}
	where $F$ can be the exponential map or a retraction as described above, $\hat{X}_{n+1}^{h}$ denotes the intermediate step, and $P(h,\xi_{n+1}) : TM \rightarrow M$ and $\xi_{n+1}$ are such that
	\begin{align}
		\mathbb{E}\Big(f\big(P_{\hat{X}_{n+1}^{h}}(h,\xi_{n+1})\big)\Big) = f(\hat{X}_{n+1}^{h})
		+ \frac{h}{2}\Delta_{M}f(\hat{X}_{n+1}^{h}) + {O}(h^{2})
	\end{align}
	for any $f \in C^{4, \epsilon}(M)$. 
	The above relation can be satisfied, for example, if $P$ is an exact simulation of Brownian motion from $t_{n}$ to $t_{n} +h$, or if 
	\begin{align}
		P_{\hat{X}_{n+1}^{h}}\big(h,\xi_{n+1}\big) = F_{\hat{X}_{n+1}^{h}}\bigg(h^{1/2}g^{-1/2}(\hat X_{n+1}^{h})\xi_{n+1}\bigg).
	\end{align}
	The sampling method of \cite{vempala2022convergence} for Hessian manifolds can be considered as a special case of (\ref{eqn_3.6})-(\ref{eqn_3.7}) with $F$ being exponential map at $X_{n}^{h}$ and $P$ being exact simulation of Brownian motion starting at $\hat{X}_{n+1}^{h}$. 
	
\end{itemize}


\section{Error Bounds} 
\label{sec:bounds}
In this section, first-order weak error bounds are given for the ensemble- and time-averaging estimators of $\mu_\phi(\varphi)=\int_{M}\varphi \diff \mu_{\phi}$, based, respectively, on a single and multiple independent trajectories of $X^h_n$ obtained from \eqref{eq:algorithm} or \eqref{eqn_3.5}.
Proofs of the three theorems stated here are in the Appendix. We prove the theorems under the stated sufficient conditions which include smoothness assumptions on $\phi$, $\varphi$ and $g$. At the same time, the algorithms presented can be successfully used when, for example, $\varphi$ is less smooth. In this respect we mention that in the Euclidean case (i.e, $M=\R^q$) weak convergence of the Euler scheme with order one was proved for $\varphi$ being just measurable and bounded \citep{BT95}. Relaxing smoothness conditions for the theorems we prove here is a subject of future study.

Theorem~\ref{theorem4.1} is proved for \rf{the} Riemannian Langevin algorithm using a retraction \eqref{eqn_3.5}. For clarity of exposition, Theorems~\ref{theorem4.2} and \ref{theorem4.3} are proved for the algorithm \eqref{eq:algorithm} which uses the exponential map; proofs of these can be generalised to accommodate retractions in a manner similar to the proof of Theorem~\ref{theorem4.1}.

\subsection{Ensemble-averaging estimator}
Equation (\ref{PA34}) implies that $\mu_\phi(\varphi)$ may be estimated upon using $\mathbb{E}(\varphi(X(T)))$ by choosing a sufficiently large $T$. This motivates the ensemble-averaging estimator $\hat \mu_{\phi,N}(\varphi)$ based on independent realizations  $\{X^{(l),h}_{N},l=1,\ldots,L\}$ of $X^h_n$ from the algorithm \eqref{eqn_3.5} up until the final step $N$ (recall that $T=Nh$), since
\begin{align}\label{eq:est_ea}
	\hat \mu_{\phi,N}(\varphi)= \frac{1}{L}\sum_{l=1}^{L}\varphi(X^{(l),h}_{N})\approx 
	\mathbb{E}(\varphi(X_N^h)) \approx \mathbb{E}(\varphi(X(T))) \approx
	\mu_\phi(\varphi)  .  
\end{align}
There are three sources of errors \rf{of} $\hat \mu_{\phi,N}(\varphi)$: (i) proximity to the ergodic limit controlled by a choice of $T$; (ii) the numerical integration error controlled by the time step $h$; and (iii) the Monte Carlo error controlled by the number of Monte Carlo runs $L$. The following result shows how the weak upper bound on its bias depend on both $h$ and $T$, and in a specific way on the class of functions $\varphi$. Note that the initial point $x$ in (\ref{eq:diffusion local}) can be random, independent of the Wiener process $w(\cdot)$, and, consequently, $X_0^h$ in the ensemble-averaging setting can be random, independent of $\xi_n$.
\begin{theorem}[\textbf{Bias of the ensemble averaging estimator}]
	\label{theorem4.1}
	Under Assumptions~\ref{assump:g} and \ref{assump:phi} and $\varphi \in C^{4,\epsilon}(M)$, the following bound holds for the the Riemannian Langevin algorithm with retraction \eqref{eqn_3.5} satisfying Assumption~\ref{assump:2nd order ret}:
	\begin{align}\label{eq:thm41}
		|\mathbb{E}(\hat \mu_{\phi,N}(\varphi))- \mu_\phi(\varphi)| \leq C \Big(h + e^{-\lambda T}\Big),
	\end{align}
	where $C, \lambda$ are positive constants independent of $h$ and $T$ and the constant $C$ linearly depends on $\Vert \varphi \Vert_{C^{4,\epsilon }(M)}$ but otherwise $C, \lambda$ are independent of $\varphi$. 
\end{theorem}


Variance of the estimator $\hat \mu_{\phi,N}(\varphi)$ is estimated in the usual way for sample means (see also Section~\ref{sec:exp}): 
\begin{align}\label{eq:eaVar}
	\mathrm{Var}(\hat \mu_{\phi,N}(\varphi))= \frac{1}{L}\mathrm{Var}(\varphi(X_N^h))=\frac{1}{L}[\mathrm{Var}(\varphi(X(T)))+ O(h)] \\
	=\frac{1}{L} [ \mu_\phi(\varphi^2)-(\mu_\phi(\varphi))^2+ O(h+e^{-\lambda T})]. \notag
\end{align}

\subsection{Time-averaging estimator} 
Since $X(t)$ is an ergodic process, we also have
\begin{align}
	\lim\limits_{T \rightarrow \infty}\frac{1}{T}\int_{0}^{T}\varphi(X(t))\diff t  =  \mu_\phi(\varphi),\;\; a.s., \label{ergotime}
\end{align}
which suggests  that we can take $\frac{1}{T}\int_{0}^{T}\varphi(X(t))\diff t$ as time-averaging estimator for $\mu_\phi(\varphi)$. Hence, the numerical time-averaging estimator is
\begin{align}\label{eq:est_ta}
	\tilde \mu_{\phi,N}(\varphi):= \frac{1}{N}\sum_{n=0}^{N-1}\varphi(X_{n}^h).  \end{align}
In time-averaging estimation, we simulate a long trajectory and average $\varphi$ at the discretized points collected along the long trajectory  as can be seen in (\ref{eq:est_ta}). There are again three errors associated with the estimator (\ref{eq:est_ta}): (i) due to proximity to the ergodic limit controlled by a choice of $T$ (i.e., by the choice of the number of steps $N$ under fixed time step $h$), (ii) the numerical integration error controlled by the time step $h$, and (iii) due to variance of the estimator controlled by $T$.

\begin{theorem}[\textbf{Bias and mean-square error of the time-averaging estimator}]
	\label{theorem4.2}
	Under Assumptions~\ref{assump:g} and \ref{assump:phi} and $\varphi \in C^{2,\epsilon}(M)$, the following bounds hold for the Riemannian Langevin algorithm (\ref{eq:algorithm}):
	\begin{align}
		|\mathbb{E}(\tilde \mu_{\phi,N}(\varphi)) - \mu_\phi(\varphi)| &\leq C \bigg(h + \frac{1}{T}\bigg), \label{theorem4.2est1}\\
		\mathbb{E}( \tilde \mu_{\phi,N}(\varphi)- \mu_\phi(\varphi))^{2} &\leq C \bigg(h^{2} + \frac{1}{T}\bigg), \label{theorem4.2est2}
	\end{align}
	where $C > 0$ is independent of $h$ and $T$ and it linearly depends on $\Vert \varphi \Vert_{C^{2,\epsilon }(M)}$ but otherwise it is independent of $\varphi$.  
\end{theorem}
We remark that the error $C/T$ in (\ref{theorem4.2est1})-(\ref{theorem4.2est2}) is due to the two sources of errors - (i) and (iii) as described before the theorem.  They correspond to properties of the considered SDE (\ref{eq:diffusion}). Further, it is not difficult to see that 
the estimates (\ref{theorem4.2est1})-(\ref{theorem4.2est2}) imply that $\Var(\tilde \mu_{\phi,N}(\varphi))\leq C (h^2 + 1/T)$, 
however one can prove under a mixing condition that $\Var(\tilde \mu_{\phi,N}(\varphi))\leq C/T$ with $C>0$ independent of $h$ and $T$ (see \cite[Section 7]{MST10} for more details when $M=\mathbb T^q$).
\rf{For a comparison of computational costs between ensemble averaging and time averaging see \citep[Remark~3.9]{Msisc25}.}

\subsection{Bound on distance to invariant measure} 
The exponential map $TM \ni (x,v)\mapsto \Exp_x(v) \in M$, and more generally retractions $F_x$, are continuous functions of $x \in M$. This and Assumption~\ref{assump:phi} imply that the recursion
\begin{equation*}
	X_{n+1}^{h} = \Phi(X_{n}^{h}, \xi_{n+1})
\end{equation*}
underlying the algorithm engenders a continuous map $\Phi:M \to M$ for every realization of $\xi$. As a result, $\{X^h_n\}$ arising from the proposed algorithm is a Feller chain: for a sequence $x_k \in M$ such that $ \rho(x_k, x) \rightarrow 0$ as $k \rightarrow \infty$, $\E \varphi(\Phi(x_k, \xi)) \rightarrow \E \varphi(\Phi(x, \xi))$ as $ k \rightarrow \infty$ for $\varphi$ belonging to the class of bounded and continuous functions. A Feller chain evolving on a compact separable state space has a stationary measure (not necessarily unique) by the Krylov-Bogoliubov theorem \citep{DaPrato96}. Denote by $\mu_{\phi}^{h}$ a stationary measure of the Markov chain $X_n^h$ from \eqref{eq:algorithm}. The following result provides an upper bound on an integral probability metric between $\mu^h_\phi$ and $\mu_\phi$ with respect to a class of test functions. 

\begin{theorem}[\textbf{Distance to invariant measure}]~\\
	\label{theorem4.3} 
	Under the assumptions of Theorem~\ref{theorem4.2}, with 
	\begin{align*}
		\dist(\mu, \nu) := \sup_{\varphi \in \mathcal{H}}\left| \int_M \varphi \mathrm{d}\mu - \int_M \varphi \mathrm{d} \nu\right|, \label{eq:distme}
	\end{align*} 
	where $\mathcal{H} := \{\varphi \in C^{2,\epsilon}(M) \text{ and }  \Vert \varphi \Vert_{C^{2,\epsilon }(M)} \leq 1 \}$, 
	the following bound holds:
	\begin{align}
		\dist(\mu_\phi,\; \mu_{\phi}^{h}) &\leq Ch,  
	\end{align}
	with $C > 0$ being independent of $h$.  
\end{theorem}

\begin{remark}\label{rem:other}
	Results similar to those in Theorems ~\ref{theorem4.1}-\ref{theorem4.3} can be proved for the two variants of the Riemannian Langevin algorithm discussed in Section~\ref{sec:other}. This is a consequence of the proof techniques used, which exploit links between the backward Kolmogorov and Poisson PDEs and the semigroup of the Langevin diffusion. 
\end{remark}

\subsection{Error extrapolation: increasing accuracy order and controlling bias}\label{sec:TT}

 Analogous to the Euclidean case \citep{TAT90} (see also \cite[Section 2.2.3]{milstein2004stochastic}), it is possible to expand  the errors of the bias of the estimators considered earlier in this section in powers of time increment $h$ assuming additional smoothness of $g$, $\phi$ and $\varphi$, which is known as the Talay-Tubaro (or Richardson-Runge) extrapolation. For brevity,  we skip technical details and merely highlight its practical importance. For definiteness, let us consider the bias of the ensemble-averaging estimator $\hat \mu_{\phi,N}(\varphi)$ with $X^h_n$ obtained from \eqref{eq:algorithm}, but it will become clear that the estimator $\tilde \mu_{\phi,N}(\varphi)$ may be treated analogously. Under appropriate smoothness assumptions, one can prove that 
\begin{align}\label{eq:TT}
		\lim_{N \to \infty}\mathbb{E}(\hat \mu_{\phi,N}(\varphi))- \mu_\phi(\varphi)=C_{\rf{1}}h+\cdots +C_{p}h^{p}+O(h^{p+1}),
\end{align}
where the constants $C_{\rf{1}},\ldots ,C_{p}$ are independent of $h,$ $T$, and $p \geq \rf{2}$ is an integer ($p$ can be arbitrarily large if $g$, $\phi$ and $\varphi$ belong to $C^{\infty }(M))$.

Choose a sufficiently large $T$ so that $ e^{-\lambda T}$ in (\ref{eq:thm41}) is negligible, and take two time steps $h_1=h$ and $h_2=\alpha h,$ $\alpha >0,$ $\alpha \neq 1,$ with the corresponding number 
of steps being $N_1$ and $N_2$, respectively. 
Then, according to (\ref{eq:TT}), 
\begin{eqnarray}
\mu_\phi(\varphi) & \approx &\mathbb{E}(\hat \mu_{\phi,N_1}(\varphi))+C_{\rf{1}} h_{1}+O(h^{2})\,,  \label{Db234} \\
\mu_\phi(\varphi) &\approx & {\E}(\hat \mu_{\phi,N_2}(\varphi))+C_{\rf{1}} h_{2}+O(h^{2})\,,  \notag
\end{eqnarray}%
whence 
\begin{equation}
C_{\rf{1}} \approx -\frac{\mathbb{E}(\hat \mu_{\phi,N_2}(\varphi))-\mathbb{E}(\hat \mu_{\phi,N_1}(\varphi))}{h_{2}-h_{1}}+O(h)\,.  \label{Db235}
\end{equation}%
By (\ref{Db234}) and (\ref{Db235}), we get the improved value with error $O(h^{2}):$ 
\begin{equation}
\mu_{\phi}(\varphi)_{imp} := \mathbb{E}(\hat \mu_{\phi,N_1}(\varphi))\frac{h_{2}}{h_{2}-h_{1}}
-\mathbb{E}(\hat \mu_{\phi,N_2}(\varphi))\frac{h_{1}}{h_{2}-h_{1}}. \label{Db236}
\end{equation}%
Thus, the obtained method has an accuracy of order two in $h$. We can continue exploiting (\ref{eq:TT}) and obtain 
a method of order three and so on. On the other hand, $C_{\rf{1}} h$ gives the leading term of the bias of $\hat \mu_{\phi,N_1}(\varphi)$, which can be used for error estimation in practice to assess the quality of the approximation. \rf{In Section~\ref{sec:numMF}, we apply the Talay-Tubaro extrapolation method to generate samples of the von-Mises Fisher distribution with bias of order 2.}
 
 \subsection{Finite-time convergence}\label{sec:FT}

Although this paper is focused on the use of SDEs to sample from a given distribution on $M$, and hence on the estimators for ergodic limits from Section~\ref{sec:bounds}, for some applications it may also be of interest to discuss weak convergence of a Riemannian Langevin algorithm at a finite time $T$; this, for example, can be used to evaluate quantities related to transient behavior of diffusions on manifold \rf{and} solving linear parabolic PDEs on manifolds by exploiting probabilistic representations (see \eqref{feynmankazformula}). 
It is straightforward to prove the following weak-sense convergence theorem along the lines of the proof of Theorem~\ref{theorem4.1}.
\begin{theorem}[\textbf{Finite-time convergence}]
	\label{theoremFT}
	Under Assumptions~\ref{assump:g} and \ref{assump:phi} and $\varphi \in C^{4,\epsilon}(M)$, the Riemannian Langevin algorithm (\ref{eq:algorithm}) for (\ref{eq:diffusion}) is of weak order one, i.e.,
	\begin{align}
		|\mathbb{E}[\varphi(X_N)- \varphi(X(T))]| \leq C h,
	\end{align}
	where $C>0$ is a constant independent of $h$. 
\end{theorem}

\begin{remark}
Theorem~\ref{theoremFT} is more broadly applicable, and holds for weak-sense schemes similar to those in Section~\ref{section3} for a general class of SDEs 
\begin{equation}\label{eq:sdewitha}
\rf{\mathrm{d} X(t)=a(t,X(t)) \mathrm{d} t + \mathrm{d} B^M(t)},\quad X(t_0)=x,
\end{equation}
with a sufficiently smooth vector field $a(t,x)$ on $[t_0,T]\times M$. 
\end{remark}

\rf{\begin{remark}
Results in this section demonstrate that the main techniques in the Euclidean setting on weak approximations and sampling carry over to the setting of compact manifolds $M$. 
In contrast, the situation with mean-square (strong) convergence is more delicate.  
For manifolds $M$ where every point has an empty cut-locus, the Riemannian Langevin algorithm (\ref{eq:algorithm}) with $\{\xi_n\}$ being independent standard Gaussian random variables has mean-square (strong) order $1/2$ \citep{man_msq}. Although this matches the result for Euclidean SDEs with multiplicative noise \citep[e.g.][Chapter 1]{milstein2004stochastic}, the proof in the manifold setting is substantially more complex.
If a manifold contains points with non-empty cut-loci (e.g. positively curved manifolds) then we believe that mean-square schemes should be based on space-time random walks; see \citet[Chapter 6]{milstein2004stochastic}, and also \citet{MT99} for a mean-square approximation of Euclidean SDEs in bounded domains. This is in contrast to the weak-sense convergence, where the Riemannian Langevin algorithm on compact $M$ has weak order $1$, and is agnostic to the presence or absence of cut points. 
\end{remark}}

\section{The Case of Non-compact Manifolds}
\label{sec:non-comp}

In this section we discuss ingredients needed to transfer the error analysis of the preceding section from compact to non-compact $M$. As mentioned in Section~\ref{sec:intro}, negative curvature encourages Brownian motion to drift away to infinity, and a uniform lower bound on the Ricci curvature is needed to address this. Additionally, since non-compactness is closely linked with non-positive curvature, when the Brownian motion has a drift component, as with an $\R^q$-valued Brownian motion, extra care is needed. 

The condition \eqref{eq:stoch complete} is sufficient to ensure stochastic completeness of a diffusion with generator $\mathcal A$ in \eqref{eq:operA} \citep{li2021stochastic}; recall that it is automatically satisfied for compact $M$. Hence here we state  \eqref{eq:stoch complete}  as the below assumption.
\begin{assumption}
  \label{assump:stochcomplete}
		There exists a $\kappa>0$ such that
		$$\Ric(v,v) + \Hess^\phi(v,v) \geq -\kappa g(v,v), \quad \text{for all }(x,v) \in TM.$$
\end{assumption}
Conditions in Assumption \ref{assump:stochcomplete} may be further relaxed \citep[Corollary 2.1.2]{Wang2013}, which, however, are harder to verify in practice. We note that Assumption 4 in the case of $M=\R^q$ resembles a one-sided Lipschitz condition which is sufficient to ensure regularity of SDEs in  $\R^q$ as well as finiteness of moments of their solutions \citep{khasminskii2011stochastic}.

Ergodicity, on the other hand, requires more. The \emph{Bakry-\`Emery criterion} \citep{BE85} demands the existence of a $\kappa>0$ such that
\begin{equation}\label{eq:bakry-emery}
	\Ric(v,v) + \Hess^\phi(v,v) \geq 2\kappa g(v,v)\,\, \text{ for all }\, (x,v) \in TM,
\end{equation}
which jointly controls both curvature of $M$ and drift of the Brownian motion, and is more restrictive than \eqref{eq:stoch complete}. For example, when $M=\R^q$, the criterion \eqref{eq:bakry-emery} reduces to $\Hess^\phi \geq 2\kappa g$ for $\kappa>0$, and implies that the density $p$ is $2\kappa$-strongly log-concave. The Bakry-\`Emery criterion provides a sufficient condition for exponential ergodicity of the diffusion \eqref{eq:diffusion} on non-compact $M$, but is redundant for compact $M$. Recall that by construction, $\mu_\phi$ is the invariant measure of the diffusion if it is ergodic. 

A few intermediate conditions signpost the route to exponential ergodicity of the Langevin diffusion that satisfies the Bakry-\`Emery criterion. The criterion ensures that the semigroup $\{P_t\}$ of operators corresponding to the diffusion satisfy the Log-Sobolev inequality with constant $2\kappa$  \citep{wang2009log}, which then implies that they satisfy the Poincar\'e inequality: $\mathrm{Var}_{\mu_\phi}(\varphi) \leq C\mathcal{E}(\varphi,\varphi)$ for a Dirichlet form $\mathcal{E}$ and constant $C$. As a consequence, one gets $|{P_t}|_{\op} \leq e^{-t/C}$, which implies exponential ergodicity. 

The Bakry-\`Emery criterion, however, is too stringent for our needs, as it precludes the possibility of sampling from a sizeable class of interesting probability measures on $M$. It is instead more natural to impose separate conditions on $\Ric$ and $\Hess^\phi$. The following assumption is sufficient for the semigroup of operators corresponding to the Langevin diffusion to satisfy the Log-Sobolev inequality \citep[Theorem 1.1]{wang2009log}.

\begin{assumption}\label{assump:ric}
		For constants $b,c>0$, 
  \begin{equation}\label{eq:ric}
  \Ric\geq (-c-b^2\rho_o^2)g,
  \end{equation}
  where $\rho_o=\rho(o,x)$, $o,x \in M$.
The Hessian 
\begin{equation}\label{eq:hess}
\Hess^\phi \geq \delta g
\end{equation}
outside of a compact set in $M$, where the constant $\delta$ satisfies $\delta>(1+\sqrt{2})b \sqrt{q-1}>0$.
\end{assumption}

In the case $M=\R^q$ the assumption (\ref{eq:hess}) is $\delta$-strong convexity condition to be satisfied outside of a compact set in $\R^q$ (which means that potential function can be locally non-convex):
\[
((x-y),\frac{1}{2}(\nabla \phi(x) - \nabla \phi(y))) \geq \delta |x-y|^2, \,\,\, \delta>0,
\]
which together with constant diffusion coefficient is sufficient for exponential ergodicity of the corresponding SDE  in $\R^q$ 
\rf{(where $\Ric = 0$)} \citep{khasminskii2011stochastic,RT96}. 

\subsection{Examples}

A few examples elucidate on Assumptions~\ref{assump:stochcomplete} 
and~\ref{assump:ric} for non-compact $M$. 

\begin{example}
	Let $M$ be  the \rf{Poincar\'e half plane} model $\H^q$ for the $q$-dimensional hyperbolic space with metric $g=\frac{1}{x_q^2}\sum_{i=1}^q \mathrm{d} x_i^2$, under which $M$ has negative sectional curvature equal to $-1$ everywhere. From the Cartan-Hadamard theorem, it is diffeomorphic to $\R^q$. Consider the Riemannian-Gaussian distribution with density $p(x) \propto \exp(-\frac{\rho(x,o)^2}{2\sigma^2})$, with respect to $\dvol$, where $o \in M$ is a fixed point and $\sigma>0$;  for brevity, let $\rho:=\rho(x,o)$ be the distance function to the fixed $o$. 
	
	The Hessian of the squared distance  satisfies $\Hess^{\rho^2} \geq 2g$ \citep{greene2006function}, and since $\H^q$ is an Einstein manifold, its Ricci curvature  $\Ric = -(q-1)g$ has a lower bound. Assumption \ref{assump:stochcomplete} is satisfied since $\Ric+\Hess^\phi \geq (\frac{1}{\sigma^2}-q+1) g$, ensuring stochastic completeness. Note that the lower bound need not be positive as $\sigma>0$ is unbounded, therefore the Bakry-\`Emery criterion may be violated for some values of $\sigma$. 
	
	The lower bound on the Hessian means that we can restrict attention to a geodesic ball cent\rf{er}ed at $o$ of finite radius to choose the values $c=(q-1)$, $b = ((1+\sqrt{2})\sigma^2\sqrt{q-1})^{-1/2}$ and $\delta = \frac{1}{\sigma^2}$ such that Assumption~\ref{assump:ric} is satisfied in tandem. Convergence exponentially fast to $p \thinspace \dvol$ is thus assured for any $\sigma>0$, including for values that violate the Bakry-\`Emery criterion (\ref{eq:bakry-emery}).
\end{example}

\begin{example}\label{example:spd}
	The space $M=\P_m$ consisting of $m\times m$ symmetric positive definite matrices with affine-invariant metric $g$ in \eqref{eq:affineinvariant} is a $q=m(m+1)/2-$dimensional manifold. As with $\H^q$, the space $\P_m$ with metric $g$ has a negative sectional curvature everywhere, which, however, unlike $\H^q$, is not constant. 
	Assumptions \ref{assump:stochcomplete} and \ref{assump:ric} 
 are verified for the following distributions, later used in Section \ref{sec:numpd}.
	
	\begin{enumerate}[label=(\roman*), leftmargin=*,itemsep=1em]
		\item Let $\phi(X)=\frac{1}{2\sigma^2} \rho(X,O)^2$ in \eqref{invmeasure} for $\sigma>0$ and fixed $O \in \P_m$. This choice of $\phi$ results in the Riemannian-Gaussian distribution (Section \ref{sec:numpd1}) on $\P_m$. It can be verified that $\Ric \geq -\frac{m}{4}g$, and by the Hessian comparison theorem $\Hess^{\rho^2} \geq 2g$ \citep[e.g.,][Section 4.4]{lewis2023contributions}. For $ \kappa = \left| \frac{1}{\sigma^2} - \frac{m}{4}\right|$, note that
		$$ \Ric+\Hess^\phi \geq \bigg(\frac{1}{\sigma^2} - \frac{m}{4}\bigg) g ,$$
		and Assumption~\ref{assump:stochcomplete} is satisfied. For (\ref{eq:ric}) in Assumption~\ref{assump:ric}, we use the lower bound on the Ricci curvature with $c=\frac{m}{4}$. The value of $b$ can be determined upon first calculating $\delta$ in (\ref{eq:hess}). Since $\Hess^\phi \geq \frac{1}{\sigma^2}g$, we select $\delta=\frac{1}{\sigma^2}$ resulting in an interval 
		$0<b < ((1+\sqrt{2})\sigma^2\sqrt{m-1})^{-1/2}$
		of possible values. Hence, Assumption~\ref{assump:ric} is satisfied. Note that the Bakry-\`Emery criterion (\ref{eq:bakry-emery}) is only satisfied for $\sigma<\sqrt{\frac{4}{m}}$.
		
		\item Let $\phi(X) = \rho(X,O)^4 - \rho(X,O)^2$ in \eqref{invmeasure} for fixed $O \in \P_m$. It will be shown in towards the end of this example that $\phi$ is geodesically non-convex as the Hessian is negative in some domain.
		To verify Assumption~\ref{assump:stochcomplete}, note that a lower bound on the Ricci curvature is available from (i). To compute the Hessian, consider first a general function $f \in C^2(\P_m)$. By the chain rule 
		$$ \Hess^{f(\rho^2)} = f''(\rho^2) \mathrm{d} \rho^2 \otimes \mathrm{d} \rho^2 + f'(\rho^2) \Hess^{\rho^2}. $$
		Then setting $f(x) = x^2$ results in
		$$ \Hess^{\rho^4} = 2\mathrm{d} \rho^2 \otimes \mathrm{d} \rho^2 + 2\rho^2 \Hess^{\rho^2},$$
		such that Hessian of $\phi$ is 
		$$ \Hess^\phi = D^2 \rho^4 - D^2 \rho^2 = 8\rho^2 \mathrm{d} \rho \otimes \mathrm{d} \rho + (2\rho^2-1)\Hess^{\rho^2} $$
		upon using the fact that $\mathrm{d} \rho^2 = 2\rho \mathrm{d} \rho$. Since $\Hess^{\rho^2}\geq 2g$ and $\mathrm{d} \rho\otimes \mathrm{d} \rho$ is a non-negative tensor:
		$$\Hess^\phi \geq 2(2\rho^2 - 1)\rf{g}.$$
		Thus, $\Ric+\Hess^\phi \geq \big(2(2\rho^2-1)-\frac{m}{4}\big)g$ and Assumption~\ref{assump:stochcomplete} is satisfied for $\kappa = \frac{m+8}{4}$. 
		Evidently, the Hessian is negative when $\rho <\frac{1}{\sqrt{2}}$, hence for the compact set we pick the closure of geodesic ball   $\mathcal B_o(1)$ of radius $1$ and cent\rf{e}red at $o$. Outside of this ball, the Hessian is strictly positive,  and choosing $\delta = 2$ ensures that (\ref{eq:hess}) in Assumption~\ref{assump:ric} is satisfied. This implies that choosing 
$b$ satisfying $0<b<(2(1+\sqrt{2})\sqrt{m-1})^{-1/2}$ takes care of Assumption \ref{assump:ric}. Non-convexity of $\phi$ can now be gleaned from the fact that inside of $\mathcal B_o(1/\sqrt{2})$ the Hessian is negative. To see this, we look at the Hessian contracted with vector fields $v$ such that $g(v,\nabla \rho) = 0$. Then
		$$ \Hess^\phi(v,v) = (2\rho^2-1)\Hess^{\rho^2}(v,v) \leq -2(1-2\rho^2)g(v,v) <0. $$
		
	\end{enumerate}
\end{example}

\subsection{Computational aspects}

When $M=\R^q$ it is known that Euler-type schemes (including explicit ones) converge when coefficients of the SDE are globally Lipschitz (implying that the coefficients grow no faster than linearly at infinity) \citep{milstein2004stochastic}. It is quite straightforward to extend the analysis presented in this paper in the case of compact manifolds to non-compact case under the globally Lipschitz assumption on $\nabla \phi$ and $g^{-1/2}$. This extension will require to prove uniform bounds for moments of $X^h_n$.

However, distributions $\phi$ and manifolds $(M,g)$ of practical interest are usually so that the globally Lipschitz assumption is violated. It is even so for Example~\ref{example:spd}(i) of the Riemannian-Gaussian distribution on $\P_m$. 

It is known for SDEs on $M=\R^q$ that when growth of coefficients at infinity is faster than linear, {\it explicit} Euler schemes can diverge \citep{MSH01,milstein2004stochastic,HJK11}. Reason for the divergence is exploding moments of the corresponding Markov chains (despite moments of the SDEs solution being bounded), typically due to a tiny number of exploding trajectories \citep{MT05,milstein2004stochastic}. 

By rejecting exploding trajectories, one can safely use any method of weak approximation for a broad class of SDEs with nonglobally Lipschitz coefficients to compute the ensemble averaging estimator $\hat \mu_{\phi,N}(\varphi)$ of $\mu_\phi(\varphi)$. SDEs from this class need to merely satisfy stochastic completeness and have sufficiently smooth coefficients, which SDEs of applicable interest typically do. The rejection technique relates to the second part of Assumption~\ref{assump:ric} (see (\ref{eq:hess})) in that one can view the condition as ignoring all sample paths of the diffusion that exit a sufficiently large geodesic ball $\mathcal B_o(R)$ for $R>0$ around a point $o \in M$. Theoretical justification of this when $M=\mathbb{R}^{q}$ case is given by \cite{MT05,MT07} (see also \cite{milstein2004stochastic}), and it is possible to transfer the proof to the setting of non-compact manifolds $M$. 

In Section~\ref{sec:numpd} we provide numerical evidence of how the rejection technique can profitably be used by choosing a large enough closed geodesic ball, as in Assumption \ref{assump:ric}, on $\P_m$ using the proposed Riemannian Langevin Monte Carlo algorithm. In particular, no trajectory was rejected for $R=2.7$ in the case of the Riemannian-Gaussian distribution (Section~\ref{sec:numpd1}) and $R=2$ for the non-convex potential (Section~\ref{sec:numpd2}). 

In $\R^q $, to address the difficulty with using the explicit Euler scheme for SDEs with non-globally Lipschitz coefficients in time-averaging estimators, one can use the Metropolis-adjusted Euler scheme \citep{BH13}.  The Metropolis adjustment also either removes the bias or makes it exponentially small, however it bears substantial additional computational cost. A Metropolis-adjusted version of the Riemannian Langevin algorithm treated in this paper requires a special separate consideration.

Alternatives to the explicit Euler, that result in bounded moments  \rf{and hence convergence}, have been studied in depth for SDEs in $\R^q$ with non-globally Lipschitz coefficients. Examples include implicit methods (and in particular, the implicit Euler scheme) and explicit methods of tamed and balanced  type \citep[e.g.,][]{MSH01,HJ15,TZ13,milstein2004stochastic}. Extensions of such methods to the non-compact manifold setting are possible and will be taken up elsewhere.

\section{Numerical Illustrations}\label{sec:exp}

We carry out numerical experiments to verify the theoretical results from Section~\ref{sec:bounds} for compact manifolds $M$, and demonstrate utility of the proposed algorithm for both compact and non-compact $M$. Section~\ref{sec:numMF} considers sampling using retractions from the von-Mises distribution on the two-dimensional unit sphere $\S^2$. \rf{A sampling algorithm using extrinsic embedding coordinates was proposed by \cite{wood1994simulation}, while \cite{fisher1993statistical} proposed an exact method using intrinsic coordinates for dimension $2$. A more recent development in \cite{kurz2015stochastic} provides an extrinsic algorithm without rejection that is exact for  even dimension spheres only.} 


Section~\ref{sec:numpd} considers two distributions on the non-compact manifold of symmetric positive definite matrices: the Riemannian-Gaussian distribution with convex $\phi$, and a distribution with non-convex potential. Versatility of the proposed algorithm to handle both convex and non-convex potential is demonstrated. 

We focus on the ensemble-averaging estimator $\hat \mu_{\phi,N}(\varphi)$ defined in \eqref{eq:est_ea}. To this end, we simulate a large number $L$ of \rf{independent} trajectories for a sufficiently long time \rf{$T=Nh$}, and then verify that it enjoys first order of convergence in terms of time step $h$ by comparing with the exact value of $\mu_\phi(\varphi)$ (obtained via numerical integration). Two types of errors are reported: 
the \rf{numerical integration} error
\[\texttt{err}=|\hat \mu_{\phi,N}(\varphi) - \mu_\phi(\varphi)|;\] 
and, the Monte Carlo error
\[
\texttt{MCerr}=\dfrac{1}{L-1}\bigg(\sum_{l=1}^L \varphi(X^{(l),h}_N)^2 - \dfrac{1}{L}\bigg(\sum_{l=1}^L \varphi(X^{(l),h}_N)\bigg)^2\bigg),
\]
with corresponding 95\% confidence interval 
\[
\hat \mu_{\phi,N}(\varphi) \pm 1.96 \sqrt{\dfrac{\texttt{MCerr}}{L}}. 
\]
We note that in a similar manner, the theoretical results (see Theorem~\ref{theorem4.2}) for the time-averaging estimator $\hat \varphi_{N}$ (\ref{eq:est_ta}) can also be verified.

Simulations were performed in \texttt{R} using a parallel architecture on 30 cores, using the packages \texttt{purrr} and \texttt{furrr} on a Supermicro 620U Linux RHEL8.8 server with 48 Intel Xeon (Ice Lake class) CPUs.

\subsection{Retraction-based sampling from the von-Mises Fisher distribution on $\mathbb S^2$}
\label{sec:numMF}
The unit sphere $\S^2$ is a compact manifold with constant positive sectional curvature $1$ everywhere, such that the cut locus of every point $x$ is its antipode $-x$. This means that a single chart cannot cover $\S^2$ in order to implement the proposed algorithm. Charts can be defined by embedding $\S^2$ into $\R^3$, and this is commonly done using the inclusion map as the embedding such that $\S^2=\{(x_1,x_2,x_3)^\intercal \in \R^3:x_1^2+x_2^2+x_3^2=1\}$, and two stereographic projections from the north and south poles onto $\R^2$ are used to define two charts which cover $\S^2$. 

We will instead consider spherical coordinates arising from the embedding of $\S^2$ into $\R^2$ given by $(r,\theta)^\intercal \mapsto (\sin r \cos \theta, \sin r \sin \theta, \cos r)^\intercal$, where $r$ is the geodesic distance to the north pole, with $r \in [0,\pi], \theta \in [0,2\pi]$. Our choice is motivated by the fact that expressions for the exponential and inverse-exponential maps are not available analytically in spherical coordinates with respect to the round metric $g$, and implementation of the proposed algorithm thus requires approximating the exponential map by numerically solving the geodesic equation. 
We thus consider a retraction of an appropriate order that preserves the order of errors derived in Section~\ref{sec:bounds} (see Section~\ref{sec:retractions}), and we verify this numerically.

The Riemannian metric for the spherical coordinates at a point $x=(r,\theta)^\intercal$ is
\begin{equation}\label{eq:gVMF}
	g=\diff r^2 + \sin^2r\, \diff \theta^2 ,
\end{equation}
which is obtained by pulling back the Euclidean metric from $\R^3$. The inverse metric tensor assumes the form
$$G^{-1} = 
\begin{pmatrix}
	1 & 0\\
	0 & \frac{1}{\sin^2 r}
\end{pmatrix},
$$
which is unbounded at $r=0$ or $r=\pi$ or equivalently, at $(0,0,1)^\intercal$ or $(0,0,-1)^\intercal$ in $\R^3$, and can lead to instability or loss of accuracy of the algorithm. The second chart is acquired by an orthogonal transformation of the coordinate axes obtained by permuting the axes corresponding to the embedding $(r,\theta)^\intercal \mapsto (\cos r, \sin r\cos \theta, \sin r \sin \theta)^\intercal$ with the metric unchanged. The transition map between the charts is then
\begin{equation}
	\label{eq:ch12}
	\psi(r,\theta) = (\arccos(\sin r \cos \theta),\mathrm{\arctan}(\cot r \csc \theta))^\intercal,
\end{equation}
which in $\R^3$ maps $(\sin r \cos \theta,\sin r \sin \theta, \cos r)^\intercal \mapsto (\cos r,\sin r \cos \theta,\sin r \sin \theta)^\intercal $; its inverse is 
\begin{equation}\label{eq:ch21}
	\psi^{-1}(\Tilde{r},\Tilde{\theta}) = (\arccos(\sin \Tilde{r} \sin \Tilde{\theta}),\mathrm{\arctan}(\tan \Tilde{r}\cos \Tilde{\theta}))^\intercal.
\end{equation}

Numerically, we restrict the domain of the embeddings to  $(\epsilon,\pi-\epsilon)\times[0,2\pi]$ for some $\epsilon>0$, whence the union of the two charts covers $\S^2$ as the singular points $r=0$ and $r=\pi$ in chart 1 are mapped to $(\pi/2,\pi/2)$ and $(\pi/2,3\pi/2)$, respectively, for any $\theta \in [0,2\pi]$. We choose $\epsilon=0.5$ in the simulations, which ensures that $G^{-1}$ is bounded.

The  cent\rf{e}red von-Mises Fisher distribution $\mu_\phi$ on $\S^2$ is the distribution
\begin{equation}\label{eq:vmf}
	\diff \mu_\phi \propto e^{\lambda \cos r} \dvol, \quad \mathrm{for}\; \lambda>0,\; r \in [0,\pi],
\end{equation}
so that $\phi((r,\theta)^\intercal)=\rf{-}\lambda \cos r$; performance of the algorithm for a non-cent\rf{e}red version can be examined with obvious modifications. It was verified in Example~\ref{exmp:vmf} that the Langevin diffusion \eqref{eq:diffusion} with $\phi((r,\theta)^\intercal) =- \lambda \cos r$ is ergodic, and converges exponentially fast to its invariant measure $\mu_\phi$.

The geodesic equation \eqref{eq:geodesic eq} can be written as the following system of ODEs in local coordinates starting at $(r_0,\theta_0)^\intercal$:
\begin{align}
	\dot{r}&=y, \,\, r(0)=r_0,\label{eq:goedS2}\\
	\dot{\theta}&=z, \,\,\ \theta(0)=\theta_0,\notag\\
	\dot{y}&=\sin r \cos r \thickspace z^2, \,\,\, y(0)=v^1,\notag\\
	\dot{z}&=-2\cot r \thickspace yz, \,\, z(0)=v^2. \notag
\end{align}
Since an exponential map in closed form is not available for this example, we use the retraction based on solving the above ODEs by a single step of size $\sqrt{h}$ of  the standard $(1/6,1/3,1/3,1/6)$ 4th-order Runge-Kutta method (RK4) as explained in Section~\ref{sec:retractions}.

\begin{algorithm}
	\small
	\caption{Sample from a von-Mises Fisher distribution on $\S^2$ using a retraction.  \label{algorithm2.1}}
	
	\begin{enumerate}
		
		\item Input $\lambda,\;h,\;N,\;X_0$ 
		\item  Initialise the chain at $X_0=(r_0,\theta_0)^\intercal$ according to the first chart and set chart counter $c=1, \epsilon=0.5$ and $n = 0$. 
		\item If $r_{n}<\epsilon$ or $r_{n}>\pi-\epsilon$: \\
		if $c=1$, change the chart according to (\ref{eq:ch12}) and set $c=2$; \\
		if $c=2$, change the chart according to (\ref{eq:ch21}) and set $c=1$.
		\item  Generate $\xi_{n+1}=(\xi^1_{n+1}, \xi^2_{n+1})^\intercal$, where $\xi^i_{n+1}$ are independent, from \eqref{eq:xi2}.
		
		\item Construct the tangent vector:\\
		if $c=1$, set $v_{n+1} = -\frac{\lambda \sqrt{h}}{2}\sin r_n \,(1,0)^\intercal + g^{-1/2}(X_n^h)\xi_{n+1}$; \\
		if $c=2$, set $v_{n+1} = \frac{\lambda \sqrt{h}}{2}(\cos r_n \sin \theta_n,\csc r_n \cos \theta_n)^\intercal + g^{-1/2}(X_n^h)\xi_{n+1}$.
		
		\item Approximate the geodesic equation (\ref{eq:goedS2}) with $(r_0,\theta_0)^\intercal=(r_n,\theta_n)^\intercal$, $(v^1,v^2)^\intercal =(v^1_{n+1},v^2_{n+1})^\intercal$ by a single step of size $\sqrt{h}$ of RK4 to obtain $X_{n+1}^h=(r_{n+1},\theta_{n+1})^\intercal$.
		\item \textbf{If} $n+1=N$ then \textbf{stop} and if $c=2$, change $X_N^h$ to first chart according to (\ref{eq:ch12}), \textbf{else} put $n := n+1$ and \textbf{return} to Step 3.
	\end{enumerate}
	
\end{algorithm}

\rf{For the testing purposes, we} take $\varphi(x)=\sin r$ and compute, for $\lambda=1$, the exact value $\mu_\phi(\varphi) = \sqrt{2} \frac{I_1(1)\Gamma(3/2)}{I_{1/2}(1)}\rf{\approx 0.7554024361}$, where  $I_n(\lambda)$ is the modified Bessel function of the first kind  \cite[Section~4.1]{lewis2023contributions}. The algorithm was terminated at time $T=5$, sufficient to ensure that the error $|\mathbb{E}(\varphi(X(T)))$ $-\mu_\phi(\varphi) |$ (i.e., the error due to proximity of the SDE's solution to the ergodic limit) is negligibly small. From Table~\ref{table:vmf} and the black line in Figure~\ref{plot:vmf}, we observe the first order convergence of Algorithm~\ref{algorithm2.1} as expected.

\begin{table}
	\caption{\small Estimation \rf{error $\texttt{err}$} and Monte Carlo error \rf{$\texttt{MCerr}$}  for sampling from a von-Mises Fisher distribution \rf{depending on the time step $h$ and number of independent realisations $L$}. The parameters are $\lambda = 1$, $X_0=(r_0,\theta_0)^\intercal=(\pi/4,\pi/4)^\intercal$ and $T = 5$.}
	\label{table:vmf} 
	\centering
	
	\begin{tabular}{c c c c c} 
		\hline\hline 
		$h$ & $L$ &  $\texttt{err}$ & \texttt{MCerr} \\ [0.5ex] 
		\hline 
		0.2 & $10^{6}$  & 0.0239 & 0.0004\\
		0.1 & $10^{6}$  & 0.0068 & 0.0004\\
		0.05 & $10^{6}$ & 0.0029 & 0.0005 \\
		0.025 & $10^{7}$ & 0.0014 & 0.0001  \\
		0.0125 & $10^7$ &  0.00065 & 0.00015   \\
		0.01 & $10^7$ &  0.00048 & 0.00015  \\ 
		\hline 
	\end{tabular}
	
\end{table}

\begin{figure}
	\centering
	\begin{tikzpicture}
		\begin{axis}[
			xlabel = $h$,
			ylabel = {\texttt{err}},
			log ticks with fixed point,
			xmode=log,
			ymode=log,
			xmin=0,
			xmax=0.3,
			scaled ticks=false,
			xticklabel style={
				/pgf/number format/precision=2,
				/pgf/number format/fixed,
				/pgf/number format/fixed zerofill,},
			yticklabel style={
				/pgf/number format/precision=2,
				/pgf/number format/fixed,
				/pgf/number format/fixed zerofill,},
			]
			\addplot+ [
			color=black,
			mark=o,
			error bars/.cd,y dir=both,y explicit,
			] coordinates {
				(0.2,0.02387048) +- (0,0.0004071002)
				(0.1,0.006840862) +- (0,0.0004481559)
				(0.05,0.002946046) +- (0,0.0004562242)
				(0.025,0.001485977)  +- (0,0.0001451195)
				(0.0125,0.0006492769)  +- (0,0.0001455306)
				(0.01,0.0004767473)  +- (0,0.0001456423)
			};
			\addplot[
			color=blue  ,
			mark=o,
			error bars/.cd,y dir=both,y explicit,
			] coordinates {
				(0.2,0.006798711) +- (0,0.0004593063)
				(0.1,0.003193959) +- (0,0.0004590083)
				(0.05,0.001251143) +- (0,0.0004601591)
				(0.025,0.0003581966)  +- (0,0.0001458583)
				(0.0125,0.000109982)  +- (0,0.00004614345)
			};
			
			\addplot[color=red,thick, samples=50,smooth,domain=0:6,dashed] coordinates {(0.01,0.01)(0.2,0.2)};
			
		\end{axis}
	\end{tikzpicture}
	\caption{\small Sampling from a von-Mises Fisher distribution. A log-log plot of the error $\texttt{err}$ against $h$. The black (blue) line corresponds to the Algorithm~\ref{algorithm2.1} with $\xi_n$ distributed according to \eqref{eq:xi} (standard Gaussian distribution). Error bars correspond to Monte Carlo error $\texttt{MCerr}$. The reference red line has gradient 1. The parameters are as in Table~\ref{table:vmf}.}
	\label{plot:vmf}
\end{figure}
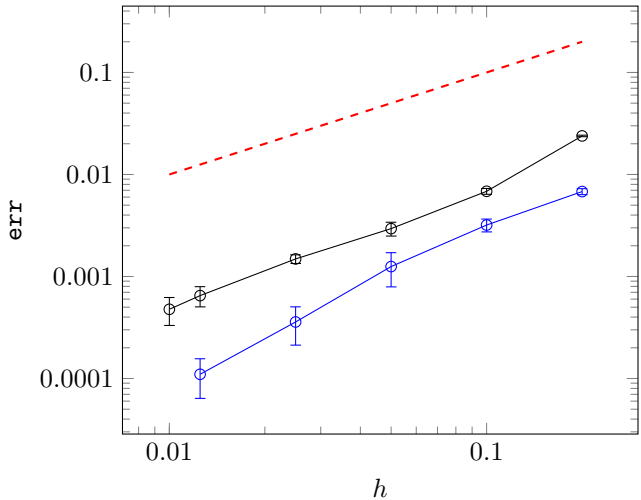

We also report results from an experiment with the modified Algorithm~\ref{algorithm2.1} in which the discrete random variables $\xi$  distributed according to \eqref{eq:xi2} are replaced with standard Gaussian random variables. The corresponding results given in Table~\ref{table:vmf gaussian} and the blue line in Figure~\ref{plot:vmf} show first order convergence of Algorithm~\ref{algorithm2.1} with Gaussian random variables. Interestingly, the error is smaller in the case of Gaussian random variables than when the discrete random variables are used.

\begin{table}
\caption{\small Sampling from a von-Mises Fisher distribution (Algorithm~\ref{algorithm2.1} with Gaussian $\xi$). The parameters are as in Table~\ref{table:vmf}.}
	\label{table:vmf gaussian}
	\centering
	
	\begin{tabular}{c c c c c} 
		\hline\hline 
		$h$ & $L$ & $\texttt{err}$ & \texttt{MCerr} \\ [0.5ex] 
		\hline 
		0.2 & $10^{6}$ & 0.0068 & 0.0005  \\
		0.1 & $10^{6}$ & 0.0032 & 0.0005  \\
		0.05 & $10^{6}$ & 0.0013 & 0.0005  \\
		0.025 & $10^{7}$ & 0.00036 & 0.00015  \\
		0.0125 & $10^8$ & 0.00011 & 0.00005   \\ 
		\hline 
	\end{tabular}

\end{table}

\rf{A further experiment was implemented in \texttt{Julia} using GPU kernels and run on an NVIDIA A100 GPU to compute the second order estimate of $\mu_\phi(\varphi)$ using the Talay-Tubaro extrapolation method detailed in Section~\ref{sec:TT}. For the initial first-order estimates, the experiment was repeated with the time steps $h=0.1,0.08,0.0625,0.05$ and with parameters $T=10$ and $L=5\times 10^9$ ($N_1 = 100,\,N_2=125,\,N_3=160,\,N_4=200 $). The respective first-order estimates of $\mu_{\phi}(\varphi)$ were $\hat\mu_{\phi,N_1}(\varphi) = 0.762017\pm0.000006 $ (\texttt{err} $ \doteq 0.006615
$), $\hat\mu_{\phi,N_2}=0.760524\pm 0.00006$ (\texttt{err} $ \doteq 0.005122$), $\hat\mu_{\phi,N_3}(\varphi) = 0.759317\pm0.000006 $ (\texttt{err}$ \doteq 0.003915$) and $\hat\mu_{\phi,N_4}(\varphi) = 0.758523\pm0.000006 $ (\texttt{err} $ \doteq 0.003121$). Then, using (\ref{Db236}) and the above data, the three second order estimates of $\mu_{\phi}(\varphi)$ are:
$\hat\mu_{\phi, imp}(\varphi) = 0.754552 \pm 0.000041$ (\texttt{err} $ \doteq 0.00085
$) based on $\hat\mu_{\phi,N_1}(\varphi)$ and $\hat\mu_{\phi,N_2}(\varphi)$; $\hat\mu_{\phi, imp}(\varphi) = 0.755006 \pm 0.000037$ (\texttt{err} $ \doteq 0.000396
$) based on $\hat\mu_{\phi,N_2}(\varphi)$ and $\hat\mu_{\phi,N_3}(\varphi)$; and $\hat\mu_{\phi, imp}(\varphi) = 0.755348 \pm 0.000041$ (\texttt{err} $ \doteq 0.000054
$) based on $\hat\mu_{\phi,N_3}(\varphi)$ and $\hat\mu_{\phi,N_4}(\varphi)$. 
We see that the Talay-Tubaro extrapolation improves the accuracy as expected. 
The combined computation time of this experiment was about 3000 seconds thanks to an effective use of GPU.   }

\subsection{Sampling on the manifold of SPD matrices}
\label{sec:numpd}
We consider the non-compact manifold of symmetric positive definite (SPD) matrices, and demonstrate numerically that the proposed Riemannian Langevin algorithm results in the desired error bounds when sampling from two distributions which satisfy the assumptions discussed in Section \ref{sec:non-comp}. In contrast to the unit sphere in Section \ref{sec:numMF}\rf{,} the algorithm is implemented using a closed form expression for the exponential map. 

Consider the general linear group $GL(m)$ consisting of the set of $m\times m$ real invertible matrices  equipped with the group operation of matrix multiplication. 
Denote by $M=\P_m$, the manifold of $m \times m$ symmetric positive definite matrices $\P_m=\{X \in GL(m) : X^\intercal = X ,\;y^\intercal X y>0, \; y \in \R^m\rf{\setminus \{0\}}  \}$, an open cone of dimension $q=m(m+1)/2$.

The tangent space $T_X \P_m = \{X\} \times \mathcal{S}_m$, where $\mathcal{S}_m$ is the set of $m \times m$ real symmetric matrices. Various choices of metric are available on $\P_m$ \citep{pennec2006riemannian}, and we choose the affine invariant metric 
\begin{equation}
	\label{eq:affineinvariant}
	g(U,V)_X = \Tr(X^{-1}UX^{-1}V), \quad U,V \in T_X \P_m,
\end{equation}
which is invariant under the transformation  $X \mapsto A^{-1} X A$, for every $A \in GL(m)$. The metric is compatible with the transitive action of $GL(m)$ on $\P_m$ in that between any pair $X_1, X_2 \in \P_m$ there exists an $A \in GL(m)$ such that $X_2=A^{-1}X_1A$, which makes $\P_m$ a homogeneous space of $GL(m)$.

Affine invariance of the metric ensures that the Riemannian distance $\rho$ and volume form $\dvol$ will also be invariant under affine transformations. The metric $g$ gives $\P_m$ a Riemannian structure with non-positive sectional curvature, and $\P_m$ is a non-compact, geodesically complete metric space by the Hopf-Rinow theorem \citep[e.g.,][]{jost2008riemannian}. Moreover, by the Cartan-Hadamard theorem \citep[e.g.,][]{kobayashi1969foundations}, the inverse exponential map is globally defined, and the cut locus of every point is empty, i.e., every two points in $\P_m$ have a unique geodesic connecting them. This means that only a single chart is needed to cover $\P_m$. 

The volume form at a point $X\in\P_m$ with affine-invariant metric $g$ is
$$ \dvol = \det(X)^{(m+1)/2}\diff x_1 \ldots \diff x_{m(m+1)/2},  $$
where $x_1,\ldots,x_{m(m+1)/2}$ are the upper triangular elements of $X$. For $X \in \P_m$ and $S \in \mathcal{S}_m$, the exponential map has the closed form 
\begin{equation}\label{eq:spd exp map}
	\exp_X(tS) = X^{1/2}\mathrm{Exp}( tX^{-1/2}SX^{-1/2})X^{1/2},
\end{equation}
where $\mathrm{Exp}$ is the matrix exponential. A geodesic that connects two points $X_1,X_2 \in \P_m$ can be parameterised as
$$ \gamma(t) = X_1^{1/2}\mathrm{Exp}( t \mathrm{Log}(X_1^{-1/2}X_2 X_1^{-1/2}))X_1^{1/2} $$
so that $\gamma(0) = X_1 $ and $\gamma(1) = X_2$, where $\mathrm{Log}$ is the matrix logarithm defined globally on $\P_m$. The distance between two points $X_1$ and $X_2$ is
$$ \rho(X_1,X_2) = \sqrt{\sum_{i=1}^\rf{m} (\log r_i)^2}, $$
where $\{r_i\}_{i=1}^m$ are the eigenvalues of $X_1^{-1}X_2$. The inverse exponential map can also be written in the closed form:
\begin{equation}\label{eq:spd inv exp}
	\Exp_{X_1}^{-1 }(X_2) = \dot{\gamma}(0) =  X_1^{1/2} \mathrm{Log}(X_1^{-1/2}X_2X_1^{-1/2})X_1^{1/2}. 
\end{equation}

For efficient use of the coordinates, we consider the vectorization and half-vectorization maps.  With $\vvec: \mathcal S_m \rightarrow \mathbb{R}^{m^2}$ as the vectorization map taking a symmetric matrix to a column vector, consider the half-vectorization map 
\[
\halfvec:\mathcal S_m \to \mathbb R^{q}, \quad \halfvec(x):=B \vvec(x), 
\]
where the matrix 
\[
B:=\sum_{i \geq j}(u_{ij}\otimes e_j^\intercal \otimes e^\intercal_i) \in \mathbb R^{q \times m^2}
\]
picks out the lower triangular part of the vectorization,
and $u_{ij}$ is a $q$-dimensional unit vector with 1 in position $(j-1)m+i-j(j-1)/2$ and 0 elsewhere, and $e_{i}$ is the standard basis of $\mathbb R^{m^2}$.  Its inverse 
\[
\mathbb R^{q}  \ni x \mapsto  \halfvec^{-1}(x):=(\halfvec(I_m)^\intercal \otimes I_m)(I_m \otimes x) \in \mathbb R^{m \times m}
\]
exists through the Moore-Penrose inverse of $A$. 

For the experiments, we fix $m=3$ and consider $\P_3$ with global coordinates $x=(x_1,x_2,x_3,x_4,x_5,x_6)$ such that
$$X = \left(
\begin{array}{ccc}
	x_1 & x_4 & x_6 \\
	x_4 & x_2 & x_5 \\
	x_6 & x_5 & x_3 \\
\end{array}
\right).$$
The metric tensor $G$ is expressed as $G=\diff x^\intercal(X^{-1}\otimes X^{-1})\diff x$, where $\diff x:=\halfvec (\diff X)$ with $\diff X=(\diff x_i)$. 
The inverse metric $G^{-1}$ can be concisely expressed \citep[e.g.,][]{moakher2011riemannian}:
$$ G^{-1}(X)=\left(
\begin{array}{cccccc}
	x_1^2 & x_4^2 & x_6^2 & x_1 x_4 &
	x_4 x_6 & x_1 x_6 \\
	x_4^2 & x_2^2 & x_5^2 & x_2 x_4 &
	x_2 x_5 & x_4 x_5 \\
	x_6^2 & x_5^2 & x_3^2 & x_5 x_6 &
	x_3 x_5 & x_3 x_6 \\
	x_1 x_4 & x_2 x_4 & x_5 x_6
	& \frac{1}{2} \left(x_1 x_2+x_4^2\right) &
	\frac{1}{2} (x_2 x_6+x_4 x_5) &
	\frac{1}{2} (x_1 x_5+x_4 x_6) \\
	x_4 x_6 & x_2 x_5 & x_3 x_5
	& \frac{1}{2} (x_2 x_6+x_4 x_5) &
	\frac{1}{2} \left(x_2 x_3+x_5^2\right) &
	\frac{1}{2}(x_3 x_4+x_5 x_6) \\
	x_1 x_6 & x_4 x_5 & x_3 x_6
	& \frac{1}{2} (x_1 x_5+x_4 x_6) &
	\frac{1}{2} (x_3 x_4+x_5 x_6) &
	\frac{1}{2} \left(x_1 x_3+x_6^2\right) \\
\end{array}
\right),$$
which must be written as a column vector in $\R^6$ using the $\halfvec^{-1}$ map for use in the algorithms.

\subsubsection{Riemannian-Gaussian distribution}\label{sec:numpd1}

The Riemannian-Gaussian distribution can be viewed as the generalisation of the isotropic normal distribution on $\R^q$ to $\P_m$, where the Euclidean distance is replaced by the Riemannian distance in the probability density function:
$$ 
\diff \mu_{\phi}(X) = \dfrac{1}{Z_m(\sigma)}e^{-\phi (X)}\dvol, \, X \in \P_m, $$ where
\begin{equation}\label{eq:rg phi}
	\phi (X) = \frac{1}{2\sigma^2}\rho(X,O)^2
\end{equation}
for a fixed $O \in \P_m$ and parameter $\sigma>0$. The function $x \mapsto \phi(x)$ is strictly convex, since, by the Hessian comparison theorem \citep{cheeger1975comparison} the map $x \mapsto \rho(x,\cdot)^2$ is geodesically convex on a manifold with everywhere non-positive sectional curvature. 

The normalisation constant $Z_m(\sigma)$ can be evaluated by employing spectral decomposition $X = Q^\intercal \mathrm{diag}(e^{\lambda_1},\ldots,e^{\lambda_m})Q$, with eigenvalues $\lambda_1\geq \cdots\geq \lambda_m>0$, and computing the following integral 
\begin{equation}\label{eq:rg normalisation}
	c_m \int_{\R^m} e^{- \frac{1}{2\sigma^2}[\lambda_1^2 + \cdots + \lambda_m^2]} \prod_{i<j}\sinh\bigg(\dfrac{|\lambda_i-\lambda_j|}{2}\bigg)\diff \lambda_1 \cdots \diff \lambda_m,
\end{equation}
where $c_m = \frac{1}{m!}\frac{\pi^{m^2/2}}{\Gamma_m(m/2)}8^{m(m-1)/4}$ and $\Gamma_m$ is the multivariate gamma function \citep{said2017riemannian}; the constant $c_m$ is part of the volume form $\dvol$, and when $m=3$, $c_3 = \frac{16 \sqrt{2}\pi^2}{3}$. Without loss of generality we set $O=I_m$, since $\P_m$ is a homogeneous space. For convenience, we use $\rho_I=\rho_I(X):=\rho(X,I_m)$.

\begin{table}[!htb]
\caption{\small Estimation and Monte Carlo errors for the Riemannian–-Gaussian distribution. The parameters are $O = I_3$, $\sigma=\frac{1}{\sqrt{2}}$, $X_0=\halfvec^{-1}((2,4,2,1,1,0)^\intercal)$ and $T=10$.}    \label{table:rg}   
	\centering
	\begin{tabular}{c c c c c} 
		\hline\hline 
		$h$ & $L$ & $\texttt{err}$ & \texttt{MCerr} \\ [0.5ex] 
		\hline 
		0.2 & $10^{6}$ & 0.148 & 0.008  \\
		0.1 & $10^{6}$  & 0.078 & 0.008 \\
		0.05 & $10^{7}$ & 0.035 & 0.002  \\
		0.025 & $10^7$  & 0.017 & 0.002  \\
		\hline 
	\end{tabular}
	 
\end{table}

To our knowledge, currently, only a rejection-based algorithm for sampling from Riemannian-Gaussian on $\P_m$, which relies on sampling the eigenvalues, is available in the literature \citep{said2017riemannian}. Our Riemannian Langevin algorithm given in Algorithm~\ref{algorithm2.3} does not contain a rejection step; however, a rejection step can in principle be incorporated if needed (for example, to be used as a proposal distribution within MCMC). 

The Riemannian-Gaussian distribution satisfies Assumptions \ref{assump:stochcomplete} and \ref{assump:ric}, 
as seen in Example \ref{example:spd} (i). This ensures that the Langevin diffusion  converges to its invariant measure exponentially fast. From \eqref{eq:spd inv exp} the gradient of $\phi$ is
$$\nabla\phi(X) = - \dfrac{\exp_X^{-1}(I)}{\sigma^2} = \dfrac{X^{1/2}\mathrm{Log}(X)X^{1/2}}{\sigma^2}, $$ which grows at infinity as $x\log x$; also note that $\nabla \phi(I) = 0$.

\begin{algorithm}
	\small 
	\caption{Algorithm to sample from Riemannian-Gaussian distribution on $\P_3$ with $O=I_3$, $\sigma=\frac{1}{\sqrt{2}}$ and $X_0=\halfvec^{-1}((2,4,2,1,1,0)^\intercal)$.\label{algorithm2.3}}

	\begin{enumerate}
		\item Input $h,\;N$
		
		\item Initialise the chain at $X_0$ and set  $k = 0$.
		
		\item Generate $\xi_k=(\xi^1_k,\xi^2_k,\xi^3_k,\xi^4_k,\xi^5_k,\xi^6_k)^\intercal$, where, for every $k$, $\xi^i_k$ are independent from \eqref{eq:xi2}.

		\item Compute the matrix $M_k = -\frac{h}{2\sigma^2}\mathrm{Log}(X_k^h) + \sqrt{h}(X_k^h)^{-1/2}\halfvec^{-1}(G^{-1/2}(X_k^h)\xi_k)(X_k^h)^{-1/2}$. 
		
		\item Set $X_{k+1}^h=(X_k^h)^{1/2}\mathrm{Exp}(M_k)(X_k^h)^{1/2}$. 
		
		\item \textbf{If} $k+1=N$ then \textbf{stop}, \textbf{else} put $k := k+1$ and \textbf{return} to Step 3.

	\end{enumerate}
\end{algorithm}
The function $\varphi: \P_3 \to \R$ we consider to assess accuracy is $ \varphi (X) = \mathrm{det}(X)$. The value of 
$\mu_{\phi}(\varphi)\approx 2.11699998$, evaluated with accuracy of order of $10^{-8}$ using the function \texttt{NIntegrate} in Mathematica. Table~\ref{table:rg} and Figure~\ref{plot:rg} demonstrate that Algorithm \ref{algorithm2.3} exhibits the first order convergence.

\begin{figure}
	\centering
	\begin{tikzpicture}
		\begin{axis}[
			xlabel = $h$,
			ylabel = {\texttt{err}},
			log ticks with fixed point,
			xmode=log,
			ymode=log,
			xmin=0,
			xmax=0.3,
			scaled ticks=false,
			xticklabel style={
				/pgf/number format/precision=2,
				/pgf/number format/fixed,
				/pgf/number format/fixed zerofill,},
			yticklabel style={
				/pgf/number format/precision=2,
				/pgf/number format/fixed,
				/pgf/number format/fixed zerofill,},
			]
			\addplot+ [
			color=black,
			mark=o,
			error bars/.cd,y dir=both,y explicit,
			] coordinates {
				(0.2,0.1475338) +- (0,0.008194115)
				(0.1,0.07838471) +- (0,0.007993021)
				(0.05,0.03491928)  +- (0,0.00248595)
				(0.025,0.01700289)  +- (0,0.002468861)
			};
			
			\addplot[color=red,thick, samples=50,smooth,domain=0:6,dashed] coordinates {(0.0125,0.0125)(0.2,0.2)};
			
		\end{axis}
	\end{tikzpicture}
	\caption{\small Sampling from the Riemannian–-Gaussian distribution on $\P_3$: log-log plot of the estimation error $\texttt{err}$ against $h$. Error bars correspond to Monte Carlo error \texttt{MCerr}. The reference red line has gradient 1. The parameters for the Riemannian-Gaussian are as in Table~\ref{table:rg}. }
	\label{plot:rg}
\end{figure}
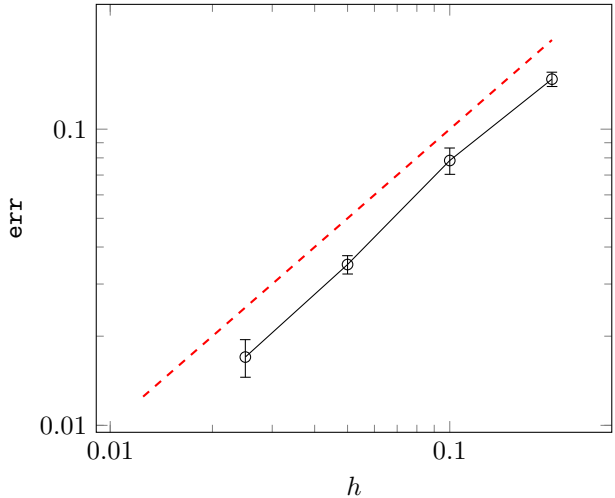

\subsubsection{Distribution with non-convex potential}\label{sec:numpd2}
The second part of Assumption~\ref{assump:ric} (see (\ref{eq:hess})) makes it evident that we need not have log-concavity of the density function to guarantee ergodicity of the SDE. We verify this and demonstrate generality of our algorithm by considering the following distribution with a non-convex potential $\phi$:
$$ \diff \mu_\phi = \dfrac{1}{Z_m} e^{-\phi(X)} \dvol,\quad X \in \P_m,$$
where 
$$\phi(X) = \rho(X,O)^4 - \rho(X,O)^2,$$
with the normalisation constant $Z_m$. This distribution is referred to as one with a double-well potential.
Similar to the Riemannian-Gaussian distribution, the constant $Z_m$ is obtained by evaluating
$$c_m \int_{\R^m} e^{-\left([\lambda_1^2 + \cdots + \lambda_m^2]^2- [\lambda_1^2 + \cdots + \lambda_m^2]\right)}\prod_{i<j}\sinh\bigg(\dfrac{|\lambda_i-\lambda_j|}{2}\bigg)\diff \lambda_1 \cdots \diff \lambda_m, $$
where $c_m$ is as in \eqref{eq:rg normalisation}. 

Example \ref{example:spd} (ii) verified that $\phi(X)$ satisfied Assumptions~\ref{assump:stochcomplete} and \ref{assump:ric}. 
The corresponding Langevin diffusion \eqref{eq:diffusion} converges exponentially fast to the invariant distribution $\mu_\phi$. The gradient takes the form
$$ \nabla\phi(X) = -4\rho(X,I_m)^2\exp_X^{-1}(I_m)+2\exp_X^{-1}(I_m) = (4\rho(X,I_m)^2-2)X^{1/2}\mathrm{Log}(X)X^{1/2}. $$

\begin{algorithm}
	\caption{Algorithm to sample from double-well potential distribution on $\P_3$ with $O=I_3$ and $X_0=\halfvec^{-1}((2,4,2,1,1,0)^\intercal)$.\label{alg:double-well}}
	\small 
	
	\begin{enumerate}
		\item Input $h,\;N$.

		\item Initialise the chain at $X_0$ and set  $k = 0$.
		
		\item Generate $\xi_k=(\xi^1_k,\xi^2_k,\xi^3_k,\xi^4_k,\xi^5_k,\xi^6_k)^\intercal$, where, for every $k$, $\xi^i_k$ are independent from \eqref{eq:xi2}.

		\item Compute the matrix 
		\begin{equation*}
			M_k = h(1-2\rho(X_k^h,I)^2)\mathrm{Log}(X_k^h) + \sqrt{h}(X_k^h)^{-1/2}\halfvec^{-1}(G^{-1/2}(X_k^h)\xi_k)(X_k^h)^{-1/2}. 
		\end{equation*}
		\item Set $X_{k+1}^h=(X_k^h)^{1/2}\mathrm{Exp}(M_k)(X_k^h)^{1/2}$. 
		
		\item \textbf{If} $k+1=N$ then \textbf{stop}, \textbf{else} put $k := k+1$ and \textbf{return} to Step 3.
		
	\end{enumerate}
\end{algorithm}

We choose $ \varphi (X) = 1/(1+\mathrm{tr}(X))$ and $\mu_\phi(\varphi) \approx 0.2204801571878534$, evaluated with accuracy of order of $10^{-8}$ using the function \texttt{NIntegrate} in Mathematica. 
Table~\ref{table:double-well} and Figure~\ref{plot:double-well} show that Algorithm \ref{alg:double-well} exhibits the first order convergence.

\begin{table}
	\caption{The double-well potential distribution. The parameters are $O = I_m$,
		$X_0=\halfvec^{-1}((2,4,2,1,1,0)^\intercal)$ and $T=5$.} \label{table:double-well} 
	\centering
	
	\begin{tabular}{c c c c c} 
		\hline\hline 
		$h$ & $L$ & $\texttt{err}$ & \texttt{MCerr} \\ [0.5ex] 
		\hline 
		0.2 & $10^{5}$ & 0.00537 & 0.0003  \\
		0.1 & $10^{5}$  & 0.00162 & 0.0003 \\
		0.05 & $10^{7}$ & 0.000605 & 0.00003  \\
		0.025 & $10^7$  & 0.000258 & 0.00003  \\
		\hline 
	\end{tabular}
	
\end{table}

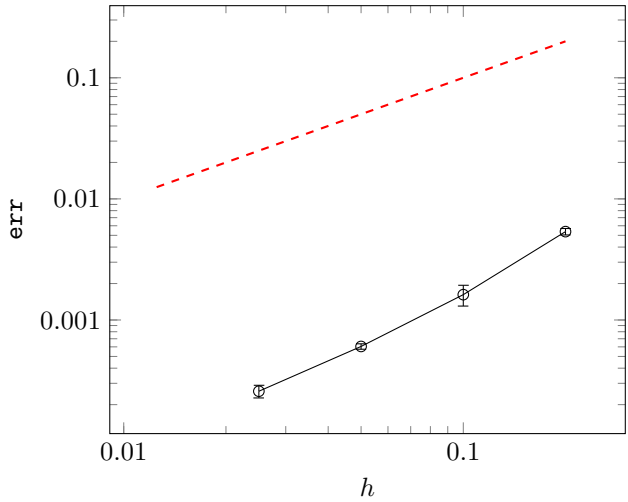
\begin{figure}
	\centering
	\begin{tikzpicture}
		\begin{axis}[
			xlabel = $h$,
			ylabel = {\texttt{err}},
			log ticks with fixed point,
			xmode=log,
			ymode=log,
			xmin=0,
			xmax=0.3,
			scaled ticks=false,
			xticklabel style={
				/pgf/number format/precision=2,
				/pgf/number format/fixed,
				/pgf/number format/fixed zerofill,},
			yticklabel style={
				/pgf/number format/precision=2,
				/pgf/number format/fixed,
				/pgf/number format/fixed zerofill,},
			]
			\addplot+ [
			color=black,
			mark=o,
			error bars/.cd,y dir=both,y explicit,
			] coordinates {
				(0.2,0.005366134) +- (0,0.0003292115)
				(0.1,0.001619112) +- (0,0.0003160915)
				(0.05,0.0006047)  +- (0,0.00003108402)
				(0.025,0.000258065)  +- (0,0.00003090649)
			};
			
			\addplot[color=red,thick, samples=50,smooth,domain=0:6,dashed] coordinates {(0.0125,0.0125)(0.2,0.2)};
			
		\end{axis}
	\end{tikzpicture}
	\caption{\small Sampling from a distribution with non-convex `double-well' potential on $\P_3$: log-log plot of the estimation error $\texttt{err}$ against $h$. Error bars correspond to Monte Carlo error \texttt{MCerr}. The reference red line has gradient 1. The parameters are as in Table~\ref{table:double-well}.}
	\label{plot:double-well}
\end{figure}

\section{Discussion and Concluding Remarks}

In this paper we have studied how the intrinsic Langevin diffusion can be used for sampling on Riemannian manifolds. To this end we have analyzed and provided error estimates for intrinsic Riemannian Langevin algorithms including those using retractions when the exponential map is not available in closed form. The algorithm ensures that the corresponding Markov chain moves on the manifold $M$ without requiring a projection, i.e. it preserves the geometrical features of the corresponding stochastic differential equation. 

The main focus is on sampling from, and computing expectation of functions of, a target distribution. As a consequence, we demonstrate how the concept of weak approximations for SDEs in $\R^q$ and the flat torus $\mathbb T^q$, and the corresponding proof techniques in deriving weak error bounds, can be transferred to the case of compact Riemannian manifolds via the covariant Taylor expansion. The PDE-based proof technique enables obtaining weak error bounds that match the optimal bounds in $\mathbb R^q$ and flat manifolds; these are superior to the weak error bound currently available in the literature, which follow by first deriving strong approximation (e.g., convergence in expected squared-distance) error bounds. The price we pay for our approach is in the absence of explicit dimension-dependent \rf{constants} 
in the derived bounds. The Talay-Tubaro extrapolation in Section \ref{sec:TT} may be used to extract more information on the \rf{constants.} 

The approach used here to deriving weak error bounds opens the door to carry over the extensive research (conducted by statistical, machine learning, and the numerical analysis communities) on approximation of SDEs in Euclidean spaces to manifolds. For example, sampling from a distribution with compact support on a non-compact manifold, or from a manifold with boundary, are of considerable interest given their practical relevance. These problems are yet to be considered within current literature; only recently has the former been considered even for $\R^q$ using a reflected diffusion \citep{leimkuhler2023simplerandom}. It is plausible that an intrinsic reflected diffusion can be used to carry out a similar program on manifolds, and this is the subject of our future work. 

In this paper we have only considered explicit Euler-type schemes. In a similar fashion, we can construct and analyze implicit schemes, second-order sampling algorithms, etc. for Langevin diffusion on manifolds using the arsenal of numerical methods developed in the Euclidean space (see e.g. \citep{milstein2004stochastic}), which potentially can lead to more efficient sampling methods.

The proposed algorithm and its analysis are also useful for applications in other research areas, including molecular dynamics, finance, and optimisation.  An important example in financial engineering involves multi-factor stochastic volatility models widely used in option pricing and risk analysis, where stochastic modelling of correlation/covariance between factors entering the asset price and volatility processes are of practical value. There is growing interest \citep{fin1,fin2} in using a Wishart process \citep{Bru91} to model the time-changing covariance structure, which is an example of stochastic process assuming values in the manifold of positive (semi-)definite matrices. 
The examples considered in Section~\ref{sec:numpd} with the intrinsic approach to modelling stochastic dynamics of  positive definite matrices provides a useful, more general class of processes to be used within this context. 

\section*{Acknowledgements}

The authors were supported by EPSRC grant no. EP/X022617/1. KB acknowledges support from grants EPSRC EP/V048104/1, NSF 2015374 and NIH R37-CA21495. AS acknowledges support by the Wallenberg AI, Autonomous Systems and Software Program (WASP) funded by the Knut and Alice Wallenberg Foundation.

\appendix
	
	\section{}
	
	
The covariant Taylor expansion plays a major role in the proofs of theorems, and we review it here. Let $ f \in C^{4}(M)$ taking values in $\mathbb{R}$ and  the curve $c(s)$, $s \in \mathbb{R}$ on $M$, be so that $c(0) = x$ and $\dot c(0) = V $. Introduce $F(s) := f(c(s)) : \mathbb{R} \rightarrow \mathbb{R}$ for which the Taylor expansion can be written as
	\begin{align}\label{eq:taylor series}
		F(s) = F(0) + s\frac{\diff}{\diff s}F(0) + \frac{s^{2}}{2}\frac{\diff^{2}}{\diff s^{2}}F(0) + \frac{s^{3}}{6}\frac{\diff^{3}}{\diff s^{3}}F(0) + \frac{s^{4}}{24}\frac{\diff^{4}}{\diff s^{4}}F(\alpha),
	\end{align}
	where $\alpha \in (0,s)$. We examine each derivative in turn. The first derivative is
 \begin{equation}\label{eq:first derivative}
     \dfrac{\diff}{\diff s}F(s) =  \dfrac{\diff}{\diff s}f(c(s)) = \diff f(c(s))(\dot c(s)).
 \end{equation}
For the second derivative, the geodesic curvature of $c$ is accounted for
\begin{align*}\label{eq:second derivative}
    \dfrac{\diff^2}{\diff s^2}F(s)  &= \dfrac{\diff^2}{\diff s^2}f(c(s)) \\
                                    &= \dfrac{\diff}{\diff s} \diff f(c(s))(\dot c(s))\\
                                    &= D\diff f(c(s))(\dot c(s),\dot c(s)) + \diff f(c(s))(D_{\dot c(s)}\dot c(s)).\numberthis
\end{align*}
Computing the third derivative, we get
\begin{align*}\label{eq:third derivative}
    \dfrac{\diff^3}{\diff s^3}F(s) &= \dfrac{\diff}{\diff s}\bigg(D\diff f(c(s))(\dot c(s),\dot                                     c(s)) + \diff f(c(s))(D_{\dot c(s)}\dot c(s))\bigg)\\
                                   &= D^2\diff f(c(s)) (\dot c(s),\dot c(s),\dot c(s)) + 2D\diff f(c(s))(\dot c(s),D_{\dot c(s)}\dot c(s)) \\
                                   & \;\;\; + D\diff f(c(s))(\dot c(s),D_{\dot c(s)}\dot c(s)) + \diff f(c(s))(D_{\dot c(s)}D_{\dot c(s)}\dot c(s))\\
                                   &=D^2\diff f(c(s)) (\dot c(s),\dot c(s),\dot c(s)) + 3D\diff f(c(s))(\dot c(s),D_{\dot c(s)}\dot c(s)) \\ &  \;\;\;+ \diff f(c(s))(D_{\dot c(s)}D_{\dot c(s)}\dot c(s)),\numberthis
\end{align*}
where we have utilized the fact that second covariant derivative of functions is symmetric $(0,2)-$tensor. 
 Evaluating \eqref{eq:first derivative}, \eqref{eq:second derivative}, \eqref{eq:third derivative} at $s=0$ and substituting them into \eqref{eq:taylor series}, we obtain the following Taylor formula of $f$ along the curve $c$:
\begin{align*}\label{eq:full taylor}
    f(c(s)) &= f(x) + s \diff f(x)(V) + \dfrac{s^2}{2}(D\diff f(x)(V,V) + \diff f(x)(D_V V)) \\
            &\;\;\; + \dfrac{s^3}{6}(D^3 f(x)(V,V,V) + 3D^2f(x)(V,D_V V) + \diff f(x)(D^2_V V))\\
            &\;\;\;+ \dfrac{s^4}{24}\rf{\dfrac{\diff^4}{\diff s^4}f(c(\alpha))},\numberthis
\end{align*}
where we have abbreviated $D^2_V = D_VD_V$.

	\section{
 Proofs of Theorems}
	\label{proof1}

	For brevity, we denote $u(t,x)$ as $u$, $\frac{\partial u}{\partial t}(t,x)$ as $ \frac{\partial u}{\partial t}$, $ \nabla u(t,x) $ as $\nabla u$, $ Du(t,x) $ as $Du$, $ D^ku(t,x) $ as $D^ku$, $\nabla \phi(x)$ as $\nabla \phi$, $g^{-1/2}(x)$ as $g^{-1/2}$ and $u(t_{n}, X_{n}^h)$ as $ u_{n} $, $n =0,\dots,N-1$, where $u(t,x)$ is the solution of \eqref{BKPDE} and $\phi(x)$ is from (\ref{eq:diffusion}). We need the following lemma on one-step error of the algorithm (\ref{eqn_3.5})  to prove Theorem~\ref{theorem4.1}.
	\begin{lemma}[One-step error]\label{onesteplemma}
		Let Assumptions~\ref{assump:g} and \ref{assump:phi} and $\varphi \in C^{4, \epsilon}(M)$ hold. Given $x \in M $, let $X_{1}^h$ be computed according to the following formula:  
		\begin{align}
			X_{1}^h = F_{x}\Big(-\frac{h}{2}\nabla \phi(x) + h^{1/2}g^{-1/2}(x)\xi \Big),
		\end{align}
  where $F_{x}$ is a retraction, and curve $c(s)$  corresponding to this retraction  satisfies Assumption~\ref{assump:2nd order ret} 
with $c(0) = x$, $\dot c(0)=V := -\frac{h^{1/2}}{2}\nabla\phi + g^{-1/2} \xi$ and hence  $F_x(V) = c(\sqrt{h})$.
		Then
		\begin{align}\label{eq:onesteplemma}
			\mathbb{E}[u(t+ h, X_{1}^h) - u(t,x)] \leq C h^{2}e^{-\lambda (T-t)},
		\end{align}
		where $C>0 $ is independent of $T$ and $h$ and it linearly depends on $\Vert \varphi \Vert_{C^{4,\epsilon }(M)}$ but otherwise is independent of $\varphi$. 
	\end{lemma}
	
	\begin{proof}		
		Applying the Taylor expansion to $u(t+h,c(s))$ around $t$, we obtain
		\begin{align}\label{gla_eqn5.5}
			u(t + h, c(s)) = u(t, c(s)) + h \frac{\partial u}{\partial t}(t,c(s)) + \frac{h^{2}}{2}\frac{\partial^{2} u}{\partial t^{2}}(t(\alpha_{1}), c(s)),
		\end{align}
		where $t(\alpha_{1}):= t + \alpha_{1} h \in (t, t + h)$ since $\alpha_{1} \in (0,1)$. 
  
  
		Applying the covariant Taylor formula \eqref{eq:full taylor} to the function $ u(t, c(s))$ along the curve $c$, we get 
		\begin{align*}\label{eqn_taylor4}
			u(t, c(s)) &= u + s \diff u(V) + \dfrac{s^2}{2}(D\diff u(V,V) + \diff u(D_V V)) \\
            &+ \dfrac{s^3}{6}(D^3 u(V,V,V) + 3D^2u(V,D_V V) + \diff u(D^2_V V))\\
            &+ \dfrac{s^4}{24}\rf{\dfrac{\diff^4}{\diff s^4} u(t,c(\alpha_2))},\numberthis
		\end{align*}
		where $\alpha_{2} \in (0, \sqrt{h})$. Similarly we obtain
		\begin{align}\label{eqn_taylor5}
			\frac{\partial u}{\partial t}(t, c(s)) &= \frac{\partial u}{\partial t} + sD\frac{\partial u}{\partial t}(V) \nonumber \\ & \;\;\; + \frac{s^{2}}{2}\bigg(D^{2}\frac{\partial u}{\partial t}(t,c(\alpha_{3}))(\dot c(\alpha_{3}),\dot c(\alpha_{3}))+ \diff \dfrac{\partial u}{\partial t}(t,c(\alpha_3))(D_{\dot c(\alpha_3)}\dot c(\alpha_3))\bigg),
		\end{align}
		where $\alpha_{3} \in  (0, \sqrt{h})$.
		
		Substituting (\ref{eqn_taylor4}) and (\ref{eqn_taylor5}) in (\ref{gla_eqn5.5}), we arrive at
		\begin{align*}
			u(t &+ h, c(s)) = u +  h\frac{\partial u}{\partial t}  + s\diff u(V) + \frac{s^{2}}{2}(D^{2}u(V,V)  + \diff u (D_V V)) \nonumber \\ & + \frac{s^{3}}{6}(D^{3}u(V,V,V)+ 3D^2u(V,D_V V) + \diff u(D^2_V V) )\\
   &+ \rf{\dfrac{s^4}{24}\dfrac{\diff^4}{\diff s^4} u(t,c(\alpha_2))} \\ & + h s D\frac{\partial u}{\partial t}(V) + h\frac{s^{2}}{2}\bigg(D^{2}\frac{\partial u}{\partial t}(t,c(\alpha_{3}))(\dot c(\alpha_{3}),\dot c(\alpha_{3})) + \diff \dfrac{\partial u}{\partial t}(t,c(\alpha_3))(D_{\dot c(\alpha_3)}\dot c(\alpha_3))\bigg)\\
   &+ \frac{h^{2}}{2}\frac{\partial^{2} u}{\partial t^{2}}(t(\alpha_{1}), c(s)). \numberthis
		\end{align*}
  
Taking expectation on both sides and putting $ s = \sqrt{h} $, we get
		\begin{align*}
			\E[u&(t + h, X_1^h)] = u + h\dfrac{\partial u}{\partial t} +\E\bigg[ \sqrt{h}\diff u(V) + \frac{h}{2}(D^{2}u(V,V)  + \diff u (D_V V)) \nonumber \\ & + \frac{h^{3/2}}{6}(D^{3}u(V,V,V)+ 3D^2u(V,D_V V) + \diff u(D^2_V V) )+ h^{3/2} D\frac{\partial u}{\partial t}(V) \\
   &+ \rf{\dfrac{h^2}{24}\dfrac{\diff^4}{\diff s^4} u(t,c(\alpha_2))} \\ & + \frac{h^2}{2}\bigg(D^{2}\frac{\partial u}{\partial t}(t,c(\alpha_{3}))(\dot c(\alpha_{3}),\dot c(\alpha_{3})) + \diff \dfrac{\partial u}{\partial t}(t,c(\alpha_3))(D_{\dot c(\alpha_3)}\dot c(\alpha_3))\bigg)\\
   &+ \frac{h^{2}}{2}\frac{\partial^{2} u}{\partial t^{2}}(t(\alpha_{1}), c(s))\bigg].\numberthis \label{eqntaylor5.9}
		\end{align*}
		Denote $i$-th component of $V$ as $V^{i}$, i.e. $V^{i} = \Big(-\frac{h^{1/2}}{2}\nabla \phi + g^{-1/2}\xi\Big)^{i}$. Using the properties of the random variables  $\xi^{i}$
		\begin{align}
			\E[\xi^{i}]=0,\; \E[\xi^{i}\xi^{j}]=\delta_{ij},\; \E[\xi^{i}\xi^{j}\xi^{k}]=0,\; \E[(\xi^{i})^{2}(\xi^{j})^{2}] = 1 \text{ for } i \neq j, \label{eq:propxi}
		\end{align} 
		we deduce that
		\begin{align*}
			\mathbb{E}(V^{i}V^{j}V^{k}) &= -\frac{h^{3/2}}{8}(\nabla \phi)^{i}(\nabla \phi)^{j} (\nabla \phi)^{k} + \frac{3h}{4}(\nabla \phi)^{i}(\nabla \phi)^{j}\mathbb{E}\big((g^{-1/2}\xi)^{k}\big) \\ & 
			\;\;\; - \frac{3h^{1/2}}{2}(\nabla \phi)^{i}\mathbb{E}\big((g^{-1/2} \xi)^{j}(g^{-1/2}\xi)^{k}\big) + \mathbb{E}\big((g^{-1/2}\xi)^{i}(g^{-1/2}\xi)^{j}(g^{-1/2}\xi)^{k}\big)\\ & 
			= -\frac{h^{3/2}}{8}(\nabla \phi)^{i}(\nabla \phi)^{j} (\nabla \phi)^{k} - \frac{3h^{1/2}}{2}(\nabla \phi)^{i}\mathbb{E}\big((g^{1/2})^{jl}\xi_{l}(g^{1/2})^{km}\xi_{m}\big)\\ & = 
			-\frac{h^{3/2}}{8}(\nabla \phi)^{i}(\nabla \phi)^{j} (\nabla \phi)^{k} - \frac{3h^{1/2}}{2}(\nabla \phi)^{i}(g^{1/2})^{jl}(g^{1/2})^{km}\delta_{lm} \\ &  = -\frac{h^{3/2}}{8}(\nabla \phi)^{i}(\nabla \phi)^{j} (\nabla \phi)^{k} - \frac{3h^{1/2}}{2}(\nabla \phi)^{i}g^{jk},  \numberthis \label{eqntaylor5.10}
		\end{align*}
		where $i,j,k = 1,\dots,q$ and $\delta_{ij}$ is the Kronecker delta. 
		Further, we have
		\begin{align*}
			\mathbb{E}(h^{1/2}Du(V)) &= -\frac{h}{2} g(\nabla \phi, \nabla u),\numberthis \label{eqntaylor5.11}\\
			\mathbb{E}(hD^{2}u(V,V)) &= h\mathbb{E}(\Hess^{u}(g^{-1/2}\xi, g^{-1/2}\xi)) + \frac{h^{2}}{4}D^{2}u(\nabla \phi, \nabla \phi) \\ &= h \mathbb{E}\big(\Hess^{u}_{ij}(g^{-1/2}\xi)^{i} (g^{-1/2}\xi)^{j}\big) + \frac{h^{2}}{4}D^{2}u(\nabla \phi, \nabla \phi)\\ & = h \mathbb{E}\big(\Hess^{u}_{ij}(g^{1/2})^{ik}\xi_{k} (g^{1/2})^{jl}\xi_{l}\big) + \frac{h^{2}}{4}D^{2}u(\nabla \phi, \nabla \phi) \\ & = 
			h \Hess^{u}_{ij}(g^{1/2})^{ik} (g^{1/2})^{jl}\delta_{kl} + \frac{h^{2}}{4}D^{2}u(\nabla \phi, \nabla \phi) \\ & = 
			h \Hess^{u}_{ij}g^{ij}  + \frac{h^{2}}{4}D^{2}u(\nabla \phi, \nabla \phi)\\ & = h \Delta_{M}u  + \frac{h^{2}}{4}D^{2}u(\nabla \phi, \nabla \phi),\numberthis \label{eqntaylor5.12} 
   \end{align*}
   \begin{align*}
			\mathbb{E}\Big(h^{3/2}D^{3}u(V,V,V)\Big) &= h^{3/2}D^{3}u(\partial_{i}, \partial_{j}, \partial_{k})\mathbb{E}(V^{i}V^{j}V^{k}) \\ & = -\frac{h^{3}}{8}D^{3}u(\partial_{i},\partial_{j}, \partial_{k})(\nabla \phi)^{i}(\nabla \phi)^{j} (\nabla \phi)^{k} \\ & \;\;\;\; - \frac{3h^{2}}{2}D^{3}u(\partial_{i},\partial_{j},\partial_{k})(\nabla \phi)^{i}g^{jk}, \numberthis \label{eqntaylor5.13}
			\\ 
			\mathbb{E}\Big(h^{3/2}D\frac{\partial u}{\partial t}(V)\Big) &= h^{3/2}D\frac{\partial u}{\partial t}(\partial_{i})\mathbb{E}(V^{i}) = \frac{h^{2}}{2}D\frac{\partial u}{\partial t}(\partial_{i})   (\nabla \phi)^{i}. \numberthis \label{eqntaylor5.14}
		\end{align*}
		Substituting (\ref{eqntaylor5.11})-(\ref{eqntaylor5.14}) in (\ref{eqntaylor5.9}), we get
		\begin{align*}
			\mathbb{E}(u(t + h, X_{1}^h)) &= u +  h\frac{\partial u}{\partial t}  - \frac{h}{2} g(\nabla \phi, \nabla u) + \frac{h}{2} \Delta_{M}u   \\ 
            & \;\;\;\; + \frac{h^{2}}{8}D^{2}u(\nabla \phi, \nabla \phi) -\frac{h^{3}}{48}D^{3}u(\partial_{i},\partial_{j}, \partial_{k})(\nabla \phi)^{i}(\nabla \phi)^{j} (\nabla \phi)^{k} \\ 
            & \;\;\;\; - \frac{h^{2}}{4}D^{3}u(\partial_{i},\partial_{j},\partial_{k})(\nabla \phi)^{i}g^{jk} + \frac{h^{2}}{2}D\frac{\partial u}{\partial t}(\partial_{i})   (\nabla \phi)^{i}\\ 
            & \;\;\;\;   + \mathbb{E}\bigg(\rf{\dfrac{h^2}{24}\dfrac{\diff^4}{\diff s^4} u(t,c(\alpha_2))}  \\ 
            & \;\;\;\;  + \frac{h^{2}}{2}D^{2}\frac{\partial u}{\partial t}(t, c(\alpha_{3}))(\dot c(\alpha_{3}),\dot c(\alpha_{3}))  + \frac{h^{2}}{2}\frac{\partial^{2} u}{\partial t^{2}}(t(\alpha_{1}), c(s))\\
            &\;\;\;\;+\diff \dfrac{\partial u}{\partial t}(t,c(\alpha_3))(D_{\dot c(\alpha_3)}\dot c(\alpha_3))\bigg) +  \frac{h}{2}\mathbb{E}(\diff u (D_V V)) \\ & \;\;\;\; + \frac{h^{3/2}}{6}\mathbb{E}(\diff u (D_V^{2}V)) + \frac{h^{3/2}}{2}\E( D^2u(V,D_V V)). \numberthis  \label{eqntaylor5.15}
		\end{align*}
		We note that the remainder terms include a linear combination of derivatives of $u$ which can be estimated using (\ref{eq:BKdecay}) together with boundedness of $\nabla \phi$ thanks to compactness of $M$ and Assumptions~\ref{assump:g} and~\ref{assump:phi}. Due to  Assumption~\ref{assump:2nd order ret} and again using Assumptions~~\ref{assump:g} and~\ref{assump:phi}, the last \rf{three} terms are also bounded by $Ce^{-\lambda (T-t)}h^{2}$, where $C$ and $\lambda $ are independent of $h$ and $T$. Consequently, we arrive at (\ref{eq:onesteplemma}).
	\end{proof}
We prove the above lemma under the conditions that retraction satisfies Assumption~\ref{assump:2nd order ret}. If we take $c(s)$ to be a geodesic curve then it is clear that $D_{\dot{c}(s)}\dot{c}(s)|_{s = 0} = 0$ and $D_{\dot{c}(s)}D_{\dot{c}(s)}\dot{c}(s)|_{s = 0} = 0$. 

	We now provide proof of  Theorem~\ref{theorem4.1}.  
	\subsection{Proof of Theorem~\ref{theorem4.1}}
	Note that using \eqref{BKPDE} and \eqref{feynmankazformula}, we can write
		\begin{align*}
			|\mathbb{E}(\varphi(X_{N}^h)) - \mu_\phi(\varphi) | &\leq |\mathbb{E}(\varphi(X_{N}^h)) - \mathbb{E}(\varphi(X(T)))| + | \mathbb{E}(\varphi(X(T))) - \mu_\phi(\varphi)| \\  & =  |\mathbb{E}(u(t_{N}, X_{N}^h)) -  u(0,x_{0})| + | \mathbb{E}(\varphi(X(T))) - \mu_\phi(\varphi)| \\ & = \bigg|\mathbb{E}\bigg(\sum\limits_{n=0}^{N-1}\mathbb{E}\Big(u_{n+1} - u_{n}\;\big|\; X_{n}^h \Big)\bigg)\bigg| + | \mathbb{E}(\varphi(X(T))) - \mu_\phi(\varphi)|.
		\end{align*}
		Using Lemma~\ref{onesteplemma} and (\ref{PA34}), we get
		\begin{align*}
			|\mathbb{E}(\varphi(X_{N}^h)) - \mu_\phi(\varphi) | \leq Ch^2\sum\limits_{n=0}^{N-1}e^{-\lambda(T-t_{n})} + Ce^{-\lambda T} \leq C(h + e^{-\lambda T}),
		\end{align*}
		where $C>0$ is a constant independent of $h$ and $T$ and it linearly depends on $\Vert \varphi \Vert_{C^{4,\epsilon }(M)}$ but otherwise $C, \lambda$ are independent of $\varphi$.  
	$\blacksquare$
	\subsection{Proof of Theorem~\ref{theorem4.2}}\label{proof2}
Let us denote $u_{n} := u(X_{n}^h)$, $\nabla \phi_{n} := \nabla \phi(X_{n}^h)$, $ g^{-1/2}_{n} = g^{-1/2}(X_{n}^h)$ and $V_{n} = -\frac{h^{1/2}}{2}\nabla\phi_{n} + g^{-1/2}_{n} \xi_{n+1} $, where $u(x)$ is the solution of (\ref{PPDE}). Define $\gamma_{V_{n}}$ to be the geodesic with initial conditions 
		\begin{align}
			\gamma_{V_{n}}(0) &= X_{n}^h,    \\
			\dot \gamma_{V_{n}}(0) &= V_{n} := -\frac{h^{1/2}}{2}\nabla\phi_{n} + g^{-1/2}_{n} \xi_{n+1}.
		\end{align}
		Applying the covariant Taylor expansion (\ref{eq:full taylor}) with $ s = \sqrt{h}$ to $u_{n+1}$ around $X_{n}^h$, we obtain
		\begin{align*}
			u_{n+1} &- u_{n} =   h^{1/2}Du_{n}(V_{n}) + \frac{h}{2}( D^{2}u_{n}(V_{n},V_{n})+\diff u_n(D_{V_n} V_n)) \nonumber \\ 
            &  + \frac{h^{3/2}}{6}(D^{3}u_{n}(V_{n},V_{n},V_{n})+3D^2u_n(V_n,D_{V_n}V_n) + \diff u_n(D^2_{V_n}V_n)  ) \\ 
            &+ \frac{h^{2}}{24}\rf{\dfrac{\diff^4}{\diff s^4}u( \gamma_{V_{n}}(\alpha))}, \numberthis \label{gula_eqn_5.20}
		\end{align*}
		where $\alpha \in (0,\sqrt{h})$. Taking expectation conditioned on $X_{n}^h$ and using the fact that $\gamma$ is a geodesic curve, we get
		\begin{align*}
			\mathbb{E}(u_{n+1}\;|\; X_{n}^h) & = u_{n}  - \frac{h}{2} g(\nabla \phi_{n}, \nabla u_{n}) + \frac{h}{2} \Delta_{M}u_{n}  \\ &  + \frac{h^{2}}{8}D^{2}u_{n}(\nabla \phi_{n}, \nabla \phi_{n}) -\frac{h^{3}}{48}D^{3}u_{n}(\partial_{i},\partial_{j}, \partial_{k})(\nabla \phi_{n})^{i}(\nabla \phi_{n})^{j} (\nabla \phi_{n})^{k} \\
			& - \frac{h^{2}}{2}D^{3}u_{n}(\partial_{i},\partial_{j},\partial_{k})(\nabla \phi_{n})^{i}g^{jk}_{n}    \\
			& + \frac{h^{2}}{24} \mathbb{E} [ D^{4}u( \gamma_{V_{n}}(\alpha))(\dot\gamma_{V_{n}}(\alpha), \dot\gamma_{V_{n}}(\alpha),\dot\gamma_{V_{n}}(\alpha),\dot\gamma_{V_{n}}(\alpha)) \;|\; X_{n}^h ],\numberthis  \label{eqn_5.16}
		\end{align*}
		where $\partial_{i}$, $i= 1,\dots, q$, denote basis vectors of $T_{X_{n}^h}M$ and $V_n \in T_{X_n^h}M$. The arguments employed to obtain (\ref{eqn_5.16}) are based on (\ref{eqntaylor5.11})-(\ref{eqntaylor5.13}) except that here we take expectation conditional on $X_{n}^h$ and $u$ represents the solution of the Poisson PDE (\ref{PPDE}). Rearranging (\ref{eqn_5.16}) and using the Poisson PDE (\ref{PPDE}), we arrive at
		\begin{align*}
			h\varphi(X_{n}^h) - h\mu_\phi(\varphi) &=  \mathbb{E}(u_{n+1}|X_{n}^h) - u_{n}  - \frac{h^{2}}{8}D^{2}u_{n}(\nabla \phi_{n}, \nabla \phi_{n}) \\
			&\quad +\frac{h^{3}}{48}D^{3}u_{n}(\partial_{i},\partial_{j}, \partial_{k})(\nabla \phi_{n})^{i}(\nabla \phi_{n})^{j} (\nabla \phi_{n})^{k} \\
			&\quad + \frac{h^{2}}{2}D^{3}u_{n}(\partial_{i},\partial_{j},\partial_{k})(\nabla \phi_{n})^{i}g^{jk}_{n}  \\
			&\quad   - \frac{h^{2}}{24} \mathbb{E} [ D^{4}u( \gamma_{V_{n}}(\alpha))(\dot\gamma_{V_{n}}(\alpha), \dot\gamma_{V_{n}}(\alpha),\dot\gamma_{V_{n}}(\alpha),\dot\gamma_{V_{n}}(\alpha)) \;|\; X_{n}^h ].\numberthis \label{eqn_5.17}
		\end{align*}
		Taking expectation on both sides, summing over $n = 0, \dots, N-1$, and applying (\ref{eq:Pmax}), we ascertain
		\begin{align*}
			\bigg| h\mathbb{E}\sum\limits_{n=0}^{N-1}\varphi(X_{n}^h) - hN\mu_\phi(\varphi)\bigg| \le  |\mathbb{E}(u_{N} - u_{0})|  + Ch^{2}N,
		\end{align*}
		where $C>0$ is independent of $T$ and $h$ and it linearly depends on $\Vert \varphi \Vert_{C^{2,\epsilon }(M)}$ but otherwise it is independent of $\varphi$.  Dividing by $T$ (note that $T =Nh$) gives us the result stated in (\ref{theorem4.2est1}):
		\begin{align*}
			|\mathbb{E} \tilde \mu_{\phi,N}(\varphi)- \mu_\phi(\varphi)| \leq C\Big(h + \frac{1}{T}\Big),
		\end{align*}
		where we have used the fact that $ |u_{N} - u_{0}|$ is bounded uniformly in $x$ due to compactness of the manifold $M$ and $u(x) \in C^{4,\epsilon}(M)$.
		
		To show the bound (\ref{theorem4.2est2}), we start with (\ref{gula_eqn_5.20}) 
		\begin{align*}
			u_{n+1} &- u_{n}=  h\mathcal{A}{u}_{n} +  h^{1/2}Du_{n}(g^{-1/2}_{n}\xi_{n+1})  + \frac{h}{2}\Big(D^{2}u_{n}(V_{n},V_{n}) -   \Delta_{M}u_{n}\Big) \\ & \;\;\; + \frac{h^{3/2}}{6}D^{3}u_{n}(V_{n},V_{n},V_{n})  + \frac{h^{2}}{24}D^{4}u( \gamma_{V_{n}}(\alpha))(\dot\gamma_{V_{n}}(\alpha), \dot\gamma_{V_{n}}(\alpha),\dot\gamma_{V_{n}}(\alpha),\dot\gamma_{V_{n}}(\alpha)). \numberthis  \label{gula_5.23}
		\end{align*}
		We  rearrange (\ref{gula_5.23}), sum over $n=0,\dots, N-1$ and use (\ref{PPDE}) to arrive at
		\begin{align*}
			Nh( \tilde \mu_{\phi,N}(\varphi)- \mu_\phi(\varphi)) &=  u_{N} - u_{0}\\
			& \quad + \sum\limits_{n=0}^{N-1}\bigg(-h^{1/2}Du_{n}(g^{-1/2}_{n}\xi_{n+1})  - \frac{h}{2}\Big(D^{2}u_{n}(V_{n},V_{n}) -   \Delta_{M}u_{n}\Big)\nonumber\\
			&\quad -\frac{h^{3/2}}{6}D^{3}u_{n}(V_{n},V_{n},V_{n}) \\
			&\quad - \frac{h^{2}}{24}D^{4}u( \gamma_{V_{n}}(\alpha))(\dot\gamma_{V_{n}}(\alpha), \dot\gamma_{V_{n}}(\alpha),\dot\gamma_{V_{n}}(\alpha),\dot\gamma_{V_{n}}(\alpha))\bigg). \numberthis 
		\end{align*}
		Squaring both sides and dividing by $T^2 $, we obtain
		\begin{align*}
			( \tilde \mu_{\phi,N}(\varphi)- &\mu_\phi(\varphi))^{2} \leq   \frac{C}{T^{2}}(u_{N} - u_{0})^{2} + \frac{Ch}{T^{2}}\bigg(\sum\limits_{n=0}^{N-1} Du_{n}(g^{-1/2}_{n}\xi_{n+1})\bigg)^{2} \\ & \;\;\; + \frac{CNh^{4}}{T^{2}}\sum\limits_{n=0}^{N-1}(D^2u_{n}(\nabla \phi_{n}, \nabla \phi_{n}))^{2}\\ & \;\;\; + \frac{Ch^{3}}{T^{2}}\bigg(\sum\limits_{n=0}^{N-1} D^2u_{n}(g^{-1/2}_{n}\xi_{n+1}, \nabla \phi_{n})\bigg)^{2} \\ & \;\;\;+ \frac{Ch^{3}}{T^{2}}\bigg(\sum\limits_{n=0}^{N-1} D^2 u_{n}( \nabla \phi_{n}, g^{-1/2}_{n}\xi_{n+1})\bigg)^{2} \\ & \;\;\; + \frac{Ch^{2}}{T^{2}}\bigg(\sum\limits_{n=0}^{N-1}D^{2}u_{n}(g^{-1/2}_{n}\xi_{n+1},g^{-1/2}_{n}\xi_{n+1}) -   \Delta_{M}u_{n}\bigg)^{2} \\
			&\quad + \frac{CNh^{6}}{T^{2}}\sum\limits_{n=0}^{N-1}(D^{3}u_{n}(\nabla \phi_{n},\nabla \phi_{n},\nabla \phi_{n}))^{2}   \\ 
			& \;\;\; + \frac{CNh^{5}}{T^{2}}\sum\limits_{n=0}^{N-1}(D^{3}u_{n}(\partial_{i},\partial_{j},\partial_{k})(\nabla \phi)^{i} (\nabla\phi)^{j} (g^{-1/2}_{n}\xi_{n+1})^{k})^{2} \\ & \;\;\;  + \frac{CNh^{4}}{T^{2}}\sum\limits_{n=0}^{N-1}(D^{3}u_{n}(\partial_{i},\partial_{j},\partial_{k})(\nabla \phi)^{i} (g^{-1/2}_{n}\xi_{n+1})^{j} (g^{-1/2}_{n}\xi_{n+1})^{k})^{2} \\& \;\;\;   + \frac{Ch^{3}}{T^{2}}\bigg(\sum\limits_{n=0}^{N-1}D^{3}u_{n}(\partial_{i},\partial_{j},\partial_{k})(g^{-1/2}_{n}\xi_{n+1})^{i} (g^{-1/2}_{n}\xi_{n+1})^{j} (g^{-1/2}_{n}\xi_{n+1})^{k}\bigg)^{2} \\ & \;\;\; + \frac{CNh^{4}}{T^{2}}\sum\limits_{n=0}^{N-1}\big(D^{4}u( \gamma_{V_{n}}(\alpha))(\dot\gamma_{V_{n}}(\alpha), \dot\gamma_{V_{n}}(\alpha),\dot\gamma_{V_{n}}(\alpha),\dot\gamma_{V_{n}}(\alpha))\big)^{2}. \numberthis \label{gula_eqn5.25}
		\end{align*}
		Using (\ref{eq:propxi}), we get 
		\begin{align*}
			\mathbb{E}\bigg(\sum\limits_{n=0}^{N-1}Du_{n}(g_{n}^{-1/2}&\xi_{n+1})\bigg)^{2} = \sum\limits_{n=0}^{N-1}\mathbb{E}(Du_{n}(g_{n}^{-1/2}\xi_{n+1}))^{2} \\ & \;\;\;+ 2 \sum\limits_{k<n}\mathbb{E}(Du_{k}(g_{k}^{-1/2}\xi_{k+1})Du_{n}(g_{n}^{-1/2}\xi_{n+1})) \\ & = \sum\limits_{n=0}^{N-1}\mathbb{E}(Du_{n}(g_{n}^{-1/2}\xi_{n+1}))^{2}, \numberthis  \label{gula_eqn5.26}
		\end{align*}
		and
		\begin{align*}
			\mathbb{E}\bigg(\sum\limits_{n=0}^{N-1} D^2u_{n}(g^{-1/2}_{n}\xi_{n+1}, \nabla \phi_{n})\bigg)^{2} =  \sum\limits_{n=0}^{N-1} \mathbb{E}(D^2 u_{n}(g^{-1/2}_{n}\xi_{n+1}, \nabla \phi_{n}))^{2}. \numberthis \label{gula_eqn5.27}
		\end{align*}
		Analogously,
		\begin{align*} 
			\mathbb{E}\bigg(\sum\limits_{n=0}^{N-1}D^{3}u_{n}&(\partial_{i},\partial_{j},\partial_{k})(g^{-1/2}_{n}\xi_{n+1})^{i} (g^{-1/2}_{n}\xi_{n+1})^{j} (g^{-1/2}_{n}\xi_{n+1})^{k}\bigg)^{2}\\ &= \sum\limits_{n=0}^{N-1}\mathbb{E}(D^{3}u_{n}(\partial_{i},\partial_{j},\partial_{k})(g^{-1/2}_{n}\xi_{n+1})^{i} (g^{-1/2}_{n}\xi_{n+1})^{j} (g^{-1/2}_{n}\xi_{n+1})^{k})^{2} \numberthis\label{gula_eqn5.28}
		\end{align*}
		Noting that $\mathbb{E} [D^{2}u_{n}(g^{-1/2}_{n}\xi_{n+1},g^{-1/2}_{n}\xi_{n+1}) -   \Delta_{M}u_{n} \;|\; X_{n}^h]=0$, we get
		\begin{align}
			\mathbb{E}\bigg(\sum\limits_{n=0}^{N-1}D^{2}u_{n}(g^{-1/2}_{n}&\xi_{n+1},g^{-1/2}_{n}\xi_{n+1}) -   \Delta_{M}u_{n}\bigg)^{2} \nonumber  \\ & = \sum\limits_{n=0}^{N-1}\mathbb{E}(D^{2}u_{n}(g^{-1/2}_{n}\xi_{n+1},g^{-1/2}_{n}\xi_{n+1}) -   \Delta_{M}u_{n})^{2}. \numberthis \label{gula_eqn5.29}
		\end{align}
		Taking expectation on both sides  of (\ref{gula_eqn5.25}), using (\ref{gula_eqn5.26})-(\ref{gula_eqn5.29}) and applying (\ref{eq:Pmax}), we obtain
		\begin{align*}
			\mathbb{E}( \tilde \mu_{\phi,N}(\varphi)- \mu_\phi(\varphi))^{2} \leq C\Big(h^{2} + \frac{1}{T}\Big),
		\end{align*}
		where $C>0$ is a constant independent of $h$ and $T$ and it linearly depends on $\Vert \varphi \Vert_{C^{2,\epsilon }(M)}$ but otherwise it is independent of $\varphi$.
$\blacksquare$
	
	\subsection{Proof of Theorem~\ref{theorem4.3}}\label{proof3}
	
	Consider  $\varphi \in \mathcal{H}$ (see (\ref{eq:distme})).  Since $\mu_{\phi}^{h}$ is stationary measure of Markov chain (see Section~\ref{section3}), we have  
	\begin{align*}
		\int_{M}\varphi(x) \mu_{\phi}^{h}(\diff x) = \int_{M}\mathbb{E}_{x}\varphi(X_{k}^{h}) \mu_{\phi}^{h}(\diff x),
	\end{align*}
	where $\mathbb{E}_{x} \varphi(X_{k}^{h})$ implies  that Markov chain starts at $x$, i.e. $X_0=x$.
	Therefore, we can write
	\begin{equation*}
		\int_{M} \varphi(x)\mu_{\phi}^{h}(\diff x) = \int_{M}\frac{1}{N}\sum\limits_{k=1}^{N}\mathbb{E}_{x}\varphi(X_{k}^{h})\mu_{\phi}^{h}(\diff x).
	\end{equation*}
	Using Theorem~\ref{theorem4.2}, we get
	\begin{align*}
		\bigg| \int_{M}\varphi(x)\big(\mu_{\phi}^{h}(\diff x) - 
		\mu(\diff x)\big)\bigg|  = \bigg|\int_{M}\bigg(\frac{1}{N}\sum\limits_{k=1}^{N}\mathbb{E}_{x}\varphi(X_{k}^{h}) - \mu_\phi(\varphi)\bigg)\mu_{\phi}^{h}(\diff x)\bigg| \leq K\bigg(h + \frac{1}{Nh}\bigg),
	\end{align*}
	where $K>0 $ is independent of $h$, $N$ and $\varphi$ from $\mathcal{H}$ (the latter thanks to the constant $C$ in Theorem~\ref{theorem4.2} being dependent on $\varphi$ only via linear dependence on $\Vert \varphi \Vert_{C^{2,\epsilon }(M)}$ which is bounded for the class $\mathcal{H}$ of functions $\varphi$). Letting $N \rightarrow \infty$, we get the desired result.
$\blacksquare$	

\bibliographystyle{plainnat}
\bibliography{references}

\begin{thebibliography}{84}
\providecommand{\natexlab}[1]{#1}
\providecommand{\url}[1]{\texttt{#1}}
\expandafter\ifx\csname urlstyle\endcsname\relax
  \providecommand{\doi}[1]{doi: #1}\else
  \providecommand{\doi}{doi: \begingroup \urlstyle{rm}\Url}\fi

\bibitem[Absil et~al.(2008)Absil, Mahony, and Sepulchre]{absil2008optimization}
P.-A. Absil, R.~Mahony, and R.~Sepulchre.
\newblock \emph{Optimization {A}lgorithms on {M}atrix {M}anifolds}.
\newblock Princeton University Press, 2008.

\bibitem[Alfonsi et~al.(2019)Alfonsi, Krief, and Tankov]{fin2}
Au. Alfonsi, D.~Krief, and P.~Tankov.
\newblock Long-time large deviations for the multiasset {W}ishart stochastic
  volatility model and option pricing.
\newblock \emph{SIAM J. Finan. Math.}, 10\penalty0 (4):\penalty0 942--976,
  2019.

\bibitem[Armstrong and King(2022)]{armstrong2022}
J.~Armstrong and T.~King.
\newblock Curved schemes for stochastic differential equations on, or near,
  manifolds.
\newblock \emph{Proceedings Royal Society A}, 478\penalty0 (2262):\penalty0
  20210785, 2022.

\bibitem[Aubin(1998)]{PDE1}
T.~Aubin.
\newblock \emph{Some {N}onlinear {P}roblems in {R}iemannian {G}eometry}.
\newblock Springer, Berlin, 1998.

\bibitem[Bakry(1986)]{bakry1986critere}
D.~Bakry.
\newblock Un {C}rit{\`e}re de non-explosion pour certaines diffusions sur une
  vari{\'e}t{\'e} {R}iemannienne compl{\`e}te.
\newblock \emph{Comptes rendus de l'{A}cad{\'e}mie des sciences. S{\'e}rie 1,
  Math{\'e}matique}, 303:\penalty0 23--26, 1986.

\bibitem[Bakry and {\'E}mery(1985)]{BE85}
D.~Bakry and M.~{\'E}mery.
\newblock Diffusions hypercontractives.
\newblock In J.~Az{\'e}ma and M.~Yor, editors, \emph{S{\'e}minaire de
  Probabilit{\'e}s XIX 1983/84}, pages 177--206. Springer, 1985.

\bibitem[Bally and Talay(1995)]{BT95}
V.~Bally and D.~Talay.
\newblock The {E}uler scheme for stochastic differential equations: error
  analysis with {M}alliavin calculus.
\newblock \emph{Math. Computers Simul.}, 38\penalty0 (1):\penalty0 35--41,
  1995.

\bibitem[Barp et~al.(2022)Barp, Costa, Franca, Friston, Girolami, Jordan, and
  Pavliotis]{Giro22}
A.~Barp, L.~Da Costa, G.~Franca, K.~Friston, M.~Girolami, M.I. Jordan, and G.A.
  Pavliotis.
\newblock Geometric methods for sampling, optimisation, inference and adaptive
  agents.
\newblock In \emph{Handbook of Statistics}, volume~46, pages 21--78. World
  Sci., 2022.

\bibitem[Bharath et~al.(2025)Bharath, Lewis, and Tretyakov]{man_msq}
K.~Bharath, A.~Lewis, and M.~V. Tretyakov.
\newblock \rf{Fundamental theorem for mean-square approximation of {SDE}s on
  manifolds}.
\newblock \emph{In preparation}, 2025.

\bibitem[Bishop and Crittenden(2011)]{bishop2011geometry}
R.~L. Bishop and R.~J. Crittenden.
\newblock \emph{Geometry of {M}anifolds}.
\newblock Academic Press, 2011.

\bibitem[Bou-Rabee and Hairer(2013)]{BH13}
N.~Bou-Rabee and M.~Hairer.
\newblock {Non-asymptotic mixing of the {MALA} algorithm}.
\newblock \emph{IMA J. Numer. Anal.}, 33\penalty0 (1):\penalty0 80--110, 2013.

\bibitem[Bru(1991)]{Bru91}
M.-F. Bru.
\newblock Wishart processes.
\newblock \emph{J. Theor. Probab.}, 4:\penalty0 725--751, 1991.

\bibitem[Cheeger and Ebin(1975)]{cheeger1975comparison}
J.~Cheeger and D.G. Ebin.
\newblock \emph{Comparison {T}heorems in Riemannian Geometry}.
\newblock North-Holland Publishing Company, 1975.

\bibitem[Cheng et~al.(2022)Cheng, Zhang, and Sra]{sra2022}
X.~Cheng, J.~Zhang, and S.~Sra.
\newblock Efficient sampling on {R}iemannian manifolds via {L}angevin {MCMC}.
\newblock \emph{Advances in Neural Information Processing Systems},
  35:\penalty0 5995--6006, 2022.

\bibitem[Dalalyan(2017)]{dalayan17}
A.~S. Dalalyan.
\newblock Theoretical guarantees for approximate sampling from smooth and
  log-concave densities.
\newblock \emph{J. Royal Stat. Society B}, 79\penalty0 (3):\penalty0 651--676,
  2017.

\bibitem[Da\;Prato and Zabczyk(1996)]{DaPrato96}
G.~Da\;Prato and J.~Zabczyk.
\newblock \emph{Ergodicity for Infinite Dimensional Systems}.
\newblock Cambr. Univ. Press, 1996.

\bibitem[Davidchack et~al.(2009)Davidchack, Handel, and Tretyakov]{Ruslan09}
R.~L. Davidchack, R.~Handel, and M.~V. Tretyakov.
\newblock Langevin thermostat for rigid body dynamics.
\newblock \emph{J. Chem. Phys.}, 130\penalty0 (23):\penalty0 234101, 2009.

\bibitem[Davidchack et~al.(2015)Davidchack, Ouldridge, and Tretyakov]{Ruslan15}
R.~L. Davidchack, T.~E. Ouldridge, and M.~V. Tretyakov.
\newblock New {L}angevin and gradient thermostats for rigid body dynamics.
\newblock \emph{J. Chem. Phys.}, 142\penalty0 (14):\penalty0 144114, 2015.

\bibitem[Downs(1972)]{downs1972orientation}
T.~D. Downs.
\newblock Orientation statistics.
\newblock \emph{Biometrika}, 59\penalty0 (3):\penalty0 665--676, 1972.

\bibitem[Durmus and Moulines(2017)]{DM17}
A.~Durmus and {\'E}.~Moulines.
\newblock {Nonasymptotic convergence analysis for the unadjusted Langevin
  algorithm}.
\newblock \emph{Ann. Appl. Probab.}, 27\penalty0 (3):\penalty0 1551 -- 1587,
  2017.

\bibitem[Elworthy(1982)]{David82}
K.~D. Elworthy.
\newblock \emph{Stochastic Differential Equations on Manifolds}.
\newblock London Mathematical Society Lecture Note Series. Cambridge University
  Press, 1982.

\bibitem[Everitt and Maclachlan(2000)]{everitt2000}
B.~Everitt and C.~Maclachlan.
\newblock Constructing hyperbolic manifolds.
\newblock \emph{Computational and Geometric Aspects of Modern Algebra. Michael
  Atkinson et al.(Editors), London Maths. Soc. Lect. Notes}, 275:\penalty0
  78--86, 2000.

\bibitem[Fisher et~al.(1993)Fisher, Lewis, and Embleton]{fisher1993statistical}
N.I. Fisher, T.~Lewis, and B.J.J. Embleton.
\newblock \emph{\rf{Statistical Analysis of Spherical Data}}.
\newblock Cambridge University Press, 1993.

\bibitem[Fonseca et~al.(2008)Fonseca, Grasselli, and Tebaldi]{fin1}
J.~Da Fonseca, M.~Grasselli, and C.~Tebaldi.
\newblock A multifactor volatility {H}eston model.
\newblock \emph{Quant. Finance}, 8\penalty0 (6):\penalty0 591--604, 2008.

\bibitem[Freidlin(1985)]{freidlin1985functional}
M.~I. Freidlin.
\newblock \emph{Functional Integration and Partial Differential Equations}.
\newblock Princeton Univ. Press, 1985.

\bibitem[Gallot et~al.(1990)Gallot, Hulin, and
  Lafontaine]{gallot1990riemannian}
S.~Gallot, D.~Hulin, and J.~Lafontaine.
\newblock \emph{Riemannian Geometry}, volume~2.
\newblock Springer, 1990.

\bibitem[Gatmiry and Vempala(2022)]{vempala2022convergence}
K.~Gatmiry and S.~S. Vempala.
\newblock Convergence of the {R}iemannian {L}angevin algorithm.
\newblock \emph{arXiv:2204.10818}, 2022.

\bibitem[Girolami and Calderhead(2011)]{girolami2011}
M.~Girolami and B.~Calderhead.
\newblock Riemann manifold {L}angevin and {H}amiltonian {M}onte {C}arlo
  methods.
\newblock \emph{J. Royal Stat. Society B}, 73\penalty0 (2):\penalty0 123--214,
  2011.

\bibitem[Greene and Wu(2006)]{greene2006function}
R.~E. Greene and H.-H. Wu.
\newblock \emph{Function theory on Manifolds Which Possess a Pole}.
\newblock Springer, 2006.

\bibitem[Grorud and Talay(1996)]{TAL96}
A.~Grorud and D.~Talay.
\newblock Approximation of {L}yapunov exponents of nonlinear stochastic
  differential equations.
\newblock \emph{SIAM J. Appl. Math.}, 56\penalty0 (2):\penalty0 627--650, 1996.

\bibitem[Hairer et~al.(1993)Hairer, N{\o}rsett, and Wanner]{HNW93}
E.~Hairer, S.~P. N{\o}rsett, and G.~Wanner.
\newblock \emph{Solving Ordinary Differential Equations. {I}. {N}onstiff
  Problems}.
\newblock Springer, Berlin, 1993.

\bibitem[Hairer et~al.(2002)Hairer, Lubich, and Wanner]{HLW02}
E.~Hairer, C.~Lubich, and G.~Wanner.
\newblock \emph{Geometric Numerical Integration: Structure Presesrving
  Algorithms for Ordinary Differential Equations}.
\newblock Springer, Berlin, 2002.

\bibitem[Hsu(2002)]{hsu2002stochastic}
E.~P. Hsu.
\newblock \emph{Stochastic Analysis on Manifolds}.
\newblock Amer. Math. Soc., 2002.

\bibitem[Hutzenthaler and Jentzen(2015)]{HJ15}
M.~Hutzenthaler and A.~Jentzen.
\newblock \emph{Numerical approximations of stochastic differential equations
  with non-globally {L}ipschitz continuous coefficients}, volume 236 of
  \emph{Mem. Amer. Math. Soc.}
\newblock AMS, Providence, 2015.

\bibitem[Hutzenthaler et~al.(2011)Hutzenthaler, Jentzen, and Kloeden]{HJK11}
M.~Hutzenthaler, A.~Jentzen, and P.~E. Kloeden.
\newblock Strong and weak divergence in finite time of {E}uler's method for
  stochastic differential equations with non-globally {L}ipschitz continuous
  coefficients.
\newblock \emph{Proc. R. Soc. A}, 467:\penalty0 1563--1576, 2011.

\bibitem[Ichihara(1982{\natexlab{a}})]{ichihara1982curvature1}
K.~Ichihara.
\newblock Curvature, geodesics and the {B}rownian motion on a {R}iemannian
  manifold i—recurrence properties.
\newblock \emph{Nagoya Math. J.}, 87:\penalty0 101--114, 1982{\natexlab{a}}.

\bibitem[Ichihara(1982{\natexlab{b}})]{ichihara1982curvature2}
K.~Ichihara.
\newblock Curvature, geodesics and the {B}rownian motion on a {R}iemannian
  manifold ii—explosion properties.
\newblock \emph{Nagoya Math. J.}, 87:\penalty0 115--125, 1982{\natexlab{b}}.

\bibitem[Ikeda and Watanabe(2014)]{ikeda2014stochastic}
N.~Ikeda and S.~Watanabe.
\newblock \emph{Stochastic Differential Equations and Diffusion Processes}.
\newblock Elsevier, 2014.

\bibitem[Jost(2008)]{jost2008riemannian}
J.~Jost.
\newblock \emph{Riemannian Geometry and Geometric Analysis}.
\newblock Universitext. Springer, 2008.

\bibitem[Jupp and Mardia(1979)]{jupp1979maximum}
P.~E. Jupp and K.~V. Mardia.
\newblock Maximum likelihood estimators for the matrix von {M}ises-{F}isher and
  {B}ingham distributions.
\newblock \emph{Ann. Stat.}, 7\penalty0 (3):\penalty0 599--606, 1979.

\bibitem[Kendall(1984)]{kendall1984brownian}
W.~S. Kendall.
\newblock Brownian motion on a surface of negative curvature.
\newblock \emph{S{\'e}minaire de probabilit{\'e}s de Strasbourg}, 18:\penalty0
  70--76, 1984.

\bibitem[Khasminskii(2012)]{khasminskii2011stochastic}
R.~Khasminskii.
\newblock \emph{Stochastic {S}tability of {D}ifferential {E}quations}.
\newblock Springer, 2012.

\bibitem[Kobayashi and Nomizu(1969)]{kobayashi1969foundations}
S.~Kobayashi and K.~Nomizu.
\newblock \emph{Foundations of Differential Geometry, Vol {II}}.
\newblock John Wiley \& Sons, 1969.

\bibitem[Kurz and Hanebeck(2015)]{kurz2015stochastic}
G.~Kurz and U.D. Hanebeck.
\newblock Stochastic sampling of the hyperspherical von {M}ises--{F}isher
  distribution without rejection methods.
\newblock In \emph{2015 Sensor Data Fusion: Trends, Solutions, Applications
  (SDF)}, pages 1--6. IEEE, 2015.

\bibitem[Ladyzhenskaya et~al.(1968)Ladyzhenskaya, Solonnikov, and
  Ural'tseva]{Lad68}
O.~A. Ladyzhenskaya, V.~A. Solonnikov, and N.~N. Ural'tseva.
\newblock \emph{Linear and quasi-linear equations of parabolic type}, volume~23
  of \emph{Trans. Math. Monog.}
\newblock Amer. Math. Soc., Providence, RI, 1968.

\bibitem[Laurent and Vilmart(2022)]{vilmart2022}
A.~Laurent and G.~Vilmart.
\newblock Order conditions for sampling the invariant measure of ergodic
  stochastic differential equations on manifolds.
\newblock \emph{Found. Comput. Math.}, 22:\penalty0 649–695, 2022.

\bibitem[Le et~al.(2024)Le, Lewis, Bharath, and Fallaize]{le2024diffusion}
H.~Le, A.~Lewis, K.~Bharath, and C.~Fallaize.
\newblock \rf{A diffusion approach to {S}tein’s method on {R}iemannian
  manifolds}.
\newblock \emph{Bernoulli}, 30\penalty0 (2):\penalty0 1079--1104, 2024.

\bibitem[Leimkuhler and Matthews(2015)]{leimkuhler_mathews_15}
B.~Leimkuhler and C.~Matthews.
\newblock \emph{Molecular Dynamics with Deterministic and Stochastic Numerical
  Methods}.
\newblock Springer, Berlin, 2015.

\bibitem[Leimkuhler et~al.(2023)Leimkuhler, Sharma, and
  Tretyakov]{leimkuhler2023simplerandom}
B.~Leimkuhler, A.~Sharma, and M.~V. Tretyakov.
\newblock Simplest random walk for approximating {R}obin boundary value
  problems and ergodic limits of reflected diffusions.
\newblock \emph{Ann. Appl. Probab.}, 33\penalty0 (3):\penalty0 1904 -- 1960,
  2023.

\bibitem[Leli\'{e}vre et~al.(2012)Leli\'{e}vre, Rousset, and
  Stoltz]{lelievre2012}
T.~Leli\'{e}vre, M.~Rousset, and G.~Stoltz.
\newblock Langevin dynamics with constraints and computation of free energy
  differences.
\newblock \emph{Math. Comp.}, 81\penalty0 (280):\penalty0 2071--2125, 2012.

\bibitem[{L}ewis(2023)]{lewis2023contributions}
{A}. {L}ewis.
\newblock \emph{Contributions to {S}tein's Method on {R}iemannian Manifolds}.
\newblock PhD thesis, University of Nottingham, 2023.

\bibitem[Li and Erdogdu(2023)]{LE2023}
M.~B. Li and M.~A. Erdogdu.
\newblock Riemannian {L}angevin algorithm for solving semidefinite programs.
\newblock \emph{Bernoulli}, 29:\penalty0 3093 -- 3113, 2023.

\bibitem[Li(1992)]{li2021stochastic}
X.-M. Li.
\newblock \emph{Stochastic Flows on Non-compact Manifolds}.
\newblock PhD thesis, University of Warwick, 1992.

\bibitem[Lunardi(1995)]{lunbook}
A.~Lunardi.
\newblock \emph{Analytic Semigroups and Optimal Regularity in Parabolic
  Problems}.
\newblock Springer, 1995.

\bibitem[Malham and Wiese(2008)]{MW08}
S.~J.~A. Malham and A.~Wiese.
\newblock Stochastic {L}ie group integrators.
\newblock \emph{SIAM J. Sci. Comp.}, 30\penalty0 (2):\penalty0 597--617, 2008.

\bibitem[Mangoubi and Smith(2018)]{mangoubi2018rapid}
O.~Mangoubi and A.~Smith.
\newblock Rapid mixing of geodesic walks on manifolds with positive curvature.
\newblock \emph{Ann. Appl. Probab.}, 28\penalty0 (4):\penalty0 2501--2543,
  2018.

\bibitem[Mattingly et~al.(2002)Mattingly, Stuart, and Higham]{MSH01}
J.~C. Mattingly, A.~M. Stuart, and D.~J. Higham.
\newblock Ergodicity for {S}{D}{E}s and approximations: locally {L}ipschitz
  vector fields and degenerate noise.
\newblock \emph{Stoch. Proc. Appl.}, 101:\penalty0 185--232, 2002.

\bibitem[Mattingly et~al.(2010)Mattingly, Stuart, and Tretyakov]{MST10}
J.~C. Mattingly, A.~Stuart, and M.~V. Tretyakov.
\newblock Convergence of numerical time-averaging and stationary measures via
  {P}oisson equations.
\newblock \emph{SIAM J. Numer. Anal.}, 48\penalty0 (2):\penalty0 552--577,
  2010.

\bibitem[Mentink et~al.(2010)Mentink, Tretyakov, Fasolino, Katsnelson, and
  Rasing]{Johan10}
J.~H. Mentink, M.~V. Tretyakov, A.~Fasolino, M.~I. Katsnelson, and Th. Rasing.
\newblock Stable and fast semi-implicit integration of the stochastic
  {L}andau-{L}ifshitz equation.
\newblock \emph{J. Phys.: Condens. Matter}, 22\penalty0 (17):\penalty0 176001,
  2010.

\bibitem[Mijatovi{\'c} et~al.(2020)Mijatovi{\'c}, Mramor, and
  Bravo]{mijatovic2020note}
A.~Mijatovi{\'c}, V.~Mramor, and G.~U. Bravo.
\newblock \rf{A note on the exact simulation of spherical {B}rownian motion}.
\newblock \emph{Stat. \& Prob. Lett.}, 165:\penalty0 108836, 2020.

\bibitem[Milstein(1978)]{GN78a}
G.~N. Milstein.
\newblock \rf{A method with second order accuracy for the integration of
  stochastic differential equations}.
\newblock \emph{Theor. Prob. Appl.}, 23:\penalty0 414--419, 1978.

\bibitem[Milstein(1985)]{GN85}
G.~N. Milstein.
\newblock \rf{Weak approximation of solutions of systems of stochastic
  differential equations}.
\newblock \emph{Theor. Prob. Appl.}, 30:\penalty0 706--721, 1985.

\bibitem[Milstein and Tretyakov(1999)]{MT99}
G.~N. Milstein and M.~V. Tretyakov.
\newblock \rf{Simulation of a space-time bounded diffusion}.
\newblock \emph{Ann. Appl. Probab.}, 9\penalty0 (3):\penalty0 732--779, 1999.

\bibitem[Milstein and Tretyakov(2005)]{MT05}
G.~N. Milstein and M.~V. Tretyakov.
\newblock Numerical integration of stochastic differential equations with
  nonglobally {L}ipschitz coefficients.
\newblock \emph{SIAM J. Numer. Anal.}, 43\penalty0 (3):\penalty0 1139--1154,
  2005.

\bibitem[Milstein and Tretyakov(2007)]{MT07}
G.~N. Milstein and M.~V. Tretyakov.
\newblock Computing ergodic limits for {L}angevin equations.
\newblock \emph{Physica D}, 229\penalty0 (1):\penalty0 81--95, 2007.

\bibitem[Milstein and Tretyakov(2021)]{milstein2004stochastic}
G.~N. Milstein and M.~V. Tretyakov.
\newblock \emph{Stochastic Numerics for Mathematical Physics}.
\newblock Springer, 2nd edition, 2021.

\bibitem[Miranda(1969)]{Miranda}
C.~Miranda.
\newblock \emph{Partial Differential Equations of Elliptic Type}.
\newblock Springer, 1969.

\bibitem[Moakher and Z{\'e}ra{\"\i}(2011)]{moakher2011riemannian}
M.~Moakher and M.~Z{\'e}ra{\"\i}.
\newblock The {R}iemannian geometry of the space of positive-definite matrices
  and its application to the regularization of positive-definite matrix-valued
  data.
\newblock \emph{J. Mathem. Imaging and Vision}, 40:\penalty0 171--187, 2011.

\bibitem[Nicolaescu(2007)]{PDE2}
L.~I. Nicolaescu.
\newblock \emph{Lectures on the Geometry of Manifolds}.
\newblock World Scientific, 2007.

\bibitem[Pennec et~al.(2006)Pennec, Fillard, and Ayache]{pennec2006riemannian}
X.~Pennec, P.~Fillard, and N.~Ayache.
\newblock A {R}iemannian framework for tensor computing.
\newblock \emph{Inter. J. Computer Vision}, 66:\penalty0 41--66, 2006.

\bibitem[Roberts and Tweedie(1996)]{RT96}
G.~O. Roberts and R.~L. Tweedie.
\newblock Exponential convergence of {L}angevin distributions and their
  discrete approximations.
\newblock \emph{Bernoulli}, 2:\penalty0 341--363, 1996.

\bibitem[Said et~al.(2017)Said, Bombrun, Berthoumieu, and
  Manton]{said2017riemannian}
S.~Said, L.~Bombrun, Y.~Berthoumieu, and J.~H. Manton.
\newblock Riemannian {G}aussian distributions on the space of symmetric
  positive definite matrices.
\newblock \emph{IEEE Transactions on Information Theory}, 63:\penalty0
  2153--2170, 2017.

\bibitem[Schwarz et~al.(2023)Schwarz, Herrmann, Sturm, and
  Wardetzky]{schwarz2023efficient}
S.~Schwarz, M.~Herrmann, A.~Sturm, and M.~Wardetzky.
\newblock Efficient random walks on {R}iemannian manifolds.
\newblock \emph{Found. Comput. Mathem.}, pages 1--17, 2023.

\bibitem[Sharma and Zhang(2021)]{sharma2021}
U.~Sharma and W.~Zhang.
\newblock Nonreversible sampling schemes on submanifolds.
\newblock \emph{SIAM J. Numer. Anal.}, 59\penalty0 (6):\penalty0 2989--3031,
  2021.

\bibitem[Talay(1986)]{TAL83a}
D.~Talay.
\newblock Discr\'etisation d'une \'equation diff\'erentielle stochastique et
  calcul approch\'e d'esp\'erances de fonctionnelles de la solution.
\newblock \emph{Math. Model. Numer. Anal. (ESAIM)}, 20\penalty0 (1):\penalty0
  141--179, 1986.

\bibitem[Talay(1990)]{talayinv}
D.~Talay.
\newblock Second-order discretization schemes of stochastic differential
  systems for the computation of the invariant law.
\newblock \emph{Stoch. Stoch. Reports}, 29:\penalty0 13--36, 1990.

\bibitem[Talay and Tubaro(1990)]{TAT90}
D.~Talay and L.~Tubaro.
\newblock Expansion of the global error for numerical schemes solving
  stochastic differential equations.
\newblock \emph{Stoch. Anal. Appl.}, 8:\penalty0 483--509, 1990.

\bibitem[Teh et~al.(2016)Teh, Thiery, and Vollmer]{JMLR:v17:teh16a}
Y.~W. Teh, A.~H. Thiery, and S.~J. Vollmer.
\newblock Consistency and fluctuations for stochastic gradient {L}angevin
  dynamics.
\newblock \emph{J. Machine Learning Research}, 17\penalty0 (7):\penalty0 1--33,
  2016.

\bibitem[Tretyakov(2025)]{Msisc25}
M.~V. Tretyakov.
\newblock \rf{Sampling from mixture distributions based on regime-switching
  diffusions}.
\newblock \emph{SIAM J. Sci. Comp.}, to appear, 2025.

\bibitem[Tretyakov and Zhang(2013)]{TZ13}
M.~V. Tretyakov and Z.~Zhang.
\newblock A fundamental mean-square convergence theorem for {SDE}s with locally
  {L}ipschitz coefficients and its applications.
\newblock \emph{SIAM J. Numer. Anal.}, 51:\penalty0 3135--3162, 2013.

\bibitem[Wang(2009)]{wang2009log}
F.-Y. Wang.
\newblock Log-{S}obolev inequalities: different roles of {R}ic and {H}ess.
\newblock \emph{Ann. Probab.}, 37\penalty0 (4):\penalty0 1587--1604, 2009.

\bibitem[Wang(2013)]{Wang2013}
F.-Y. Wang.
\newblock \emph{Analysis for Diffusion Processes on Riemannian Manifolds}.
\newblock World Scientific, 2013.

\bibitem[Watson(1965)]{watson1965equatorial}
G.~S. Watson.
\newblock Equatorial distributions on a sphere.
\newblock \emph{Biometrika}, 52\penalty0 (1/2):\penalty0 193--201, 1965.

\bibitem[Wood(1994)]{wood1994simulation}
A.~T.~A. Wood.
\newblock Simulation of the von {M}ises {F}isher distribution.
\newblock \emph{Commun. Statistics-Simulation Comput.}, 23\penalty0
  (1):\penalty0 157--164, 1994.

\end{thebibliography}

\end{document}